\documentclass[11pt,reqno]{amsart}
\usepackage{amsthm}
\usepackage{amssymb}
\usepackage{latexsym}
\usepackage{multicol}
\usepackage{verbatim,enumerate}
\usepackage[usenames]{color}
\usepackage[colorlinks=true, linkcolor=blue, citecolor=blue, urlcolor=blue]{hyperref}
\usepackage{hyperref}
\usepackage{hyperref}
\usepackage{amsmath, amscd}
\usepackage{tikz}
\usetikzlibrary{matrix,arrows,decorations.pathmorphing}
\usetikzlibrary{positioning}

\advance\textwidth by 1.2in \advance\oddsidemargin by -.6in \advance\evensidemargin by -.6in

\theoremstyle{plain}
\newtheorem{prop}{Proposition}
\newtheorem{thm}{Theorem}

\newtheorem{lem}{Lemma}

\newtheorem{cor}{Corollary}%[section]

\theoremstyle{definition}
\newtheorem{example}{Example}
\newtheorem{defn}{Definition}
\newtheorem*{theom}{Theorem}

\theoremstyle{remark}
\newtheorem{rem}{Remark}

\newcommand{\lie}[1]{\mathfrak{#1}}

\newcommand\bc{\mathbb C}
\newcommand\bn{\mathbb N}
\newcommand\bz{\mathbb Z}

   \newcommand\bod{\bold d}

\newcommand{\st}{\operatorname{st}}

\def\a{\alpha}

\DeclareMathOperator{\htt}{ht}
\DeclareMathOperator{\mult}{mult}
\DeclareMathOperator{\supp}{supp}
\DeclareMathOperator{\wt}{wt}
\newenvironment{pf}{\proof}{\endproof}
\newcounter{cnt}
 \makeatletter
\def\mydggeometry{\makeatletter\dg@YGRID=1\dg@XGRID=20\unitlength=0.003pt\makeatother}
\makeatother \theoremstyle{remark}

\numberwithin{equation}{section}
\makeatletter
\def\section{\def\@secnumfont{\mdseries}\@startsection{section}{1}%
  \z@{.7\linespacing\@plus\linespacing}{.5\linespacing}%
  {\normalfont\scshape\centering}}
\def\subsection{\def\@secnumfont{\bfseries}\@startsection{subsection}{2}%
  {\parindent}{.5\linespacing\@plus.7\linespacing}{-.5em}%
  {\normalfont\bfseries}}
\makeatother

\begin{document}

%%%%%%%%%%%%%%%%%%%%%%%%%%%%%%%%%%%%%%%%%%%%

\title{A study on free roots of Borcherds-Kac-Moody Lie superalgebras}

%Some combinatorial results on Root multiplicities of Borcherds-Kac-Moody superalgebras

\author{Shushma Rani}
\address{Indian Institute of Science Education and Research, Mohali, India}
\email{shushmarani95@gmail.com, ph16067@iisermohali.ac.in.}

\author{G. Arunkumar}
%\thanks{Authors thank Tanusree Khandai for many helpful discussions on Lie superalgebras and also for her constant support}
\address{Indian Institute of Science, Bangalore, India}
\email{arun.maths123@gmail.com, garunkumar@iisc.ac.in.}

\thanks{$^{1}$-These roots spaces are free from Serre relations. Also, these roots spaces can be identified with certain grade spaces of free partially commutative Lie superalgebras [c.f. Lemma \ref{identification lem}]. So we call them free roots and the associated root spaces free root spaces of $\lie g$.}


\maketitle

\begin{abstract}
Let $\lie g$ be a Borcherds-Kac-Moody Lie superalgebra (BKM superalgebra in short) with the associated graph $G$. 
Any such $\lie g$ is constructed from a free Lie superalgebra by introducing three different sets of relations on the generators:
	(1) Chevalley relations,
	(2) Serre relations, and
	(3) Commutation relations coming from the graph $G$.
By Chevalley relations we get a triangular decomposition $\lie g = \lie n_+ \oplus \lie h \oplus \lie n_{-}$ and each roots space $\lie g_{\alpha}$ is either contained in $\lie n_+$ or $\lie n_{-}$. In particular, each $\lie g_{\alpha}$ involves only the relations (2) and (3). In this paper, we are interested in the root spaces of $\lie g$ which are independent of the Serre relations. We call these roots free roots$^{1}$ of $\lie g$. Since these root spaces involve only commutation relations coming from the graph $G$ we can study them combinatorially. We use heaps of pieces to study these roots and prove many combinatorial properties. We construct two different bases for these root spaces of $\lie g$: One by extending the Lalonde's Lyndon heap basis of free partially commutative Lie algebras to the case of free partially commutative Lie superalgebras and the other by extending the basis given in \cite{akv17} for the free root spaces of Borcherds algebras to the case of BKM superalgebras. This is done by studying the combinatorial properties of super Lyndon heaps. We also discuss a few other combinatorial properties of free roots.
\end{abstract}
%\maketitle
\tableofcontents
%\clearpage

\section{Introduction}
%[c.f. Section \ref{sec 2.1}]
In this paper, we are interested in the combinatorial properties of roots of a Borcherds-Kac-Moody Lie superalgebra (BKM superalgebra in short) $\lie g$. In particular, we are interested in the roots of $\lie g$ whose associated root spaces are independent of the Serre relations. We call these roots free roots of $\lie g$. BKM superalgebras  \cite{UR95,UR06,jp12,nk20} are a natural generalization of two important classes of Lie algebras namely Borcherds algebras (Generalized Kac-Moody algebras) \cite{Bor88,ej95,ej98, akv17} and the Kac-Moody Lie superalgebras \cite{kac77,cjw15,JD00,mr20,afm21}. BKM superalgebras have a wide range of applications in mathematical physics \cite{p12, hbl12, gh11, ghp11,cp19}. For example, physicists applied these algebras to describe supersymmetry, chiral supergravity, and  Gauge theory \cite{GS19,HPV19,k15,glry13,ca09,hl15}. 
%In \cite{GS19}, Govindarajan and Samanta observed that BKM superalgebras appeared as generating function of quarter-BPS states (dyons) in  supersymmetric four-dimensional string theory. Due to its wide use in mathematical physics, not only mathematicians but physicists are also taking interest in this BKM superalgebras. In \cite{HPV19}, Harrison, Paquette, and Volpato  showed that the BKM algebra is generated by the physical states (BRST cohomology classes) in a chiral superstring theory and they used this construction to produce denominator identities for the chiral partition functions of the Conway module. They showed that these superalgebras act on BPS states in string theory. 
%This class of superalgebras is also important in Mathematics because they are natural extension of two important and well-studied class of Lie algebras namely Borcherds algebras and Kac-Moody superalgebras [?,?,?]. %In a different direction hyperbolic reflection groups are Weyl groups of BKM Lie superalgebras with the non-trivial odd part. 
%Borcherds derived a Borcherds Weyl Character formula which as a consequence gives denominator identity. 

We explain our results in detail: In \cite[Theorem 1]{akv17} the following connection between the root multiplicities of a Borcherds algebra $\lie g$ and the $\bold k$-chromatic polynomial $\pi_{\bold k}^G(q)$ [c.f. Definition \ref{chmpoly}] of the associated quasi Dynkin diagram $G$ [c.f. Definition \ref{qdd}] is proved. 

\begin{thm}\label{mainthmchb}
	Let $\lie g$ be a Borcherds algebra with associated Borcherds-Cartan matrix $A$ and the quasi Dynkin diagram $G$. Assume that the matrix $A$ is indexed by the countable (finite/countably infinite) set $I$. Let $\bold k=(k_i: i \in I) \in \mathbb Z_+[I]:= \oplus_{i \in I}\mathbb Z \alpha_i$ be such that $k_i \le 1$ for $i \in I^{re}$. Then
	$$
	\pi^{G}_{\mathbf{k}}(q)=(-1)^{\mathrm{ht}(\eta(\bold k))} \sum_{\mathbf{J}\in L_{G}(\bold k)} (-1)^{|\mathbf{J}|}\prod_{J\in\bold J}\binom{q\text{ mult}(\beta(J))}{D(J,\mathbf{J})}.
	$$
	where $\eta(\bold k) := \sum_{i \in I}k_i \alpha_i$, $\htt (\eta(\bold k)) := \sum_{i \in I}k_i$, and $L_G(\bold k)$ is the bond lattice of weight $\bold k$ of the graph $G$ [c.f. Definition \ref{bond}].
\end{thm}
%\begin{rem}
%	When $\lie g$ is a Kac-Moody algebra,  the root space $\lie g_{\eta(\bold k)}$  considered in Theorem \ref{mainthmchb} has dimension equal to the number of freeLyndon heaps of weight $\bold k$ [c.f. Lemma \ref{identification lem} and Theorem \ref{lafp}]. For this reason, we call any root of a BKM superalgebra which satisfies the hypothesis of Theorem \ref{mainthmchb} a free root.
%\end{rem}

Further in \cite{akv17}, \textbf{using Theorem \ref{mainthmchb}}, the following theorem is proved in which a set of the basis for the root space $\lie g_{\eta(\bold k)}$   is constructed using the combinatorial model $C^{i}(\bold k,G)$ [c.f. Equation \eqref{cikg}]. 

\begin{thm}\label{mainthmbb}
	Let $\bold k \in \mathbb Z_+[I]$ be as in Theorem \ref{mainthmchb}. Then the set $\left\{\iota(\bold w): \bold w\in C^{i}(\bold k, G)\right\}$ is a basis for the root space $\lie g_{\eta(\bold k)}$. Moreover, if $k_i=1$, the set
	$\left\{e(\bold w): \bold w\in \mathcal{X}_i,\ \rm{wt}(\bold w)=\eta(\bold k)\right\}$
	forms a left-normed basis of $\lie g_{\eta(\bold k)}$ [c.f. Section \ref{section 5}]. 
	%$$C^{i}(\bold k, G)=\{\bold w\in \mathcal{X}_i:  e(\bold w)\neq 0,\  \rm{wt}(\bold w)=\eta(\bold k)\}.$$  
	We call these bases (one basis for each $i \in I$) as LLN bases (Lyndon - Left normed bases.) [c.f. Example \ref{LLLN}]. 
\end{thm}

%The notions and the symbols used in the above theorem are explained in Section \ref{section 5}. 
We observe that, by Lemma \ref{identification lem}, the root spaces considered in Theorems \ref{mainthmchb} and \ref{mainthmbb} are precisely the set of free roots of the Borcherds algebra $\lie g$. In particular, the proof shows that the associated root spaces are independent of the Serre relations. In this paper, we construct two different bases for the free root spaces of BKM superalgebras. First, we prove Theorem  \ref{mainthmbb} for the case BKM superalgebras. In particular, This will give us the LLN basis for the free root spaces of BKM superalgebras. Our proof is completely different from \cite{akv17} even for Borcherds algebras. For example, our proof of Theorem \ref{mainthmbb} is independent of Theorem \ref{mainthmchb} (for Borcherds algebras) and its super analog Theorem \ref{mainthmch} (for BKM superalgebras). We give direct simple proof. Second, we construct a Lyndon heaps basis for the free root spaces of BKM superalgebras. We remark that this basis is not discussed in \cite{akv17}.

This is done in the following steps.
  Let $\lie g(A,\Psi)$ be a BKM superalgebra with the associated quasi Dynkin diagram $(G,\Psi)$. Assume that $\bold k \in \mathbb Z_+[I]$ satisfies $k_i \le 1$ for $i \in I^{re}\sqcup \Psi_0$.

	(1) First, we introduce the notion of a supergraph $(G,\Psi)$ [c.f. Definition \ref{sg}]. Using this definition, we give the definition of free partially commutative Lie superalgebra $\mathcal{LS}(G,\Psi)$ associated with a supergraph $(G,\Psi)$ (when $\Psi$ is the empty set $\mathcal{LS}(G)$ is the free partially commutative Lie algebra $\mathcal L(G)$ associated with the graph $G$). Then we explain the Lyndon heaps basis of $\mathcal L(G)$ due to Lalonde \cite{la93}: The set $\{\Lambda(\bold E): \bold E \in \mathcal H(I,\zeta) \text{ is a Lyndon heap}\}$ forms a basis of $\mathcal{L}(G)$ [c.f. Theorem \ref{lafp}] (Basics definitions and results in the theory of heaps of pieces are given in Section \ref{pyramidsec}).

 Next, we identify a free root space $\lie g_{\eta(\bold k)}$ of the BKM superalgebra $\lie g(A,\Psi)$ with the $\bold k$-grade space $\mathcal{LS}_{\bold k}(G,\Psi)$ of the free partially commutative Lie superalgebra $\mathcal{LS}(G,\Psi)$ associated with the supergraph $(G,\Psi)$.  The exact statement is as follows:	
 %Fix $\bold k \in \mathbb Z_{\ge 0}^I$ such that $|\supp \bold k|<\infty$ and $k_i \le 1$ for $i \in I^{re}\sqcup \Psi_0$. 
  The root space $\lie g_{\eta(\bold k)}$ can be identified with the grade space $\mathcal {LS}_{\bold k}(G)$ of the free partially commutative Lie superalgebra $\mathcal{LS}(G,\Psi)$. In particular, $\mult \eta(\bold k) = \dim \mathcal {LS}_{\bold k}(G)$.   \textbf{This is our first main result.}
  This observation plays a crucial role in giving an alternate proof of Theorem \ref{mainthmbb} when $\lie g$ is a Borcherds algebra [c.f. Section \ref{section 5}].
 %By using this identification we prove that, when $\lie g$ is a Borcherds algebra, the basis of the root space $\lie g_{\eta(\bold k)}$ given in Theorem \ref{mainthmbb} is same as the Lalonde's Lyndon heaps basis of $\mathcal L_{\bold k}(G)$. \textbf{This is our first main result}. This is done by identifying the elements of $C^{i}(\bold k,G)$ as Lyndon heaps over the graph $G$.  The main tool is Proposition \ref{liffl} which illustrates the relationship between the combinatorial model $C^{i}(\bold k,G)$ and the Lyndon heaps over the graph $G$.

%Let $L \in \mathcal A_i^*(I,\zeta)$ then $L$ is a Lyndon word in $\mathcal A_i^*(I,\zeta)$ if and only if $L$ is a Lyndon heap as an element of $\mathcal H(I,\zeta)$.

		(2) Step (1) shows that to construct a basis for a free root space $\lie g_{\eta(\bold k)}$ of the BKM superalgebra $\lie g(A,\Psi)$ it is enough to extend the Lyndon heaps basis of the free partially commutative Lie algebra $\mathcal L(G)$ to the case of free partially commutative Lie superalgebra $\mathcal{LS}(G,\Psi)$. We introduce the notion of super Lyndon heaps and construct a (super) Lyndon heaps basis for $\mathcal{LS}(G,\Psi)$ following the proof idea of Lalonde. The precise statement is as follows [c.f. Theorem \ref{lafps}]:
			%	\begin{prop}\label{lafps}
			The set $\{\Lambda(\bold E): \bold E \in \mathcal H(I,\zeta) \text{ is super Lyndon heap}\}$ forms a basis of $\mathcal{LS}(G,\Psi)$.
		%	\hfill\qed
	%	\end{prop}
	\textbf{This is our second main result}. 
	%We accomplish this by introducing the notion of super Lyndon heaps.   
	This gives us Lyndon heaps basis for the free root spaces of BKM superalgebras. 
	
	%\item The $\bold k$-grade space of the free partially commutative Lie algebra $\mathcal L(G)$ associated with the graph $G$ has a well-known basis indexed by Lyndon heaps due to Lalonde \cite{la93}. We explain this result.
	(3) Next, we extend Theorem \ref{mainthmbb} to the case of BKM superalgebras, namely we construct LLN basis for the free root spaces of BKM superalgebras. The main step is the identification of the combinatorial model $C^{i}(\bold k,G)$ given in Theorem \ref{mainthmbb} with the (super) Lyndon heaps of weight $\bold k$ over the graph $G$. This gives a different proof (heap theoretic proof) to Theorem \ref{mainthmbb} when $\lie g$ is a Borcherds algebra.  We are using neither the denominator identity nor Theorem \ref{mainthmchb} in our proof. We use only the identification of the spaces $\lie g_{\eta(\bold k)}$ and $\mathcal{LS}_{\bold k}(G)$, and the super Lyndon heaps basis for free partially commutative Lie superalgebra $\mathcal{LS}(G,\Psi)$ [c.f. Theorem \ref{lafps}]. In this sense, our proof is simpler and transparent.  \textbf{This is our third main result}. We remark that the LLN basis and the Lyndon heaps basis of a free roots space $\lie g_{\eta(\bold k)}$ are different in general. The cases when the elements of these two bases are the same are discussed in Section \ref{same}.
	
		(4) Next, along with various other combinatorial results on the free roots, we prove the following super analogue of Theorem \ref{mainthmchb} and its corollary. \textbf{This is our final main result}. 
		%Again, we don't need denominator identity of BKM superalgebras to prove this result. We will use the much simpler analogues identity of free partially commutative Lie superalgebra $\mathcal{LS}(G,\Psi)$ in view of Lemma \ref{identification lem}. 
		\begin{thm}\label{mainthmch}
			Let $G$ be the quasi Dynkin diagram of a BKM superalgebra $\lie g$. Assume that $\bold k=(k_i:i\in I) \in \mathbb Z_+[I]$ satisfies the assumptions of Theorem \ref{mainthmchb} and in addition $k_i \le 1$ for $i \in \Psi_0$. Then
			$$
			\pi^{G}_{\mathbf{k}}(q)=(-1)^{\mathrm{ht}(\eta(\bold k))} \sum_{\mathbf{J}\in L_{G}(\bold k)} (-1)^{|\mathbf{J}| + |\bold J_1|}\prod_{J\in\bold J_0}\binom{q\text{ mult}(\beta(J))}{D(J,\mathbf{J})}\prod_{J\in\bold J_1}\binom{-q\text{ mult}(\beta(J))}{D(J,\mathbf{J})}.
			$$
			where $L_G(\bold k)$ is the bond lattice of weight $\bold k$ of the graph $G$.
		\end{thm}
			%An immediate application of the above theorem gives the following combinatorial formula for the root multiplicities.
			We have the following corollary to the above theorem which gives us a recurrence formula for the root multiplicities of free roots of $\lie g$.
			\begin{cor}\label{recursionmult}
			We have
			\begin{equation*}\label{emult}
			\mult(\eta(\bold k)) = \sum\limits_{\ell | \bold k}\frac{\mu(\ell)}{\ell}\ |\pi^{G}_{\bold k/\ell}(q)[q]|\end{equation*} 
			if $\eta(\bold k) = \sum_{i \in I}k_i\alpha_i \in \Delta_0^+$ and 
			\begin{equation*}\label{omult}
			\mult(\eta(\bold k)) = \sum\limits_{\ell | \bold k}\frac{(-1)^{l+1} \mu(\ell)}{\ell}\ |\pi^{G}_{\bold k/\ell}(q)[q]|\end{equation*}
			if $\eta(\bold k) \in \Delta_1^+$ where $|\pi^{G}_{\bold k}(q)[q]|$ denotes the absolute value of the coefficient of $q$ in $\pi^{G}_{\bold k}(q)$ and $\mu$ is the M\"{o}bius function. See Example \ref{multf0} for an working example of this formula.
			
			If $k_i$'s are relatively prime (In particular if $k_i =1$ for some $i \in I$) then the above formula becomes much simpler: \begin{equation*}
			\mult(\eta(\bold k)) =  |\pi^{G}_{\bold k}(q)[q]|\end{equation*} for any $\eta(\bold k) \in \Delta^+$.
		\end{cor}
			%The proofs of Theorem \ref{mainthmch} and its corollary are given in Section \ref{mainsch}.
			We also discuss why the expression given in Theorem \ref{mainthmch} exists only for free roots explaining the main assumptions made in \cite{VV15} and \cite{akv17}. Various examples explaining our results are provided throughout the paper.
			
The paper is organized as follows. In Section \ref{section2}, the definition and the basic results on the Borcherds Kac Moody Lie superalgebra are given. In Section \ref{identificationsec}, we construct the Lyndon heaps basis of free root spaces of BKM superalgebras. In Section \ref{section 5}, we construct the LLN basis of free root spaces of BKM superalgebras. In Section \ref{multis}, we study the further combinatorial properties of free roots of BKM superalgebras. %In define the generalized Chromatic polynomial and give a combinatorial formula for the root multiplicities of free roots of BKM superalgebras. %In section 4, we define combinatorial tool Lyndon words and Lyndon heaps. In section 5, we state the main theorem \ref{mainthms}. The following section 6,7,8 gives proof of our main theorem. In section 9, we will give proof of theorem \ref{lafps}, which is an important step to prove our main theorem.

{\em Acknowledgments. The authors thank Tanusree Khandai for many helpful discussions on Lie superalgebras and also for her constant support. The first author acknowledges the CSIR research grant: 09/947(0082)/2017-EMR-I. The second author thanks Apoorva Khare for his constant support and also acknowledges the National Board for Higher Mathematics postdoctoral research grant: 0204/7/2019/R$\&$D-II/6831.} %She also thanks Department of Mathematics, Indian Institute of Science Education and Research, Mohali for providing the wonderful working environment.}

%%%%%%%%%%%%%%%%%%%%%%%%%%%%%%%%%%%%%%%%%%%%%%%%%
%%%%%%%%%%%%%%%%%%%%%%%%%%%%%%%%%%%%%%%%%%%%%%%%%
%%%%%%%%%%%%%%%%%%%%%%%%%%%%%%%%%%%%%%%%%%%%%%%%%
%%%%%%%%%%%%%%%%%%%%%%%%%%%%%%%%%%%%%%%%%%%%%%%%%
%%%%%%%%%%%%%%%%%%%%%%%%%%%%%%%%%%%%%%%%%%%%%%%%%

\section{Structure theory of BKM superalgebras} \label{section2} In this section, we recall the basic properties and the denominator identity of BKM superalgebras from \cite{WM01}.  The theory of BKM superalgebras can also be seen in \cite{UR95}.
%The complex numbers, integers, non-negative integers, and positive integers are denoted  by $\bc$, $\bz$, $\bz_+$, and $\bn$. 
Our base field will be complex numbers throughout the paper. % %We mainly follow the notations and definitions from this book.
%%%%%%%%%%%%%%%%%%%%%%%%%%%%%%%%%%%%%%%%%%%%%%%%%
%%%%%%%%%%%%%%%%%%%%%%%%%%%%%%%%%%%%%%%%%%%%%%%%%
%%%%%%%%%%%%%%%%%%%%%%%%%%%%%%%%%%%%%%%%%%%%%%%%%
%%%%%%%%%%%%%%%%%%%%%%%%%%%%%%%%%%%%%%%%%%%%%%%%%
%%%%%%%%%%%%%%%%%%%%%%%%%%%%%%%%%%%%%%%%%%%%%%%%%
%\section{Structure theory of BKM superalgebras}\label{section2}
%Throughout this paper our base field will be complex numbers, i.e., all the algebras and representations are complex--vector spaces. The complex numbers, integers, non-negative integers, and positive integers are denoted  by $\bc$, $\bz$, $\bz_+$, and $\bn$. 
%\subsection{BKM superalgebra}
%In this section, we give the definition and basis result in the theory of BKM superalgebras. 
%The theory of BKM superalgebras can also be seen in \cite{UR95}.
%The complex numbers, integers, non-negative integers, and positive integers are denoted  by $\bc$, $\bz$, $\bz_+$, and $\bn$. 
%Our base field will be complex numbers throughout the paper. %In what follows, for a finite/countably infinite set $I$, $\mathbb Z_+[I]$ will denote the set of all tuples $(k_i: k_i \in \mathbb Z_+ \text{ and } i \in I)$ such that all but a finitely many $k_i$s are zero.
\subsection{Generators and Relations}\label{sec 2.1}
Let $I = \{1,2,\dots,n\}$ or the set of natural numbers. Fix a subset $\Psi$ of $I$ to describe the odd simple roots.
%\textcolor{blue}{(should this be finite? for chromatic polynomial pblm need not be finite and for unique factorization we may need finite as in VS)} and $S \subset I$. The set $I$ will be indexing the simple roots and $S$ will be indexing the odd generators of a BKM superalgebras.
A complex matrix  
$A=(a_{ij})_{i,j\in I}$ together with a choice of $\Psi$ is said to be a Borcherds-Kac-Moody supermatrix (BKM supermatrix in short) if the following conditions are satisfied:
For $i,j \in I$ we have
\begin{enumerate}
	\item $a_{ii}=2$ or $a_{ii}\leq 0$.
	\item $a_{ij}\leq 0$ if $i\neq j$.
	\item $a_{ij} = 0$ if and only if $a_{ji} =0$.
	\item $a_{ij} \in\mathbb{Z}$ if $a_{ii}=2$.
	\item $\a_{ij} \in2\mathbb{Z}$ if $a_{ii}=2$ and $i \in \Psi$.
	\item $A$ is symmetrizable, i.e., $DA$ is symmetric for some diagonal matrix $D=\mathrm{diag}(d_1, \ldots, d_n)$ with positive entries.
	%\item $a_{ij}=0$ if and only if $a_{ji}=0$.
\end{enumerate} 
%Recall that a matrix $A$ is called symmetrizable if .\item $A$ is symmetrizable. i.e., $DA$ is symmetric for some diagonal matrix $D=\mathrm{diag}(d_1, \ldots, d_n)$ with positive entries;
Denote by $I^{\mathrm{re}}=\{i\in I: a_{ii}=2\}$, $I^{\mathrm{im}}=I\backslash I^{\mathrm{re}}$, $\Psi^{re} = \Psi \cap I^{re}$, and $\Psi_0 = \{i \in \Psi : a_{ii}= 0\}$. 
The Borcherds-Kac-Moody Lie superalgebra (BKM superalgebra in short) associated with a BKM supermatrix $(A,\Psi)$ is the Lie superalgebra $\lie g(A,\Psi)$ (simply $\lie g$ when the presence of $A$ and $\Psi$ are understood)  generated by $e_i, f_i, h_i, i \in I$ with the following defining relations \cite[Equations (2.10)-(2.13), (2.24)-(2.26)]{WM01}:
\begin{enumerate}
	\item $[h_i, h_j]=0$ for $i,j\in I$,
	\item $[h_i, e_j]=a_{ij}e_j$,  $[h_i, f_j]=-a_{ij}f_j$ for $i,j\in I$,
	\item $[e_i, f_j]=\delta_{ij}h_i$ for $i, j\in I$,
	\item $\deg h_i = 0, i \in I$,
	\item $\deg e_i = 0 = \deg f_i$ if $i \notin \Psi$,
	\item $\deg e_i = 1 = \deg f_i$ if $i \in \Psi$,
	\item $(\text{ad }e_i)^{1-a_{ij}}e_j=0 = (\text{ad }f_i)^{1-a_{ij}}f_j$ if $i \in I^{re}$ and $i \ne j$,
	\item $(\text{ad }e_i)^{1-\frac{a_{ij}}{2}}e_j=0 = (\text{ad }f_i)^{1-\frac{a_{ij}}{2}}f_j$ if $i \in \Psi^{re}$ and $i \ne j$,
	\item $(\text{ad }e_i)^{1-\frac{a_{ij}}{2}}e_j=0 = (\text{ad }f_i)^{1-\frac{a_{ij}}{2}}f_j$ if $i \in \Psi_0$ and $i = j$,
	\item $[e_i, e_j]= 0 = [f_i, f_j]$ if $a_{ij}=0$. 
	%In particular, $[e_i, e_i]= 0 = [f_i, f_i]$ for $i \in \Psi_0$.
\end{enumerate}
The relations (7), (8) and (9) are called the Serre relations of $\lie g$. We define $I_j = \{i \in I: \deg e_i = j\}$ for $j=0,1$ and theses sets will be identified with the set of even and odd simple roots of $\lie g$ respectively.
\begin{rem}
	If $\Psi$ is the empty set then $(A,\Psi)$ becomes a Borcherds Cartan matrix and the resulting Lie algebra $\lie g(A)$ is a Borcherds algebra. Assume that $a_{ii} \ne 0$. If $i \in I \backslash \Psi$, then the Lie subsuperalgebra $S_i = \mathbb{C}f_i \oplus \mathbb{C}h_i  \oplus \mathbb{C}e_i$ of the BKM superalgebra $\lie g$ is isomorphic to $\mathfrak{sl}_2$ and if $i \in \Psi$, then the Lie sub-superalgebra $S_i = \mathbb{C}[f_i,f_i]\oplus \mathbb{C}f_i \oplus \mathbb{C}h_i \oplus \mathbb{C}e_i \oplus \mathbb{C}[e_i,e_i]$ is isomorphic to $\lie {sl}(0,1)$. If $a_{ii}=0$, the Lie sub-superalgebra $S_i = \mathbb{C}f_i \oplus \mathbb{C}h_i  \oplus \mathbb{C}e_i$ is isomorphic to the three dimensional Heisenberg Lie algebra (resp. superalgebra) if $i \in I \backslash \Psi$ (resp. if $i \in \Psi$). The conditions (5) and (6) defines a $\mathbb Z_2$ gradation on $\lie g$ which makes it a Lie superalgebra. The presence of $\lie {sl}(0,1)$ explains the appearances of even integers in the definition of BKM supermatrix. This fact is also reflected in condition (8) of the defining relations of $\lie g$. This is one of the main structural difference between Borcherds algebras and the BKM superalgebras.
\end{rem}

\subsection{Quasi Dynkin diagram}\label{section3}
First, we define the notion of a supergraph.
\begin{defn}\label{sg}
	Let $G$ be a finite/countably infinite simple graph with vertex set $I$. Let $\Psi \subseteq I$ be a subset of the vertex set. Then the pair $(G,\Psi)$ is said to be a supergraph and the vertices in $\Psi$ (resp. $I \backslash \Psi$) are said to be odd (resp. even) vertices of $G$. Let $A$ be the classical adjacency matrix of the graph $G$. Then the pair $(A,\Psi)$ is said to be the adjacency matrix of the supergraph $(G,\Psi)$.
\end{defn}
 %Given this definition the quasi Dynkin diagram of a BKM superalgebra $\lie g(A)$ is defined as follows \cite[Section 2.1]{WM01}.
\begin{defn}\label{qdd}
	The quasi Dynkin diagram of a BKM superalgebra is defined as follows \cite[Section 2.1]{WM01}. Let $(A=(a_{ij}),\Psi)$ be a BKM supermatrix and let $\lie g$ be the associated BKM superalgebra. The quasi Dynkin diagram of $\lie g$ is the supergraph $(G,\Psi)$ with vertex set $I$ and two vertices $i,j \in 
	I$ are connected by an edge if and only if $a_{ij} \ne 0$.   We often refer to $(G,\Psi)$ simply as the graph of $\lie g$. An example of a quasi Dynkin diagram of a BKM superalgebra $\lie g$ is given in Example \ref{basis1_ex1}.%Note that the set $I$ is the indexing set of the matrix $A$.
\end{defn}

\begin{rem}
	We observe that the quasi Dynkin diagram of $\lie g$ is the underlying simple graph of the classical Dynkin diagram of $\lie g$ \cite[Definition 2.4 above]{WM01}. In other words, the quasi Dynkin diagram can be obtained from the Dynkin diagram of $\lie g$ by replacing all the multi edges with a single edge. In Section \ref{multis}, we find a connection between root multiplicities of a BKM superalgebra $\lie g(A)$ and the chromatic polynomial of the associated supergraph $(G,\Psi)$. We will see that the chromatic polynomial of $G$ depends on whether two vertices of $G$ are adjacent or not but not on the actual number of edges connecting them. Therefore, we work with quasi Dynkin diagrams instead of the Dynkin diagrams. 
\end{rem}

%We will refer to $G$ as the graph of $\lie g.$ 
For any subset $S\subseteq \Pi$, we denote by $|S|$ the number of elements in $S$. The subgraph induced by the subset $S$ is denoted by $G_S$. 
We say a subset $S\subseteq \Pi$ is \em{connected} if the corresponding subgraph $G_S$ is connected. Also, we say $S$ is {\em{independent}} if there is no edge between any two  elements of $S$, i.e., $G_S$ is totally disconnected.

\subsection{Root system and the Weyl group}
Let $\Delta$ be the root system of a BKM superalgebra $\lie g$ \cite[Section 2.3]{WM01}. Let $\Pi$ the set of simple roots of $\lie g$. 
Define $Q:=\bigoplus _{\a\in \Pi}\mathbb{Z}\alpha,\ \ Q_+ :=\sum _{\a\in \Pi}\mathbb{Z}_{+}\alpha.$ % :=\{ \alpha \in (D \ltimes\lie h)^*\backslash \{0\} \mid \mathfrak{g}_{\alpha }\neq 0\}$ 
%and
\begin{defn}\label{free def}
	An element $\alpha= \sum_{i \in I}k_i \alpha_i \in Q_+$ (or its weight $\bold k = (k_i : i \in I) \in \mathbb Z_+[I]$) is said to be free if $k_i \le 1 $ for $i \in I^{re}\sqcup \Psi_0$.
\end{defn}
\begin{rem}
	In Lemma \ref{identification lem}, we will show that any root space of a BKM superalgebra that corresponds to a free root is independent of (or free from) the Serre relations.
\end{rem}

The set of positive roots is denoted by  $\Delta_+:=\Delta\cap Q_+$. 
%\textcolor{blue}{But we want to assume that roots are in $\lie h$ and so roots are need not be even or odd. We need to talk about $m_0(\alpha) = \dim \lie g_{\alpha} \cap \lie g_0$ and $m_1(\alpha) = \dim \lie g_{\alpha} \cap \lie g_1$}. 
All the root spaces of $\lie g$ are finite dimensional and for any $\alpha \in \lie h^{*}$ either $\lie g_{\alpha} \subset \lie g_0$ or $\lie g_{\alpha} \subset \lie g_1$, i.e, every root is either even or odd. Set $\Delta_+^0$ (resp. $\Delta_+^{1}$) to be the set of positive even (resp. odd) roots.
%The elements in $\Pi^\mathrm{re}:=\{\alpha_i:  a_{ii}>0\} = \Pi_0^{\mathrm{re}} \sqcup \Pi_1^{\mathrm{re}}$ where $\Pi_0^{\mathrm{re}} = \{\alpha_i \in \Pi^{\mathrm{re}} : i \notin \Psi\}$ and $\Pi_1^{\mathrm{re}} = \Pi^{\mathrm{re}} \backslash \Pi_0^{\mathrm{re}}$. Also, we define  $\Pi^\mathrm{im}:=\Pi\backslash \Pi^\mathrm{re}$. We define the sets of real and imaginary non-reduced simple roots respectively as follows: $\widetilde{\Pi}^{re} = \Pi^{re} \sqcup \{2 \alpha_i : i \in \Psi, a_{ii}>0 \}$ and  $\widetilde{\Pi}^{im} = \Pi^{im} \sqcup \{2 \alpha_i : i \in \Psi, a_{ii}<0 \}$. We have $\Delta =\Delta_+ \sqcup - \Delta_+$ and 
%$$\mathfrak{g}_0=\mathfrak{h},\ \ \lie g_\alpha=\lie g(k_1, k_2, \dots),\ \text{ if }\ \alpha=\sum_{i\in I} k_i\alpha_i\in \Delta.$$
%The combinatorial properties of the set $\Delta^m(\lie g)$ is explored in Section \ref{}
%\subsection{Weyl group of $\lie g$}
%There is a symmetric bilinear form (since $A$ is symmetrizable) on $D\ltimes\lie h$ defined by
%$(h_i,h) = \alpha_i(h), h \in D\ltimes\lie h$ and $(h,h^{'}) = 0, h,h^{'} \in D$. This form is non-degenerate on $D\ltimes\lie h$ and kernel of $(.,.)|\lie h$ is the center $C (\subset \lie h)$ of $\lie g$. This form extends to a unique bilinear  supersymmetric invariant form on $\lie g$ with the following properties: $(\lie g _{\alpha},\lie g_{\beta}) = 0$ if $\alpha + \beta \ne 0$; $[x,y] =(x,y) \nu^{-1} (\alpha), \alpha \in \Delta$ and $(.,.)$ is non-degenerate on $\lie g$. 
%The coroot associated with $\a\in \Pi$ is denoted by $h_\a.$
\noindent
We have a triangular decomposition
$\lie g\cong \lie n^{-}\oplus \lie h \oplus \lie n^+,$
where
$\lie n^{\pm}=\bigoplus_{\alpha \in \pm\Delta_{+}}
\mathfrak{g}_{\alpha}.$
Given $\gamma=\sum_{i\in I}k_i\alpha_i\in Q_+$, we set $\text{ht}(\gamma):=\sum_{i\in I}k_i.$ 
%Finally, for $\lambda, \mu\in (D\ltimes \mathfrak{h})^*$  we say that $\lambda\ge \mu$ if $\lambda-\mu\in Q_+.$
%\subsection{The Weyl group}
The real vector space spanned by $\Delta$ is denoted by $R=\mathbb{R}\otimes_{\bz} Q$. %There is a symmetric bilinear form on $R$ given by $(\alpha_i, \alpha_j)=d_i a_{ij}$ 
%for $i, j\in I.$ 
For $\alpha\in \Pi^{\mathrm{re}}$, define the linear isomorphism $\bold{s}_\alpha$ of $R$ by 
$\bold{s}_\alpha(\lambda)=\lambda-2\frac{(\lambda,\alpha)}{(\alpha,\alpha)}\alpha,\ \ \lambda\in R.$
Note that the simple reflections are defined for odd real simple roots also. The Weyl group $W$ 
of $\mathfrak{g}$ is the subgroup of $\mathrm{GL}(R)$ generated by the simple reflections $\bold{s}_\alpha$, $\alpha\in \Pi^\mathrm{re}$.
Note that the above bilinear form is $W$--invariant and $W$ is a Coxeter group with canonical generators $\bold{s}_\alpha, \a\in \Pi^\mathrm{re}$. 
Define the length of $w\in W$ by  $\ell(w):=\mathrm{min}\{k\in \mathbb{N}: w=\bold{s}_{\a_{i_1}}\cdots \bold{s}_{\a_{i_k}}\}$ and any expression $w=\bold{s}_{\a_{i_1}}\cdots \bold{s}_{\a_{i_k}}$ with 
$k=\ell(w)$ is called a reduced expression. The set of real roots is denoted by $\Delta^{\mathrm{re}}=W(\widetilde{\Pi}^{\mathrm{re}})$ and the set of imaginary roots is denoted by $\Delta^\mathrm{im}=\Delta\backslash \Delta^\mathrm{re}$. 
Equivalently, a root $\alpha$ is real if and only if $(\alpha, \alpha)> 0$ and else imaginary. 
%We can extend $(.,.)$ 
%to a symmetric form on $(D\ltimes \lie h)^*$ satisfying $(\lambda,\alpha_i)=\lambda(d_ih_i)$ and also
%$\bold{s}_\alpha$ to a linear isomorphism of $(D\ltimes \lie h)^*$ by 
%$$\bold{s}_\alpha(\lambda)=\lambda-2\frac{(\lambda,\alpha)}{(\alpha,\alpha)}\alpha,\ \ \lambda\in (D\ltimes \lie h)^*.$$
%Note that $\lambda(h_\a)=2\frac{(\lambda,\alpha)}{(\alpha,\alpha)}$ for $\a\in \Pi$.
Let $\rho$ be any element of $\lie h^*$ satisfying $2(\rho,\alpha)=(\alpha,\alpha)$ for all $\a\in \Pi$.

\subsection{Denominator identity of BKM superalgebras}
Let $\Omega$ be the set of all $\gamma \in Q_+$ such that 
\begin{enumerate}
	\item $\gamma = \sum_{j=1}^r \alpha_{i_j} + \sum_{k=1}^s l_{i_k}\beta_{i_k}$ where the $\alpha_{i_j}$ (resp. $\beta_{i_k}$) are  distinct even (resp. odd) imaginary simple roots,
	\item $(\alpha_{i_j},\alpha_{i_k}) = (\beta_{i_j},\beta_{i_k}) = 0$ for $j \ne k$; $(\alpha_{i_j},\beta_{i_k}) =0$ for all $j,k$;
	\item if $l_{i_k} \ge 2$, then $(\beta_{i_k},\beta_{i_k}) = 0$.
\end{enumerate}
%Let $\Omega(\lambda)$ be the set of all $\gamma \in Q_+$ such that 
%\begin{enumerate}
%	\item $\gamma = \sum_{j=1}^r \alpha_{i_j} + \sum_{k=1}^s l_{i_k}\beta_{i_k}$ where the $\alpha_{i_j}$ (resp. $\beta_{i_k}$) are  distinct even (resp. odd) simple imaginary roots,
%	\item $(\alpha_{i_j},\alpha_{i_k}) = (\beta_{i_j},\beta_{i_k}) = 0$ for $j \ne k$; $(\alpha_{i_j},\beta_{i_k}) =0$ for all $j,k$;
%	\item $(\lambda,\alpha_{i_j}) = (\lambda,\beta_{i_k}) = 0$ and if $l_{i_k} \ge 2$, then $(\beta_{i_k},\beta_{i_k}) = 0$.
%\end{enumerate}
%i.e., $\gamma$ is the sum of mutually orthogonal distinct imaginary simple roots which are orthogonal to $\lambda$.
%We define the normalized Weyl numerator by:
%\begin{equation}\label{Ulambda}
%U_\lambda:= \sum\limits_{(w,\mu)\in W \times \Omega(\lambda)}\epsilon(w)\epsilon(\mu)e^{w(\lambda+\rho-\mu)-(\lambda + \rho)}.
%\end{equation}
%\noindent
%Note that $0\in \Omega(\lambda)$ and that an imaginary simple root $\alpha$ is in $\Omega(\lambda)$ if $(\lambda, \alpha)=0.$
The following denominator identity of BKM superalgebras is proved in \cite[Section 2.6]{WM01}:

%\begin{equation}\label{WeylKac}
%\text{ch} L(\lambda)e^{-\lambda}=  
%\frac{\sum\limits_{(w, \mu) \in W\times \Omega(\lambda) } \epsilon(w)\epsilon(\mu) e^{w(\lambda+\rho -\mu)-(\lambda + \rho)}}{\sum\limits_{(w, \mu) \in W\times \Omega(0) } \epsilon(w)\epsilon(\mu) e^{(w(\rho-\mu)-\rho)}}=\frac{U_\lambda}{U_0}
%\end{equation}

%By substituting $\lambda = 0$ in \eqref{WeylKac} we get the following denominator identity.

\begin{eqnarray}\label{denominator}
U:=  \sum_{w \in W }  \sum_{ \gamma \in
	\Omega} \epsilon(w) \epsilon(\gamma) e^{w(\rho -\gamma)-\rho } & =& \frac{\prod_{\alpha \in \Delta_+^0} (1 - e^{-\alpha})^{\mult(\alpha)}} {\prod_{\alpha \in \Delta_+^1} (1 + e^{-\alpha})^{\mult(\alpha)}}
\label{eq:0,1denom} \end{eqnarray}  
where $\mult(\alpha) = \dim \lie g_{\alpha}, \epsilon(w) = (-1)^{l(w)}$ and $\epsilon(\gamma) = (-1)^{\htt \gamma}$.
 %The following definition will be used to describe the factors of $U_\lambda$.
%\begin{defn}
%	A $k$--partition of the graph $G$ is an ordered $k$--tuple $(J_1,J_2,...,J_k)$ such that the following conditions hold
%	\begin{enumerate}
%		\item the $J_i$ are non-empty pairwise disjoint subsets of the vertex set $V$ whose union is $V$; 
%		\item each $J_i$ is an independent subset  of $V$.
%	\end{enumerate}
%%	We denote by $P_k(G)$ the set of all $k$--partitions of $G$ and $c_k(G):=|P_k(G)|$.  We also define 
%	$$c(G):=(-1)^n \sum\limits_{k=1}^{n}(-1)^k\frac{c_k(G)}{k}.$$
%\end{defn}
%\noindent
\begin{rem}
	If $\Psi$ is the empty set then Equation \eqref{denominator} reduces to the denominator identity of the Borcherds algebras. Further, if $I^{im}$ is also empty then Equation \eqref{denominator}  reduces to the denominator identity of the Kac-Moody algebras.
\end{rem}

\section{Main result I: Lyndon basis of BKM superalgebras}\label{identificationsec}
In this section, we identify the free root spaces of a BKM superalgebra with the grade spaces of free partially commutative Lie superalgebra. Using this identification, we construct the Lyndon basis for the free root spaces of a BKM superalgebra $\lie g$. We start with the definition of free partially commutative Lie superalgebras.%In view of Lemma \ref{identification lem} we construct a Lyndon heaps basis for the free partially commutative Lie superalgebras [c.f Theorem \ref{fpcsm}]. This is the main result of this section.
%In this section, we identify the free root spaces of a BKM superalgebra with certain grade spaces of free partially commutative Lie superalgebras. As an application, we will show that the basis of free root spaces of Borcherds algebras given in Theorem \ref{mainthmbb} is same as the Lalonde's Lyndon heaps basis of free partially commutative Lie algebras given in Proposition \ref{lafp}. 

\subsection{Free partially commutative Lie superalgebra $\mathcal{LS}(G,\Psi)$} %We start with the definitions of supergraph and the associated free partially commutative Lie superalgebra.

Given a supergraph $(G,\Psi)$, the associated free partially commutative Lie superalgebra $\mathcal{LS}(G,\Psi)$ is defined as follows. First, we define the free Lie superalgebra on a $\mathbb Z_2$-graded set.
\begin{defn}
	Let $I = I_0 \sqcup I_1$ be a non-empty superset  ($\mathbb Z_2$-graded set). Let $I^{\ast}$ be the free monoid generated by $I$. A word $\bold w \in I^{\ast}$ is called even if the number of odd alphabets (i.e. the elements of $I_1$) in $\bold w$ is even otherwise it is called odd. This defines a $\mathbb Z_2-$gradation on $I^*$. Let $V$ be the free super vector space on the set $I$ and let $T(V)$ be the tensor algebra on $V$. The algebra $T(V)$ has an induced $\mathbb Z_2$-gradation makes it an associative superalgebra for which $I^{*}$ is a basis. Since $T(V)$ is associative, it has a natural Lie superalgebra structure. Given this, the free Lie superalgebra on the superset $I = I_0 \sqcup I_1$ is defined to be the smallest Lie subsuperalgebra of $T(V)$ containing $I$. We denote the free Lie superalgebra on the superset $I$ by $\mathcal{FLS}(I)$. 
\end{defn}
\begin{rem}\label{rootrootvector}
	If $I_1$ is the empty set then $\mathcal{FLS}(I)$ is the free Lie algebra on set $I_0$. Whenever we talk about a free Lie superalgebra on a set $I = \{1,\dots,n\}$ or $\{1,2,\dots\}$, we consider the elements of $I$ as $\{e_1,\dots,e_n\}$ and $\{e_1,e_2, \dots\}$ instead  $I$. This way, we can relate the elements of $\mathcal{FLS}(I)$ with the elements of a free root spaces of a BKM superalgebra $\lie g$.
	
	%Let $\lie g$ be a BKM superalgebra with the associated graph $G$. Then the vertex set $I$ of $G$ can be identified with the set of simple roots $\Pi$ of $\lie g$. 
	
	%We observe that $FLS(I)$ consists of all the Lie monomials in $I$ and the positive part $\lie n_+$ of $\lie g$ consists of Lie monomials in $\{e_i:\alpha_i \in I\}$. Here after, we always identify the simple root $\alpha_i$ with its root vector $e_i$. This way we can compare the Lie monomials in $FLS(I)$ and the lie monomials in some positive root space $\lie g_{\eta(\bold k)}$. Same is true for if we replace $e_i$s by $f_i$s if we consider the negative of simple roots.
\end{rem}
\begin{defn}
	Let $(G,\Psi)$ be a supergraph with the vertex set $I$ and the edge set $E$. Let $\mathcal{FLS}(I)$ be the free Lie superalgebra on the set $I = I_0 \sqcup I_1$ where we take $I_1$ to be $\Psi$. Let $J$ be the ideal in $\mathcal{FLS}(I)$ generated by the relations $\{[e_i,e_j] : \{i,j\} \notin E \}$ [c.f. Remark \ref{rootrootvector}]. The quotient algebra $\frac{\mathcal{FLS}(I)}{J}$, denoted by $\mathcal{LS}(G,\Psi)$, is the free partially commutative Lie superalgebra associated with the supergraph $(G,\Psi)$. When $\Psi$ is the empty set, $\mathcal{LS}(G,\Psi)$ is the free partially commutative Lie algebra associated with the graph $G$ and is denoted by $\mathcal{L}(G)$. It is well-known that $\mathcal{FLS}(I)$ and hence $\mathcal{LS}(G,\Psi)$ is graded by $\mathbb{Z}_{+}[I]$.
\end{defn}

\subsection{Free partially commutative super monoid}\label{fpcsm} %\textit{Since $\lie g_{\eta(\bold k)}\neq 0$ implies that $\text{supp}(\bold k)=\{i\in I: k_i\neq 0\}$ is connected we will assume without loss of generality for the rest of this section that $I$ is connected and $I=\text{supp}(\bold k)$ and in particular $I$ is finite}. We freely use the notations introduced in the previous sections.

Let $(G,\Psi)$ be a supergraph with a finite/countably infinite vertex set $I = I_0 \sqcup I_1$ (where $I_1 = \Psi$). We assume that $I$ is totally ordered. Let $I^*$ be the free monoid generated by $I$.  We note that $I^*$ is totally ordered by the lexicographical order. 
%A word $\bold w \in I^*$ is called even if the number of odd alphabets (i.e. the elements of $I_1$) in $\bold w$ is even otherwise it is called odd. This defines a $\mathbb Z_2-$gradation on $I^*$.  
The free partially commutative super monoid associated with a supergraph $(G,\Psi)$ is denoted by $M(I, G,\Psi):=I^*/\sim$, where $\sim$ is generated by the relations $ab\sim ba, \text{ if } (a, b)\notin E(G). $ When $\Psi$ is empty,  $M(I, G,\Psi)$ is called the free partially commutative monoid associated with the graph $G$ and denoted simply by $M(I, G)$.
We observe that $M(I,G,\Psi)$ has a natural $\mathbb Z_2$-gradation induced from the $\mathbb Z_2$-gradation of $I^*$. We associate with each element $[a] \in M(I,G,\Psi)$ a unique element $\tilde{a}\in I^*$ which is the maximal element in $[a]$ with respect to the lexicographical order. We call this element the standard word of the class $[a]$ and denoted by $\st([a])$. A total order on $M(I, G,\Psi)$ is then given by
\begin{equation}\label{totor}[a]<[b] :\Leftrightarrow \st{[a]}<\st{[b]}.\end{equation}

Next, we explain the Lyndon heaps basis of free partially commutative Lie algebras. For this reason, in the next subsection, we give the essential definitions from the theory of heaps of pieces to define pyramids and Lyndon heaps from \cite{la95}. 
\subsection{Heaps monoid, Pyramids, and Lyndon heaps}\label{pyramidsec}
%\subsection{Heaps monoid}\label{heaps}
Heaps of pieces were introduced by Xavier Viennot in \cite{xv89}. He proved many combinatorial results on heaps of pieces and gave applications of heaps of pieces to a wide range of areas: directed animals, polyominoes, Motzkin paths, and orthogonal polynomials, Rogers-Ramanujan identities, fully commutative elements in Coxeter groups, Bessel functions, Lorentzian quantum gravity and may many more applications in mathematical physics. 
In \cite{a3, vgx06,MR2002764}, special types of heaps namely pyramids are the important tools in proving results. Heaps of pieces have also applications in the representation theory of complex simple Lie algebras \cite{rmg}. In this book, the combinatorial aspects of minuscule representations are studied using heaps of pieces.  In \cite{a3}, the connection between heaps of piece , chromatic polynomials and the free partially commutative Lie algebras is discussed along with many other combinatorial properties of heaps of pieces.

Let $(G,\Psi)$ be a supergraph with a (finite/countably infinite) totally ordered vertex set $I = \{\alpha_1,\dots,\alpha_k\}$ or $I = \{\alpha_1,\alpha_2,\dots\}$. %Let $M(I,G)$ be the free partially commutative monoid associated with the graph $G$.
%The \emph{free monoid} on $I$, denoted by $\mathcal{M}(I)$, is totally ordered by the lexicographic order induced from the total order on $I$. We say two elements $i$ and $j$ of $I$ commute if $\{i,j\} \notin E(G)$. We use this commutation relation on $I$ to define an equivalence relation $\eta$ on $\mathcal{M}(I)$: Two words $w_{1}$ and $w_{2}$ in $\mathcal{M}(I)$ are related by $\eta$ if $w_{2}$ is obtained from $w_{1}$ by a sequence of interchanges of adjacent commuting alphabets. The \emph{free partially commutative monoid} associated with $G$, denoted by $\mathcal{M}(I,\eta)$, is defined to be the set of all equivalence classes $\frac{\mathcal{M}(I)}{\eta}$. Note that, $\mathcal{M}(I,\eta)$ has a natural monoid structure induced from the monoid structure on $\mathcal{M}(I)$.
Let $\zeta$ be the concurrency relation complement to the commuting relation $\sim$ on $I^{*}$ [c.f. Section \ref{fpcsm}]. A \emph{pre-heap} $E$ over $(I,\zeta)$ is a \textbf{finite} subset of $I \times \{0,1,2, \dots \}$ satisfying, if $(\alpha_1,m),(\alpha_2,n) \in E$ with $\alpha_1 \,\zeta\, \alpha_2$, then $m \ne n$. Each element $(\alpha,m)$ of $E$ is called a basic piece.  If $(\alpha,m) \in E$, we write $\pi(\alpha,m) = \alpha $ (the position of the piece $(\alpha,m)$) and $h(\alpha,m) = m$ (the level of the piece $(\alpha,m)$). A basic piece will be simply denoted by $\alpha$ when we don't need to emphasize the level. The set $\pi(E)$ is defined to be the set of all positions occupied by the pieces of $E$.
A pre-heap $E$ defines a partial order $\le_E$ by taking the transitive closure of the relation : 
$(\alpha_1,m) \le_E (\alpha_2,n)$ if $\alpha_1 \zeta \alpha_2$ and $m < n$. 
We say that two heaps $E$ and $F$ are \emph{isomorphic} if there exists a position preserving order isomorphism $\phi$ between $(E,\le_E) $ and $(F,\le_F)$. 
A \emph{heap} $E$ over $(I,\zeta)$ is a pre-heap over $(I,\zeta)$ such that: if $(\alpha,m) \in E$ with $m > 0$ then there exists $(\beta,m-1) \in E$ such that $\alpha \zeta \beta$.
Every isomorphism class of pre-heaps contains exactly one heap and this is the unique pre-heap $E$ in the class for which $\sum_{\alpha \in E}h(\alpha)$ is minimal. 
\begin{rem}
	The graph $G$ can have a countably infinite number of vertices, but each heap $E$ over the graph $G$ has only a finite number of pieces by definition. This fact leads to a natural $\mathbb Z_+[I]$-gradation on the set of all heaps over the graph $G$. 
\end{rem}	

%The pictorial representation of heaps can be seen in \cite{xv89,viennot-imsc2,viennot-Talca,viennot-imsc5}.

Let $\mathcal{H}(I,\zeta)$ be the set of all heaps over $(I,\zeta)$. This set can be made into a monoid with a product called the \emph{superposition} of heaps. To get superposition $E \circ F$ of $F$ over $E$, let the heap $F$ `fall' over $E$. Let $\mathcal{H}_{\bold k}(I,\zeta)$ be the set of all heaps of weight $\bold k$ for $\bold k \in \mathbb{Z}_{+}[I]$ where the weight counts the number of pieces in each of the positions. This gives a $\mathbb{Z}_{+}[I]$-gradation on $\mathcal{H}(I,\zeta)$.
We define a map $\psi : I^* \rightarrow \mathcal{H}(I,\zeta)$ as follows:  For a word $p_1\,p_2\, \cdots \,p_k \in I^*$ define $\psi(p_1\,p_2\, \cdots \,p_k) = p_1 \,\circ\, p_2 \,\circ \cdots \circ\, p_k $. Note that $\psi^{-1}(E)$ is the set of all linear orders compatible with $\le_E$. It is clear that $\psi$ extends to  weight and order-preserving isomorphism of the monoids $M(I,G,\Psi)$ and $\mathcal{H}(I,\zeta)$. This defines a total order on $\mathcal H(I,\zeta)$. It also defines a $\mathbb Z_2$-grading $\mathcal H(I,\zeta) = \mathcal H_0(I,\zeta)  \oplus \mathcal H_1(I,\zeta) $. The standard word of a heap $E$ is defined to be $\st(E) = \st(\psi^{-1}(E))$ [c.f. Equation \eqref{totor}].
%\subsection{Pyramids and Lyndon heaps}\label{pyramidsec}
For a heap $E$, $\min E$ is the heap composed of minimal pieces of $E$ with respect to $\le_E$ and $\max E$ is defined similarly. We write $|E|$ for the number of pieces in $E$ and $|E|_{\alpha}$ for the number of pieces of $E$ in the position $\alpha$. A heap $E$ such that $\min (E) = \{ \alpha \}$ is said to be a \emph{pyramid} with the basis $\alpha$. The set of all pyramids in $\mathcal{H}(I,\zeta)$ is denoted by $\mathcal{P}(I,\zeta)$ and the set of all pyramids with basis $\alpha_i$ is denoted by $\mathcal{P}^{i}(I,\zeta)$. %We note that if $E$ is a pyramid then $\pi(E)$ is connected. 
%For a heap $E$, we define $\st(E) := \max \Psi^{-1}(E)$ the standard word associated with $E$. For two heaps $E,F$ we say $E \le F$ if $\st(E) \le \st(F)$. This defines a total order in the heaps monoid $\mathcal{H}(I,\zeta)$. \emph{For the rest of this paper we fix this total order in $\mathcal{H}(I,\zeta)$}.
%\end{defn}

Let $E$ be a heap, we say that $E$ is \emph{periodic} if there exists a heap $F \ne 0$ (0 - empty heap) and an integer $k \ge 2$ such that $E = F^k$.
% and \\
Similarly, $E$ is \emph{primitive} if $E = U \circ V = V \circ U$ then either $U = 0$ or $V = 0.$
%\end{defn}
%\end{defn}
%\begin{defn}
Pyramids in which the minimum piece has the lowest position (with respect to the total order on $I$) are known as \emph{admissible pyramids}. 
%and the set of all admissible pyramids is denoted by $\mathcal{AP}(I,\zeta)$. ie.,  $\mathcal{AP}(I,\zeta) = \{E \in \mathcal{P}(I,\zeta) : \text{min}(E) =  \text{min}~ \pi(E) ~ \}$. 
%\begin{defn}\label{multi}
A pyramid $E$ with the basis $p$ such that $|E|_p = 1$ is said to an elementary pyramid.
An admissible pyramid that is also elementary is known as a \emph{super-letter}. The set of all super-letters in $\mathcal H(I,\zeta)$ is denoted by $\mathcal A(I,\zeta)$.
A heap $E$ in $\mathcal{H}(I,\zeta)$ is said to be \emph{multilinear} if every basic piece occurs exactly once in $E$. %The set of all multilinear heaps of $G$ is denoted by $\mathcal{H}_{\bold 1}(I,\zeta)$.  %\begin{defn}	
%\end{defn}
%\begin{defn}

Let $E$ be a heap. If $E$ = $U \circ V$ for some heaps $U$ and $V,$ we say that $V \circ U$ is a \emph{transpose} of $E$. The transitive closure of transposition is an equivalence relation on $\mathcal{H}(I,\zeta)$, which we call the conjugacy relation of heaps and is denoted by $\sim_c$.    
%\end{defn}
%\begin{defn}
A non-empty heap $E$ is said to be \emph{Lyndon} if $E$ is primitive and minimal in its conjugacy class. We write $\mathcal{LH}(I,\zeta)$ for the set of all Lyndon heaps over the super graph $(G,\Psi)$.
% and $\mathcal{LH}_{\bold k}(I,\zeta)$ denotes the Lyndon heaps of weight $\bold k$. 

\begin{rem}\label{pyrcon}
	According to Viennot, a pyramid is a heap with a unique maximal piece \cite[Definition 5.9]{xv89}. 
	In this paper, we follow Lalonde's convention on pyramids \cite{la95}, i.e., A pyramid is a heap with a unique minimal piece.
\end{rem}
Given the definition of Lyndon heaps, in the next subsection, we explain the Lyndon heaps basis of free partially commutative Lie algebras.
\subsection{Lyndon heaps basis of free partially commutative Lie algebras}
In this subsection, we recall the Lyndon heaps basis of Lalonde from \cite{la93}. 
If $ E$ is a Lyndon heap then the standard factorization $\sum( E)$ of $E$ is given by $\sum( E) = ( F, N)$, where \begin{enumerate}
	\item $F \ne 0 \,\,\,(\text{empty heap})$
	\item $E =  F \circ  N$
	\item $N$ is Lyndon
	\item $N$ is minimal in the total order on $\mathcal H(I,\zeta)$.
\end{enumerate}
To each Lyndon heap $E \in \mathcal H(I,\zeta)$ we associate a Lie monomial $\Lambda(E)$ in $\mathcal{L}(G)$ as follows. If $E\in I$, then $\Lambda(E)=E$ and otherwise $\Lambda( E)=[\Lambda(F_1),\Lambda(F_2)]$, where $\Sigma(E)=(F_1, F_2)$ is the standard factorization of $E$. Given these notions, we have the following theorem which gives the Lyndon basis of the free partially commutative Lie algebra $\mathcal{L}(G)$.
\begin{thm}\label{lafp}\cite{la93}
	The set $\{\Lambda(E): E\in \mathcal H(I,\zeta) \text{ is a Lyndon heap}\}$ forms a basis of $\mathcal {L}(G)$.
\end{thm}
\subsection{The identification of the spaces $\lie g_{\eta(\bold k)}$ and  $\mathcal{LS}_{\bold k}(G)$}
Let $\lie g$ be a BKM superalgebra with the associated supergraph $(G,\Psi)$. Let $\mathcal{LS}(G,\Psi)$ be the free partially commutative Lie superalgebra associated with the supergraph $(G,\Psi)$. Let $I$ be the vertex set of $G$. Fix $\bold k \in \mathbb Z_+[I]$ such that $k_i \le 1$ for $i \in I^{re}\sqcup \Psi_0$. In this subsection, we claim that there is a natural vector space isomorphism between the root space $\lie g_{\eta(\bold k)}$ of $\lie g$ and the grade space $\mathcal{LS}_{\bold k}(G,\Psi)$ of $\mathcal{LS}(G,\Psi)$. The precise statement is the following. 
 \begin{lem}\label{identification lem}
	Fix $\bold k \in \mathbb Z_+[I]$ such that $k_i \le 1$ for $i \in I^{re}\sqcup \Psi_0$.  Then 
	\begin{enumerate}
		\item The root space $\lie g_{\eta(\bold k)}$ can be identified with the grade space $\mathcal {LS}_{\bold k}(G)$ of the free partially commutative Lie superalgebra $\mathcal{LS}(G,\Psi)$. In particular, $\mult \eta(\bold k) = \dim \mathcal {LS}_{\bold k}(G,\psi)$.
		\item The root space $\lie g_{\eta(\bold k)}$ is independent of the Serre relations.
	
	\end{enumerate} 
	%= |\mathcal{LH}_{\bold k}(I,\zeta)|$, the number of super Lyndon heaps of weight $\bold k$.
\end{lem}
\begin{pf}
	The positive part $\lie n_+$ of $\lie g$ can be written as $\Big(\bigoplus\limits_{\substack{\alpha \in \Delta_+ \\ \text{free}}}\lie g_{\alpha}\Big) \bigoplus \Big(\bigoplus\limits_{\substack{\alpha \in \Delta_+ \\ \text{non-free}}}\lie g_{\alpha}\Big)$. From the defining relations (relation (9)) of $\lie g$, there is a natural grade preserving surjection $\phi$ from $\mathcal{LS}(G,\Psi)$ to $\lie n_+$. Further, by the defining relations (7), (8) and (9), the kernel of this map is generated by the elements 	$$(\text{ad }e_i)^{1-a_{ij}}e_j \text{ if } i \in I^{re} \text{ and }i \ne j,$$   $$(\text{ad }e_i)^{1-\frac{a_{ij}}{2}}e_j \text{ if } i \in \Psi^{re} \text{ and } i \ne j, \text{ and}$$
	$$(\text{ad }e_i)^{1-\frac{a_{ij}}{2}}e_j  \text{ if } i \in \Psi_0 \text{ and } i = j$$
	  of $\mathcal{LS}(G,\Psi)$. We observe that in all these elements some $e_i$'s (corresponding to a real simple root or an odd simple root of norm zero) are occurring at least twice. Since $\phi$ preserves the grading, by our assumption on $\bold k$, the grade space $\mathcal{LS}_{\bold k}(G,\Psi)$ is injectively mapped onto the free root space $\lie g_{\eta(\bold k)}$ of $\lie g$. This completes the proof.
\end{pf}

\subsection{Super Lyndon heaps and the standard factorization}\label{fpcl}
Let $\lie g$ be a BKM superalgebra with the associated supergraph $(G,\Psi)$. In Theorem \ref{lafp}, we saw the Lyndon heaps of the free partially commutative Lie algebra $\mathcal L(G)$.  We construct a similar basis for the free root spaces of the BKM superalgebra $\lie g$. 
%extend this basis to the case of free partially commutative Lie superalgebras by introducing the notion of super Lyndon heaps. 
%In Section \ref{identificationsec}, we have shown that the basis of free roots of a Borcherds algebra $\lie g$ given in Theorem \ref{mainthmbb} is same as the Lyndon heaps basis of free partially commutative Lie algebras given in Proposition \ref{lafp}. Therefore, 
	Given Lemma \ref{identification lem}, to construct such a basis it is enough to extend Theorem \ref{lafp} to the case of free partially commutative Lie superalgebra associated with the supergraph $(G,\Psi)$. This extension is the main result of this section [c.f. Theorem \ref{lafps}]. This is done by introducing and studying the combinatorial properties of super Lyndon heaps.

%In this subsection, we prove an analog result of Proposition \ref{reuts} for the free partially commutative Lie superalgebras $\mathcal{LS}(G,\Psi)$. 
\iffalse
In the previous section, we have seen that the Lyndon heaps in $\mathcal H(I,\zeta)$ indexes a basis for free partially commutative Lie algebras. Now, similar to the above discussion, we can define super Lyndon heaps. Again, we just need to take care of the grade spaces of $\mathcal{LS}_1(G)$ because this leads to  $[\mathcal{LS}_{\bold k}(G),\mathcal{LS}_{\bold k}(G)] \subseteq \mathcal{LS}_0(G)$ is non-zero. Then by the same argument as the above, these super Lyndon heaps will index a basis for free partially commutative Lie superalgebras. More precisely, we have the following definition of super Lyndon heaps.
\fi 
\begin{defn}
	Let $I = I_0 \sqcup I_1$ be a non-empty set. Let $(G,\Psi)$ be a supergraph with vertex set $I$ and $I_1 = \Psi$. A heap $E \in \mathcal H(I,\zeta) = \mathcal H_0(I,\zeta) \oplus \mathcal H_1(I,\zeta)$ (heap monoid over the supergraph $(G,\Psi)$) is said to be a super Lyndon heap if $E$ satisfies one of the following conditions:
	\begin{itemize}
		\item $E$ is a Lyndon heap.
		\item $E= F \circ F$ where $F \in \mathcal H_{1}(I,\zeta)$ is Lyndon.
	\end{itemize}
The set of all super Lyndon heaps over the supergraph $(G,\Psi)$ is denoted by $\mathcal{SLH(I,\zeta)}$.
\end{defn}

\begin{example} A super Lyndon heap over the path graph on $4$ vertices with $I_1=\{\alpha_1, \alpha_2, \alpha_3\}$ and $I_0=\{\alpha_4\}$ is the following.
	
	\iffalse
	\begin{center}
		\begin{tikzpicture}{centering}
		\tikzstyle{every node}=[draw, shape=circle];
		\node (a) at (0,0){$\alpha_1$};
		\node (b) at (1.5,0){$\alpha_2$};
		\node (c) at (3,0){$\alpha_3$};
		\node (d) at (4.5,0){$\alpha_4$};
		%\node (e) at (6,0){$\alpha_5$};
		%\node (f) at (7.5,0){$\alpha_6$};
		\draw (a)--(b)--(c)--(d);
		\draw (0,1.5) circle(1.06cm);
		\draw (1.5,3.0) circle(1.06cm);
		\draw (0,4.5) circle(1.06cm);
		\draw (3,4.5) circle(1.06cm);
		\draw (1.5,6.0) circle(1.06cm);
		\draw (3.0,7.5) circle(1.06cm);
		%\draw (0,9.0) circle(1.06cm);
		%\draw (3,9.0) circle(1.06cm);
		%\draw (4.5,6.0) circle(1.06cm);
		%\draw (6,7.5) circle(1.06cm);
		%\draw (7.5,9.0) circle(1.06cm);
		%\draw (3,7.5) circle(1.06cm);

		\end{tikzpicture}
	\end{center}
	\begin{center}
		Fig 2
	\end{center}
	\fi
	\begin{center}

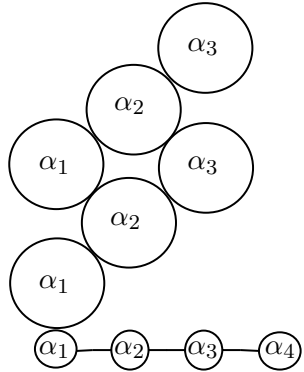
		\tikzset{every picture/.style={line width=0.75pt}} %set default line width to 0.75pt        
		
		\begin{tikzpicture}[x=0.75pt,y=0.75pt,yscale=-1,xscale=1]
		%uncomment if require: \path (0,601); %set diagram left start at 0, and has height of 601
		
		%Shape: Output [id:dp5937096909277091] 
		\draw   (131.04,432.97) .. controls (131.04,427.53) and (135.24,423.13) .. (140.42,423.13) .. controls (145.6,423.13) and (149.79,427.53) .. (149.79,432.97) .. controls (149.79,438.4) and (145.6,442.81) .. (140.42,442.81) .. controls (135.24,442.81) and (131.04,438.4) .. (131.04,432.97) -- cycle (121.67,432.97) -- (131.04,432.97) (159.17,432.97) -- (149.79,432.97) ;
		%Shape: Output [id:dp4001460320007155] 
		\draw   (168.23,432.97) .. controls (168.23,427.53) and (172.43,423.13) .. (177.61,423.13) .. controls (182.79,423.13) and (186.99,427.53) .. (186.99,432.97) .. controls (186.99,438.4) and (182.79,442.81) .. (177.61,442.81) .. controls (172.43,442.81) and (168.23,438.4) .. (168.23,432.97) -- cycle (158.86,432.97) -- (168.23,432.97) (196.36,432.97) -- (186.99,432.97) ;
		%Shape: Ellipse [id:dp6587119517853615] 
		\draw   (80,399.22) .. controls (80,386.72) and (90.63,376.59) .. (103.74,376.59) .. controls (116.85,376.59) and (127.48,386.72) .. (127.48,399.22) .. controls (127.48,411.72) and (116.85,421.85) .. (103.74,421.85) .. controls (90.63,421.85) and (80,411.72) .. (80,399.22) -- cycle ;
		%Shape: Ellipse [id:dp15098296497323993] 
		\draw   (116.16,368.7) .. controls (116.16,356.2) and (126.79,346.07) .. (139.9,346.07) .. controls (153.01,346.07) and (163.64,356.2) .. (163.64,368.7) .. controls (163.64,381.2) and (153.01,391.33) .. (139.9,391.33) .. controls (126.79,391.33) and (116.16,381.2) .. (116.16,368.7) -- cycle ;
		%Shape: Ellipse [id:dp2581879851046973] 
		\draw   (154.7,280.49) .. controls (154.7,267.99) and (165.33,257.86) .. (178.44,257.86) .. controls (191.55,257.86) and (202.18,267.99) .. (202.18,280.49) .. controls (202.18,292.99) and (191.55,303.12) .. (178.44,303.12) .. controls (165.33,303.12) and (154.7,292.99) .. (154.7,280.49) -- cycle ;
		%Shape: Ellipse [id:dp9339433366705028] 
		\draw   (79.55,339.19) .. controls (79.55,326.69) and (90.17,316.55) .. (103.29,316.55) .. controls (116.4,316.55) and (127.03,326.69) .. (127.03,339.19) .. controls (127.03,351.68) and (116.4,361.82) .. (103.29,361.82) .. controls (90.17,361.82) and (79.55,351.68) .. (79.55,339.19) -- cycle ;
		%Shape: Ellipse [id:dp13579089185274396] 
		\draw   (155.02,341.01) .. controls (155.02,328.51) and (165.65,318.38) .. (178.76,318.38) .. controls (191.87,318.38) and (202.5,328.51) .. (202.5,341.01) .. controls (202.5,353.51) and (191.87,363.64) .. (178.76,363.64) .. controls (165.65,363.64) and (155.02,353.51) .. (155.02,341.01) -- cycle ;
		%Shape: Ellipse [id:dp8905586219299884] 
		\draw   (118.55,311.63) .. controls (118.55,299.13) and (129.17,289) .. (142.29,289) .. controls (155.4,289) and (166.03,299.13) .. (166.03,311.63) .. controls (166.03,324.13) and (155.4,334.26) .. (142.29,334.26) .. controls (129.17,334.26) and (118.55,324.13) .. (118.55,311.63) -- cycle ;
		%Shape: Ellipse [id:dp4465632082010451] 
		\draw   (92.5,432.4) .. controls (92.5,427.21) and (97.2,423) .. (103,423) .. controls (108.8,423) and (113.5,427.21) .. (113.5,432.4) .. controls (113.5,437.6) and (108.8,441.81) .. (103,441.81) .. controls (97.2,441.81) and (92.5,437.6) .. (92.5,432.4) -- cycle ;
		%Shape: Ellipse [id:dp40304588829913324] 
		\draw   (205.5,433.4) .. controls (205.5,428.21) and (210.2,424) .. (216,424) .. controls (221.8,424) and (226.5,428.21) .. (226.5,433.4) .. controls (226.5,438.6) and (221.8,442.81) .. (216,442.81) .. controls (210.2,442.81) and (205.5,438.6) .. (205.5,433.4) -- cycle ;
		%Straight Lines [id:da8097486319972063] 
		\draw    (113.5,433.4) -- (121.67,432.97) ;
		%Straight Lines [id:da7070627349456258] 
		\draw    (205.5,433.4) -- (196.36,432.97) ;
		
		% Text Node
		\draw (93,427) node [anchor=north west][inner sep=0.75pt]   [align=left] {$\displaystyle \alpha _{1}$};
		% Text Node
		\draw (168,427) node [anchor=north west][inner sep=0.75pt]   [align=left] {$\displaystyle \alpha _{3}$};
		% Text Node
		\draw (131,427) node [anchor=north west][inner sep=0.75pt]   [align=left] {$\displaystyle \alpha _{2}$};
		% Text Node
		\draw (207,427) node [anchor=north west][inner sep=0.75pt]   [align=left] {$\displaystyle \alpha _{4}$};
		% Text Node
		\draw (93,394) node [anchor=north west][inner sep=0.75pt]   [align=left] {$\displaystyle \alpha _{1}$};
		% Text Node
		\draw (129,362) node [anchor=north west][inner sep=0.75pt]   [align=left] {$\displaystyle \alpha _{2}$};
		% Text Node
		\draw (131,304) node [anchor=north west][inner sep=0.75pt]   [align=left] {$\displaystyle \alpha _{2}$};
		% Text Node
		\draw (168,274) node [anchor=north west][inner sep=0.75pt]   [align=left] {$\displaystyle \alpha _{3}$};
		% Text Node
		\draw (168,334) node [anchor=north west][inner sep=0.75pt]   [align=left] {$\displaystyle \alpha _{3}$};
		% Text Node
		\draw (93,334) node [anchor=north west][inner sep=0.75pt]   [align=left] {$\displaystyle \alpha _{1}$};

		\end{tikzpicture}
		
	\end{center}
	
	This is an example of a super Lyndon heap $E=123123$ with $F=123$ is a Lyndon heap in $\mathcal{H}_1(I,\zeta)$.
\end{example}

Let $ E$ be a super Lyndon heap in $\mathcal H(I,\zeta)$. 
Assume that $E =  F \circ F$ with $ F$ is a Lyndon heap in $\mathcal H_1(I,\zeta)$. We define the standard factorization of $ E$ to be $\Sigma(E) = (F, F)$.  To each super Lyndon heap $ E \in \mathcal H(I,\zeta)$ we associate a Lie word $\Lambda(E)$ in $\mathcal{LS}(G,\Psi)$ as follows. If $E\in I$, then $\Lambda( E)=E$ and otherwise $\Lambda(E)=[\Lambda( F_1),\Lambda(F_2)]$, where $ E= F_1 \circ F_2$ is the standard factorization of $ E$. Given these notions, we have our following theorem which gives a basis for free partially commutative Lie superalgebra $\mathcal{LS}(G,\Psi)$.
\subsection{Super Lyndon heaps basis of free partially commutative Lie superalgebras}
The following theorem is the main result of this subsection in which we construct the Lyndon heaps basis for free partially commutative Lie superalgebras.
\begin{thm}\label{lafps}
	The set $\{\Lambda(E):  E\in \mathcal H(I,\zeta) \text{ is super Lyndon}\}$ forms a basis of $\mathcal {LS}(G,\Psi)$.
\end{thm}
The rest of the section is dedicated to the proof of the above theorem. The proof of the following lemma is immediate.

\begin{lem}Let $E \in \mathcal{SLH}_{\bold k}(I,\zeta)$. Then $\Lambda(E)= \sum\limits_{F \in \mathcal{SLH}_{\bold k}(I,\zeta)} \alpha_F F $ where $\alpha_F \in \mathbb{Z}$.  Since there are finite number of heaps of degree $\bold k$, the sum is a finite sum.
\end{lem}

\begin{prop}\label{ubasis}
	The set $\mathcal H(I,\zeta)$ indexes a basis for the universal enveloping algebra of the free partially commutative Lie superalgebra $\mathcal{LS}(G,\Psi)$.
\end{prop}
\begin{pf}
	Let $\lie U$ be the $\mathbb C$-span of the heaps monoid $\mathcal H(I,\zeta)$ associated with the supergraph $(G,\Psi)$. Then $\lie U$ has an algebra structure induced from the multiplication in $\mathcal H(I,\zeta)$. This is the free partially commutative superalgebra associated with the supergraph $(G,\Psi)$. This is the smallest associative superalgebra containing $\mathcal {LS}(G)$. Therefore $\lie U$ is the universal enveloping algebra of the Lie superalgebra $\mathcal{LS}(G,\Psi)$. 
\end{pf}
%\begin{rem}
%	We observe that, from the definition of $\lie U$, the algebra $\lie U$ has a basis consists of the elements of $\mathcal H(I,\zeta)$.
%\end{rem}

\begin{prop}\label{shush5}
	Let $L$ be a super Lyndon heap of weight $\bold k$ over the supergraph $(G,\Psi)$. Put $\Lambda(L)= \sum\limits_{E \in \mathcal{SLH}_{\bold k}(I,\zeta)} \alpha_E E $. Then
	\begin{itemize}
		\item[(i)] $\alpha_L =1$ if $L$ is a Lyndon heap
		\item[(ii)] $\alpha_L =2$ if $L= L_1 \circ L_1$, $L_1$ is Lyndon heap in $\mathcal{H}_1(I,\zeta)$
		\item[(iii)] If $\alpha_E \neq 0$ then $E \ge L$. 
	\end{itemize}
\end{prop} 
\begin{proof}
	If $E$ is a Lyndon heap, then the proofs of (i) and (iii) are given in \cite[Theorem 4.2]{la93}. So we prove (ii) and (iii) when $L$ is a super Lyndon heap. Let $L= L_1 \circ L_1$ where $L_1$ is Lyndon heap in $\mathcal{H}_1(I,\zeta)$. Now,
	\begin{align*}
	\Lambda(L)&= [\Lambda(L_1), \Lambda(L_1)]\\
	&=[\sum\limits_{E \ge L_1}\alpha_E E, \sum\limits_{E' \ge L_1} \alpha_{E'} E']  \text{ (Using part(i) and (iii)) for Lyndon heaps)}\\
	& = [L_1,L_1]+ \sum\limits_{\substack{{E >L_1}\\{E'>L_1}}} \alpha_E \alpha_{E'}[E, E']+\sum\limits_{E >L_1} \alpha_E[E,L_1]+ \sum\limits_{E'>L_1}\alpha_{E'}[L_1,E']\\
	&=2L+ \sum_{K>L} \alpha_K K.
	\end{align*}
	This proves (ii). Now, $st( E \circ E^{'}) \ge st(E) \cdot st(E^{'}) >st(L_1) \cdot st(L_1) =: st(L_1 \circ L_1) \Rightarrow E\circ E^{'} >L $. Similarly, we have $E^{'} \circ E > L.$ Also, $E \circ L_1 >L, L_1 \circ E>L.$ Hence (iii) follows. 
\end{proof}

\begin{cor} The set $\mathcal{B}= \{\Lambda(L): L$ is a super Lyndon heap\} is linearly independent in $\mathcal{LS}(G,\Psi)$.
\end{cor} 
\begin{proof}
	Assume that $$\sum\limits_{L \in \mathcal{B}} \beta_L \Lambda(L) =0, \,\,\,\,\, \beta_L \in \mathbb{C} $$ where all but finitely many $\beta_L$ are zero. Then by the above proposition, we have 
	$$\sum\limits_{L \in \mathcal{B}} \beta_L \left( \sum\limits_{\substack{E \ge L \\ \wt(E) = \wt(L)}} \alpha_E E\right) =0$$
	$$\Rightarrow \sum\limits_{L \in \mathcal{B}} \beta_L \left( \alpha_L L+ \sum\limits_{\substack{E > L \\ \wt(E) = \wt(L)}} \alpha_E E\right) =0 \text{ where } 
	\alpha_L = \begin{cases}   1 & \text{ if }  L \in \mathcal{LH}(I,\zeta) \\  2 & \text{ if }  L=E\circ E, E \in \mathcal H_1(I,\zeta) 
	\end{cases}$$
	$$\Rightarrow \sum\limits_{L \in\mathcal{LH}(I,\zeta)} \beta_L L+ 2\sum\limits_{\substack{L=E\circ E \\ E \in \mathcal{LH}_1(I,\zeta)}} \beta_L L+ \sum\limits_{L \in \mathcal B} \sum\limits_{\substack{E > L \\ \wt(E) = \wt(L)}}\beta_L  \alpha_E E =0 $$ Taking modulo in the grade space $\mathcal{LS}_{\bold k}(G,\Psi)$ we get $$\begin{cases}
	\sum\limits_{\substack{{L \in \mathcal{LH}_{\bold k}(I,\zeta)}}}  \beta_L L + \sum\limits_{\substack{{E>L}\\ \wt(E) = \wt(L) 
			\\ {E \in \mathcal{SLH}_{\bold k}(I,\zeta)}}} \beta_L \alpha_E E=0 \text{ if $k_i$ is odd for some }  i \in \supp(\bold k) \\
	\sum\limits_{\substack{{L \in \mathcal{LH}_{\bold k}(I,\zeta)}}}  \beta_L L + 2 \sum\limits_{\substack{L \in \mathcal{LH}_{\bold k}(I,\zeta) \\ L = E \circ E \\ E \in \mathcal{LH}_1(I,\zeta)}}  \beta_L L +  \sum\limits_{\substack{{\substack{E > L \\ \wt(E) = \wt(L)}}\\{L \in \mathcal{SLH}_{\bold k}(I,\zeta)}}} \beta_L \alpha_E E=0 \text{  otherwise}.
	\end{cases}$$
	%	$$\begin{cases}
	%\sum\limits_{\substack{{L \in \mathcal{LH}_{\bold k}(I,\zeta)}}}  \beta_L L + 2 \sum\limits_{\substack{L \in \mathcal{LH}_{\bold k}(I,\zeta) \\ L = E \circ E \\ E \in \mathcal{LH}_1(I,\zeta)}}  \beta_L L +  \sum\limits_{\substack{{E>L}\\{E \in \mathcal{LH}_{\bold k}(I,\zeta)}}} \beta_L \alpha_E E=0 \text{ if $\bold k$ is odd} \\
	%\sum\limits_{\substack{L = E \circ E \\ {E \in \mathcal{LH}_{\frac{\bold k}{2}}(I,\zeta)}}}  2 \beta_L L + 2 \sum\limits_{\substack{L \in \mathcal{LH}_{\bold k}(I,\zeta) \\ L = E \circ E \\ E \in \mathcal{LH}_1(I,\zeta)}}  \beta_L L + \sum\limits_{\substack{{\substack{E > L \\ \wt(E) = \wt(L)}}\\{L \in \mathcal B_{\bold k}}}} \beta_L \alpha_E E=0 \text{ if $\bold k$ is even}.
	%\end{cases}$$
	%\end{center}
	which are  finite linear combinations of heaps of weight $\bold k$ in the universal enveloping algebra $\lie U$ of $\mathcal{LS}(G,\Psi)$. Since heaps form a basis of $\lie U = \mathbb{C}(\mathcal{H}(I,\zeta))$ we get all the $ \beta_L =0$ in the above equations. Since $\bold k$ is arbitrary we get $ \beta_L =0$ for all $L \in \mathcal{B} $. This completes the proof.
\end{proof}
\begin{prop}\label{shush7}
	Let $ L$ and $M$ be super Lyndon heaps such that $L <M$. Then we can write
	$[\Lambda(L), \Lambda(M)]= \sum\limits_{\substack{{N \in \mathcal{SLH}(I,\zeta)}\\N<M \\ {deg(N)= deg(L\circ M)}}} \alpha_N \Lambda(N)$.
\end{prop}
\begin{proof}
	The proof is by case by case analysis. 
	
	Case (i):- Suppose $L, M $ are Lyndon heaps satisfying $L<M$  then result follows from \cite[Theorem 4.4]{la93}. 
	
	Case (ii):-	
	Suppose exactly one of $L, M$ is a super Lyndon heap. Without loss of generality assume that $L= L_1  \circ L_1$ where $L_1$ is Lyndon heap in ${\mathcal{H}_1(I,\zeta)}$ and $M$ is an arbitrary Lyndon heap. Now,
	
\begin{math}
$\begin{align} [\Lambda(L),\Lambda(M)] &=[[\Lambda(L_1),\Lambda(L_1)],\Lambda(M)]\\ &= 2[\Lambda(L_1),[\Lambda(L_1),\Lambda(M)]]  \\ &= 2[\Lambda(L_1), \sum\limits_{\substack{{N_1 \in \mathcal{LH}(I,\zeta)}\\{N_1 < M}\\{deg(N_1)= deg(L_1 \circ M)}}} \alpha_{N_1} \Lambda(N_1)] \,\,\,(\because L_1 < L_1 \circ L_1 =L < M)\\  &= 2 \sum\limits_{\substack{{N_1 \in \mathcal{LH}(I,\zeta)}\\{N_1 < M}\\{deg(N_1)= deg(L_1 \circ M)}}} \alpha_{N_1} [\Lambda(L_1), \Lambda(N_1)] \\ &= 2 \left(\sum\limits_{\substack{{N_1 \in \mathcal{LH}(I,\zeta)}\\{L_1<N_1 < M}\\{deg(N_1)= deg(L_1 \circ M)}}} \alpha_{N_1} [\Lambda(L_1), \Lambda(N_1)]+ \sum\limits_{\substack{{N_1 \in \mathcal{LH}(I,\zeta)}\\{N_1 < L_1}\\{deg(N_1)= deg(L_1 \circ M)}}} \alpha_{N_1} [\Lambda(L_1), \Lambda(N_1)]\right)
\end{align}$
\end{math}

%	For those $N_1 \in \mathcal{LH}(I,\zeta)$ s.t. $ L_1 < N_1$
	Using case (i) in the first term of the above equation we get
	\begin{align*}
	\sum\limits_{\substack{{N_1 \in \mathcal{LH}(I,\zeta)}\\{L_1 < N_1 < M}\\{deg(N_1)= deg(L_1 \circ M)}}} \alpha_{N_1} [\Lambda(L_1), \Lambda(N_1)] &=  \sum\limits_{\substack{{N_1 \in \mathcal{LH}(I,\zeta)}\\{N_1 < M}\\{deg(N_1)= deg(L_1
				\circ M)}}} \alpha_{N_1} \left(  \sum\limits_{\substack{{N_2 \in \mathcal{LH}(I,\zeta)}\\{ N_2<N_1 < M}\\{ deg(N_2)=deg(L_1\circ N_1)=deg(L\circ M)}}} \alpha_{N_2}\Lambda(N_2)\right) \\
	&=\sum\limits_{\substack{{N_2 \in \mathcal{LH}(I,\zeta)}\\{ N_2< M}\\{ deg(N_2)=deg(LoM)}}} \underbrace{\left( \sum\limits_{\substack{{N' \in \mathcal{LH}(I,\zeta)}\\{N_2< N' < M}\\{deg(N')= deg(L_1\circ M)}}} \alpha_{N'}\right)}_{\text{some constant } c_{N_2}} \alpha_{N_2}\Lambda(N_2)\\
	&= \sum\limits_{\substack{{N_2 \in \mathcal{LH}(I,\zeta)}\\{ N_2 < M}\\{ deg(N_2)=deg(L\circ M)}}} (c_{N_2}\alpha_{N_2})\Lambda(N_2)
	%&= \sum\limits_{\substack{{N_2 \in \mathcal{LH}(I,\zeta)}\\{ N_2 < M}\\{ deg(N_2)=deg(L\circ M)}}} \alpha_{N_2}'\Lambda(N_2)
	\end{align*}
		%For those $N_1 \in \mathcal{LH}(I,\zeta)$ s.t. $ L_1 > N_1$ 
	For the second summation, $[\Lambda(L_1), \Lambda(N_1)]= -(-1)^{a_{N_1}b_{L_1}} [\Lambda(N_1), \Lambda(L_1)]$ where $ a_{N_1}, b_{L_1} \in\{0,1\}$  according to $N_1, L_1 \in \mathcal{H}_i(I,\zeta)$ for $i \in \{0,1\}$. Therefore,

	$[\Lambda(L_1), \Lambda(N_1)] = -(-1)^{a_{N_1}b_{L_1}}\sum\limits_{\substack{{K \in \mathcal{LH}(I,\zeta)}\\{ K<L_1 < M}\\{ deg(K)=deg(N_1\circ L_1)=deg(L \circ M)}}} \alpha_K\Lambda(K)$
	
	$=	\sum\limits_{\substack{{N_1 \in \mathcal{LH}(I,\zeta)}\\{N_1 < M}\\{deg(N_1)= deg(L_1 \circ M)}\\{L_1>N_1}}} \alpha_{N_1} [\Lambda(L_1), \Lambda(N_1)] $

$= \sum\limits_{\substack{{N_1 \in \mathcal{LH}(I,\zeta)}\\{N_1 < M}\\{deg(N_1)= deg(L_1\circ M)}}} \alpha_{N_1} [\Lambda(L_1), \Lambda(N_1)]$

$=-\sum\limits_{\substack{{N_1 \in \mathcal{LH}(I,\zeta)}\\{N_1 < M}\\{deg(N_1)= deg(L_1oM)}}} \alpha_{N_1} \left((-1)^{a_{N_1}b_{L_1}} \sum\limits_{\substack{{K \in \mathcal{LH}(I,\zeta)}\\{ K<L_1 < M}\\{ deg(K)=deg(N_1 \circ L_1)=deg(L \circ M)}}} \alpha_K\Lambda(K) \right)$

	$=- \sum\limits_{\substack{{K \in \mathcal{LH}(I,\zeta)}\\{ K< M}\\{ deg(K)=deg(L \circ M)}}} \underbrace{\left(  \sum\limits_{\substack{{N' \in \mathcal{LH}(I,\zeta)}\\{K< N' < M}\\{deg(N')= deg(L_1 \circ M)}}}(-1)^{a_{N'}b_{L_1}} \alpha_{N'}\right)}_{\text{constant }c_{K}} \alpha_K\Lambda(K)$
	
	$=\sum\limits_{\substack{{K \in \mathcal{LH}(I,\zeta)}\\{ K< M}\\{ deg(K)=deg(L \circ M)}}} \alpha_K'\Lambda(K) \,\, \text{ where } \alpha_K'= -c_{K}\alpha_K$

	$ \Rightarrow [\Lambda(L),\Lambda(M)]=	2\left(\sum\limits_{\substack{{N_2 \in \mathcal{LH}(I,\zeta)}\\{ N_2 < M}\\{ deg(N_2)=deg(L\circ M)}}} (c_{N_2}\alpha_{N_2})\Lambda(N_2)+\sum\limits_{\substack{{K \in \mathcal{LH}(I,\zeta)}\\{ K< M}\\{ deg(K)=deg(L \circ M)}}} \alpha_K'\Lambda(K)\right)$
	
		Case (iii):- Suppose $L= L_1 \circ L_1,\, M= M_1 \circ M_1 $ where $L_1, M_1$ are Lyndon heaps in ${\mathcal{H}_1(I,\zeta)}$ satisfying $L<M$. Then 
			\begin{align*}
	[\Lambda(L),\Lambda(M)] &= [[\Lambda(L_1),\Lambda(L_1)],\Lambda(M)]\\ &=2[\Lambda(L_1),[\Lambda(L_1), \Lambda(M)]]\\&= 2[\Lambda(L_1), \sum\limits_{\substack{{N \in \mathcal{LH}(I,\zeta)}\\{N < M}\\{deg(N)= deg(L_1 \circ M)}}} \alpha_{N}\Lambda(N)] \,\, (\text{by the previous case})\\ &= 2\sum\limits_{\substack{{N \in \mathcal{LH}(I,\zeta)}\\{N < M}\\{deg(N)= deg(L_1 \circ M)}}} \alpha_{N}[\Lambda(L_1), \Lambda(N)]
		\end{align*}
	\begin{align*}
	 &=2 \left(\sum\limits_{\substack{{N \in \mathcal{LH}(I,\zeta)}\\{L_1<N < M}\\{deg(N)= deg(L_1 \circ M)}}} \alpha_{N}[\Lambda(L_1), \Lambda(N)]+\sum\limits_{\substack{{N \in \mathcal{LH}(I,\zeta)}\\{N < M \text{ and } L_1>N}\\{deg(N)= deg(L_1 \circ M)}}} \alpha_{N}[\Lambda(L_1), \Lambda(N)]\right)
	\end{align*}

	For those $N \in \mathcal{LH}(I,\zeta)$ such that $L_1 <N$ then by first case$$[\Lambda(L_1), \Lambda(N)]=\sum\limits_{\substack{{K \in \mathcal{LH}(I,\zeta)}\\{K < N<M}\\{deg(K)= deg(L_1\circ N)}}} \beta_{K}\Lambda(K)$$
	\begin{align*}
	\Rightarrow \sum\limits_{\substack{{N \in \mathcal{LH}(I,\zeta)}\\{N < M}\\{deg(N)= deg(L_1 \circ M)}\\{L_1<N}}} \alpha_{N}[\Lambda(L_1), \Lambda(N)] &=\sum\limits_{\substack{{N \in \mathcal{LH}(I,\zeta)}\\{N < M}\\{deg(N)= deg(L_1\circ M)}}} \alpha_{N} \left(  \sum\limits_{\substack{{K \in \mathcal{LH}(I,\zeta)}\\{ K< M}\\{ deg(K)=deg(L \circ M)}}} \beta_K\Lambda(K)\right) \\&= \sum\limits_{\substack{{K \in \mathcal{LH}(I,\zeta)}\\{ K< M}\\{ deg(K)=deg(L\circ M)}}} \alpha_K'\Lambda(K)
	\end{align*}

	For those $N \in \mathcal{LH}(I,\zeta)$ such that $L_1 >N$ then 
	\begin{align*}
	[\Lambda(L_1), \Lambda(N)] &=-[\Lambda(N), \Lambda(L_1)]\\ &= - \sum\limits_{\substack{{N \in \mathcal{LH}(I,\zeta)}\\{K_2 < L_1}\\{deg(K_2)= deg(N\circ L_1)}}} \beta_{K_2}\Lambda(K_2)
	\end{align*}
	\begin{align*}
	\Rightarrow \sum\limits_{\substack{{N \in \mathcal{LH}(I,\zeta)}\\{N < M}\\{deg(N)= deg(L_1 \circ M)}\\{L_1>N}}} \alpha_{N}[\Lambda(L_1), \Lambda(N)] &= \sum\limits_{\substack{{N \in \mathcal{LH}(I,\zeta)}\\{N < M}\\{deg(N)= deg(L_1\circ M)}}} \alpha_{N} \left(  \sum\limits_{\substack{{K_2 \in \mathcal{LH}(I,\zeta)}\\{ K<L_1< M}\\{ deg(K_2)=deg(L\circ M)}}} \beta_{K_2}\Lambda(K_2)\right) \\&= \sum\limits_{\substack{{K_2 \in \mathcal{LH}(I,\zeta)}\\{ K_2< M}\\{ deg(K_2)=deg(L\circ M)}}} \alpha_{K_2}'\Lambda(K_2)
	\end{align*}
	$\Rightarrow [\Lambda(L),\Lambda(M)]=2\left(\sum\limits_{\substack{{K \in \mathcal{LH}(I,\zeta)}\\{ K< M}\\{ deg(K)=deg(L\circ M)}}} \alpha_K'\Lambda(K)+\sum\limits_{\substack{{K_2 \in \mathcal{LH}(I,\zeta)}\\{ K_2< M}\\{ deg(K_2)=deg(L\circ M)}}} \alpha_{K_2}'\Lambda(K_2)\right)$

	This completes the proof.
	
\end{proof}

By the above proposition, the Lie subsuperalgebra generated by $\mathcal{B}=\{ \Lambda(L): L$ is super Lyndon heap\} in $\mathcal{LS}(G,\Psi)$ contains $I$. So this subalgebra is equal to $\mathcal{LS}(G,\Psi)$. This completes the proof of Theorem \ref{lafps} and in turn, gives the Lyndon basis for the free roots spaces of BKM superalgebra $\lie g$ whose associated supergraph is $(G,\Psi)$ [c.f Lemma \ref{identification lem}].

\begin{example}\label{basis1_ex1} Consider the BKM superalgebra $\lie g$ associated with the BKM supermatrix $$A=\begin{bmatrix}
	2 & -1 & 0&0&0&0\\
	-1 & -3 &-4&-1&0&0\\
	0&-4&-4& 0&0&-1\\
	0&-1&0&2&-1&0\\
	0&0&0&-1&-2&0\\
	0&0&-1&0&0&-3\\ 
	\end{bmatrix}.$$ The quasi-Dynkin diagram $G$ of $\lie g$ is as follows: 	
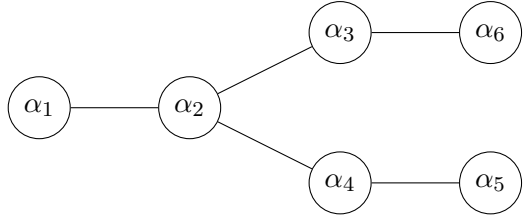
\begin{center}
		
		\begin{tikzpicture}
		\tikzstyle{every node}=[draw, shape=circle];
		\node (a) at (0,0){$\alpha_1$};
		\node (b) at (2,0){$\alpha_2$};
		\node (c) at (4,1){$\alpha_3$};
		\node (d) at (4,-1){$\alpha_4$};
		\node (e) at (6,1){$\alpha_6$};
		\node (f) at (6,-1){$\alpha_5$};
		\draw (a)--(b)--(c)--(e);
		\draw (b)--(d)--(f);
		\end{tikzpicture}
	\end{center}
	
	We have $I=\{ 1,2,3,4,5,6\}, \Psi=\{3,5\}$ and $I^{re} =\{1,4\}$. Assume the natural total order on $I$. Let $\bold k = (0,0,3,0,0,3) \in \mathbb Z_+[I]$. Then $\eta(\bold k)= 3\alpha_3+3\alpha_6 \in \Delta_+^1$.
	
	Fix $i=3$ (minimal element in the support of $\bold k$), then the super Lyndon heaps of weight $\eta(\bold k)$ 
	%in $\{E \in \mathcal{A}^\ast_3(I,\zeta): \wt(E)=\eta(\bold k), E \text{ is super Lyndon heap}  \}$ 
	are $\{336636,333666,336366\}$ with standard factorization $3366|36, 3|33666$ and $3|36366$. The associated Lie monomials

 $$\{[[3,[[3,6],6]],[3,6]], [3,[3,[[[3,6],6],6]]],   [3,[[3,6],[[3,6],6]]]\}$$ spans $\mathfrak{g}_{\eta(\bold k)}$. We have $\mult(\eta(\bold k))=3$ [c.f. Example \ref{ch_ex1}]. So these Lie monomials form a basis for $\mathfrak{g}_{\eta(\bold k)}.$ 
\end{example}

\begin{example}\label{basis1_ex2}Consider the BKM superalgebra $\lie g$ associated with the BKM supermatrix $$A=\begin{bmatrix}
	2 & -1 & 0&0&0&0\\
	-1 & -3 &-1&0&0&0\\
	0&-2&-4& -1&0&0\\
	0&0&-1&2&-1&0\\
	0&0&0&-1&-2&-1\\
	0&0&0&0&-1&-3\\ 
	\end{bmatrix}.$$ The quasi-Dynkin diagram $G$ of $\lie g$ is as follows: 	
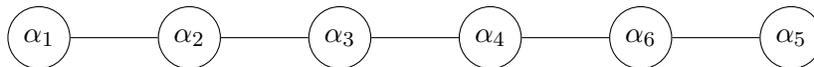
\begin{center}
		
		\begin{tikzpicture}
		\tikzstyle{every node}=[draw, shape=circle];
		\node (a) at (0,0){$\alpha_1$};
		\node (b) at (2,0){$\alpha_2$};
		\node (c) at (4,0){$\alpha_3$};
		\node (d) at (6,0){$\alpha_4$};
		\node (e) at (8,0){$\alpha_6$};
		\node (f) at (10,0){$\alpha_5$};
		%\draw (a)--(b);
		\draw (a)--(b)--(c)--(d)--(e)--(f);
		\end{tikzpicture}
	\end{center}
	We have $I=\{ 1,2,3,4,5,6\},  \, \Psi=\{3,5\}, I^{re} =\{1,4\}, \eta(\bold k)= 2\alpha_3+\alpha_4+2 \alpha_5+\alpha_6 $. Assume the natural total order on $I$. Fix $i=3$, observe that the Lyndon heaps %$\mathcal{A}_3(I,\zeta)=\{3,34,345,3456,344,34545,\cdots\}$. Only
	%Lyndon heaps on $\mathcal{A}^\ast_3(I,\zeta)$ 
	of weight $\eta(\bold k)$ are $\{334556,334565\}$ with standard factorizations $3|34565,3|34556$ respectively. The associated Lie monomials are $$\{[3,[3,[4,[[5,6],5]]]], [3,[3,[4,[5,[5,6]]]]]\}$$ which form a spanning set of $\mathfrak{g}_{\eta(\bold k)}$. Since $\mult(\eta(\bold k))=2$ [c.f. Example \ref{multf0}]. These lie monomials form a basis of $\mathfrak{g}_{\eta(\bold k)}$.
	
\end{example}
\hfill
\section{Main result II: LLN basis of BKM superalgebras}\label{section 5}%\label{mainthms}

Let $\lie g$ be the BKM Lie superalgebra associated with a Borcherds-Cartan matrix $(A,\Psi)$.  Let $(G,\Psi)$ be the quasi-Dynkin diagram of $\lie g$ with the vertex set $I$ [c.f. Definition \ref{qdd}]. In this section, we extend Theorem \ref{mainthmbb} to the case of free root spaces of $\lie g$.
%In case of Borcherds algebras, this assumption is same as the assumptions made in Theorems \ref{mainthmchb} and \ref{mainthmbb}.%\cite[Sections 3 and 4]{akv17}.

In what follows in this section, we use the super-Jacobi identity (up to sign) to prove our results. Also, whenever we fix an $i \in I$, it is assumed that $i$ is the least element in the total order of $I$.

\subsection{Initial alphabet and Left normed Lie word associated with a word}
Let $\bold w\in M(I, G)$ [c.f. Section \ref{fpcsm}] and $\bold w=a_1\cdots a_r$ be an element in the class $\bold w$. We define the length of the word $|\bold w|=r$. We define, $i(\bold w) = |\{j: a_j = i\}|$ and
% $i^1(\bold w)=|\{j : i_j=i, i_j \ne i_{j+1} \ne i_{j-1}\}|$, $i^2(\bold w)=|\{j : i_j=i_{j+1}=i\}|$, $i(\bold w)=i^1(\bold w)+i^2(\bold w)$ for all $i\in I_1$ and $i(\bold w)= |\{j : i_j=i\}| \text{ for } i \in I_0$.
$\supp (\bold w)=\{i\in I :  i(\bold w)\neq 0\}$. We define the weight of $\bold w$ to be
$\wt (\bold w)=\sum\limits_{i\in I} i(\bold w) \alpha_i.$
%\textcolor{blue}{ Let $\bold w \in M(I,G)$ and write $w=i_1\cdots i_r,   \text{ Fix } i \in I^{im}\backslash \Psi_0= \{j \in I_1: a_{ii} \neq 0\}\\
% i^1(\bold w)=|\{j : i_j=i, i_j \ne i_{j+1} \text{ and } i_j \ne i_{j-1}\}| \\
% i^2(\bold w)=|\{j :i_{j-1} \neq i_j=i_{j+1}=i \neq i_{j+2}\}|\\
%  i^3(\bold w)=|\{j :i_{j-1} \neq i_j=i_{j+1}= i_{j+2} = i \neq i_{j+3}\}|\\
%   i^k(\bold w)=|\{j:i_{j-1} \neq i_j=i_{j+1}=\cdots =i_{j+k-1}=i   \neq i_{j+k}\}|\\
% \text{ for } 1 \le k \le r \text{ define } i(\bold w)= \sum\limits_{p =1}^{r} i^p(\bold w) $. Thus for all $i \in I_0 \cup  \Psi_0, i(\bold w)=|\{j: i_j=i \}|$ .
% \[ wt(\bold w)= \sum\limits_{i \in \Psi \backslash \Psi_0}\left( \sum\limits_{p=1}^{r} pi^p(\bold w)\right) \alpha_i +\sum\limits_{i \in I_0 \cup \Psi_0} i(\bold w)\alpha_i\] 
%}
%Let $\bold w=a_1\cdots a_r \in M(I,G)$.   %\text{ Fix } i \in I^{im}\backslash \Psi_0= \{j \in I_1: a_{ii} \neq 0\}$
%The initial alphabet (\textcolor{red}{This definition needs to be changed as in the thesis}) of $\bold w$ is a multiset denoted by $\rm{IA_m}(\bold w)$ and defined by $i\in \rm{IA_m}(\bold w)$ (counted with multiplicities) if and only if $\exists \ \bold u\in M(I, G)$ satisfying one of the following two conditions \begin{enumerate}
%	\item $\bold w=\bold u i$ if $i \in I$,
%	\item $\bold w=\bold u ii$ if $i \in I_1$.
%\end{enumerate}
The following definition of the initial alphabet is different from the one given in \cite[Section 4]{akv17}. Our definition is compatible with the definition of pyramids given in Section \ref{pyramidsec} and pyramids  will be the main tool in this section. For $i \in I$, its initial multiplicity in $\bold w$ is defined to be the largest $k \geq 0$ for which there exists  $\bold u\in M(I, G)$ such that $\bold w= i^k \bold u $.  We define the {\em initial alphabet} $\rm{IA_m}(\bold w)$ of $\bold w$ to be the multiset in which each $i \in I$ occurs as many times as its initial multiplicity in $\bold w$. The underlying set is denoted by $\rm{IA}(\bold w)$. The left normed Lie word associated with $\bold w$ is defined by 
%$$e(\bold w)=[e_{i_1},[e_{i_{2}},\cdots [e_{i_{r-2}},[e_{i_{r-1}}, e_{i_r}]]\cdots]]]\in \lie g.$$
%$$e(\bold w)=[[[\cdots [[e_{i_{r}}, e_{i_{r-1}}],e_{i_{r-2}}]\cdots e_{i_2}]e_{i_1}]\in \lie g.$$
\begin{equation}\label{e(w)}
	e(\bold w)=[[\cdots [[a_1,a_2],a_3]\cdots,a_{r-1}]a_r] \in \lie g.
\end{equation}
Using the Jacobi identity, it is easy to see that the association $\bold w\mapsto e(\bold w)$ is well-defined and preserves the $\mathbb Z_2-$grading.
\subsection{Lyndon words and their Standard factorization}
%Next, we recall the definition and some properties of Lyndon words and state the main theorem of this section;
For a fixed $i \in I$ (which is assumed to be minimal in the total order on $I$), consider the set $$\mathcal{X}_{i}=\{\bold w\in M(I, G): \rm{IA}_m(\bold w)=\{i\} \text{ and } i(\bold w) = 1\}.$$ Observe that $\mathcal{X}_i$ (and hence $\mathcal{X}_i^*$) is $\mathbb Z_2-$graded and also totally ordered using \eqref{totor}. We denote by $FLS(\mathcal{X}_i)$ the free Lie superalgebra generated by $\mathcal{X}_{i} = \mathcal{X}_{i,0} \sqcup\mathcal{X}_{i,1}$ where $ \mathcal{X}_{i,0}$ (resp. $\mathcal X_{i,1}$) is the set of even (resp. odd) elements in $\mathcal X_i$. A non--empty word $\bold w\in \mathcal{X}_i^{*}$ is called a Lyndon word if it satisfies one of the following equivalent definitions:
\begin{itemize}
	\item $\bold w$ is strictly smaller than any of its proper cyclic rotations.
	\item $\bold w \in \mathcal{X}_i$ or $\bold w=\bold u \bold v$ for Lyndon words $\bold u$ and $\bold v$ with $\bold u<\bold v$.
\end{itemize}
There may be more than one choice of $\bold u$ and $\bold v$ with $\bold w=\bold u \bold v$ and $\bold u<\bold v$ but if $\bold v$ is of maximal possible length we call it the standard factorization. Equivalently, we can define $\bold v$ to be the minimal Lyndon word in the lexicographic order satisfying $\bold w =\bold u \bold v$. This is called the standard factorization of $\bold w$ and denoted by $\sigma(\bold w) = (\bold u,\bold v)$. Note that, when $G$ is a complete graph, the heap monoid $\mathcal H(I,\zeta)$ is isomorphic to the free monoid $I^*$. Further, the standard factorization $\Sigma(E)$ of a Lyndon heap $E \in \mathcal H(I,\zeta)$ coincides with the factorization $\sigma(E)$.
\subsection{Super Lyndon words and their associated Lie word}Next, we recall the definition and some properties of super Lyndon words and state the main theorem of this section, we define a word $\bold w\in \mathcal{X}_i^{*}$ to be super Lyndon if $\bold w$ satisfies one of the following conditions \cite{Chi06}:
\begin{itemize}
	\item $\bold w$ is a Lyndon word.
	\item $\bold w= \bold u \bold u$ where $\bold u \in \mathcal X_{i,1}^{*}$ is Lyndon. In this case, we define $\bold w=\bold u \bold u$ is the standard factorization of $\bold w$.
\end{itemize}
In \cite{mi85}, super Lyndon words are known as $s-$regular words. We will use Lyndon words (resp. super Lyndon words)  to construct a basis for the Borcherds algebras (resp. BKM Lie superalgebras).  To each super Lyndon word $\bold w\in \mathcal{X}_{i}^*$ we associate a Lie word $L(\bold w)$ in $FLS(\mathcal{X}_i)$ as follows. If $\bold w\in \mathcal{X}_i$, then $L(\bold w)=\bold w$ and otherwise $L(\bold w)=[L(\bold u),L(\bold v)]$ where $\bold w=\bold u \bold v$ is the standard factorization of $\bold w$. If $\mathcal X_{i,1}$ is empty then the map $L$ assigns Lie monomials in free Lie algebras to Lyndon words. For more details about Lyndon words and super Lyndon words, we refer the readers to \cite{Chisuper06,mi85,Re93}. 
The following result can be seen in \cite{Chisuper06,mi85} and the basis constructed is known as the Lyndon basis for free Lie superalgebras.
\begin{prop}\label{reuts}
	The set $\{L(\bold w): \bold w\in \mathcal{X}_i^* \text{ is super Lyndon}\}$ forms a basis of $FLS(\mathcal{X}_i)$.%	\hfill\qed
\end{prop}
\begin{cor}\label{correuts}
	If the set $\mathcal X_{i,1}$ is empty then $FLS(\mathcal X_i)$ becomes the free Lie algebra $FL(\mathcal X_i)$. In this case, $\{L(\bold w): \bold w\in \mathcal{X}_i^* \text{ is  Lyndon}\}$ forms a basis of $FL(\mathcal{X}_i)$.
\end{cor}

%Let $\bar{\mathcal X_i} =  \{\bold w : \bold w \in \mathcal X_i \text{ or } \bold w = \bold u \bold u, \bold u \in \mathcal{X}_{i,1} \}$. If $\bold w = \bold u \bold u, \bold u \in \mathcal X_{i,1}$ then we define $e(\bold w) = [e(\bold u),e(\bold u)]$. Let $\lie g^{i}$ be the Lie supersubalgebra of $\lie g$ generated by $\{e(\bold w): \bold w\in \bar{\mathcal{X}_{i}}\}$. By the universal property of $FLS(\mathcal{X}_i)$ we have a surjective homomorphism

Universal property of the free Lie superalgebra $FLS(\mathcal X_i)$: Let $\lie l$ be a Lie superalgebra and let $\Phi: \mathcal X_i \to \lie l$ be a set map that preserves the $\mathbb Z_2$ grading. Then $\Phi$ can be extended to a Lie superalgebra homomorphism $\Phi: FLS(\mathcal X_i) \to \lie l$. 
\subsection{Idea of the proof}\label{motivation}
Let $\lie g$ be a BKM superalgebra with the associated supergraph $(G,\Psi)$. Define a map $\Phi : \mathcal X_i \to \lie g$ by $\Phi(\bold w) = e(\bold w)$, the left normed Lie word associated with $\bold w$. The map $\Phi$ preserves the $\mathbb Z_2$ grading. By the universal property, we have a Lie superalgebra homomorphism
\begin{equation}\label{homsursp}\Phi: FLS(\mathcal{X}_{i})\to \lie g,\ \ \bold w\mapsto e(\bold w)\ \ \forall\ 
\bold w\in \mathcal{X}_{i}.\end{equation}
Since $\Phi$ preserves the $Q_+ = \mathbb Z_+[I]$-grading and $\lie g$ can be finite-dimensional the map $\Phi$ need not be surjective. Let $\lie g^{i}$ be the image of the homomorphism $\Phi$ in $\lie g$. Then $\lie g^{i}$ is the Lie sub-superalgebra of $\lie g$ generated by $\{e(\bold w): \bold w\in \mathcal{X}_{i}\}$. Note that $\lie g^{i}$ is again $Q_+$-graded. Any basis of the free Lie superalgebra $FLS(\mathcal{X}_i)$ can be pushed forward through the map $\Phi$ to $\lie g^{i}$ and the image will be a spanning set of $\lie g^{i}$.  We construct a basis for the root space $\lie g_{\eta(\bold k)}$ from this spanning set; This is the main theorem of this section. This is done by identifying the following combinatorial model from \cite{akv17} with the set of super Lyndon heaps of weight $\bold k$.
%root space $\lie g_{\eta(\bold k)}$ with the grade space $\mathcal L_{\bold k}(G)$ for $\bold k$ satisfying $k_i\leq 1$ for $i\in I^{\text{re}}\sqcup \Psi_0$. Set
\begin{equation}\label{cikg}
C^{i}(\bold k, G)=\{\bold w\in \mathcal{X}_i^*:  \bold w \text{ is a super Lyndon word},\  \rm{wt}(\bold w)=\eta(\bold k)\},\ \ \iota(\bold w)=\Phi\circ L(\bold w).
\end{equation}
The precise statement is given in the next subsection.

\subsection{Theorem \ref{mainthmb}: LLN basis of BKM superalgebras} In this subsection, we state Theorem \ref{mainthmb} and two main lemmas  which are essential to prove this theorem. 
%deduce the proof of Theorem~\ref{mainthmb} from these lemmas. 
The proofs of these lemmas are postponed to the subsequent subsection.
\begin{thm}\label{mainthmb}
	The set $\left\{\iota(\bold w): \bold w\in C^{i}(\bold k, G)\right\}$ is a basis of the root space $\lie g_{\eta(\bold k)}$. Moreover, if $k_i=1$, the set
	$\left\{e(\bold w): \bold w\in \mathcal{X}_i,\ \rm{wt}(\bold w)=\eta(\bold k)\right\}$
	forms a left-normed basis of $\lie g_{\eta(\bold k)}$.
	%and
	%$$C^{i}(\bold k, G)=\{\bold w\in \mathcal{X}_i:  e(\bold w)\neq 0,\  \rm{wt}(\bold w)=\eta(\bold k)\}.$$ 
\end{thm}
%If $k_i=1$, the above theorem implies that the root space $\lie g_{\eta(\bold k)}$ has a very special type of basis, namely consists of right--normed Lie words $e(\bold w)$ of weight $\eta(\bold k)$.
\iffalse
\begin{prop}\label{helppropp}
	We have 
	\begin{enumerate}
		\item[(i)] The root space $\lie g_{\eta(\bold k)} = \lie g_{\eta(\bold k)}^i$.
		\item [(ii)] The spaces $\lie g_{\eta(\bold k)} = \mathcal {LS}_{\bold k}(G,\Psi)$
		%	\item[(ii)] Let $\bold w\in M(I,G)$ and $\rm{wt}(\bold w)=\eta(\bold k)$. Then 
		%	$$\bold w \in \mathcal X_i \implies e(\bold w)\neq 0.$$ 
		and 
		$$|C^{i}(\bold k, G)| = \dim FLS_{\bold k}(\mathcal X_i) = \mathcal{LS}_{\bold k}(G,\Psi).$$
	\end{enumerate}
\end{prop}\fi
\begin{lem}\label{helplem1}
	The root space $\lie g_{\eta(\bold k)} = \lie g_{\eta(\bold k)}^i$ for $\bold k \in \mathbb Z_+[I]$ satisfying $k_i \le 1$ for $i \in I^{re}\sqcup \Psi_0$.
\end{lem}
\begin{lem}\label{helplem2}
	%The spaces $\lie g_{\eta(\bold k)} = \mathcal {LS}_{\bold k}(G,\Psi)$
	%	\item[(ii)] Let $\bold w\in M(I,G)$ and $\rm{wt}(\bold w)=\eta(\bold k)$. Then 
	%	$$\bold w \in \mathcal X_i \implies e(\bold w)\neq 0.$$ 
	%and
	With the notations as above we have 
	$$|C^{i}(\bold k, G)| = \dim FLS_{\bold k}(\mathcal X_i) = \dim \mathcal{LS}_{\bold k}(G,\Psi).$$
\end{lem}
From the above lemmas Theorem \ref{mainthmb} can be deduced as follows. Since $\lie g_{\eta(\bold k)} = \lie g_{\eta(\bold k)}^i$ we get $\left\{\iota(\bold w): \bold w\in C^{i}(\bold k, G)\right\}$ is a spanning set for $\lie g_{\eta(\bold k)}$ of cardinality equal to $|C^{i}(\bold k, G)|$. Now, Lemmas  \ref{helplem2} and \ref{identification lem}  show that $\left\{\iota(\bold w): \bold w\in C^{i}(\bold k, G)\right\}$ is in fact a basis. 

%The identification of $\lie g_{\eta(\bold k)}$ and $\mathcal L_{\bold k}(G)$ is crucial to prove Proposition \ref{helppropp}(iii).

%We have the following lemma from \cite[Lemma 4.6]{akv17} which proves Proposition~\ref{helppropp}(i).
%\begin{lem}\label{spanlem}
%	The root space $\lie g_{\eta(\bold k)}$ is contained in $\lie g^i$
%Fix an index $i\in I$. Then the root space $\lie g_{\eta(\bold k)}$ is spanned by all right normed Lie words $e(\bold{w})$, where $\bold{w}\in M(I, G)$ is such that $\rm{wt}(\bold w)=\eta(\bold k)$ and $\rm{IA_m}(\bold w)=\{i\}.$
%\end{lem} 
%\begin{pf}
%	\textcolor{red}{to be done}
%\end{pf}
%%%%%%%%%%%%%%%%%%%%%%%%%%%5
\subsection{Examples to Theorem \ref{mainthmb}}
First, we explain Theorem \ref{mainthmb} by an example before giving the proofs of the Lemmas \ref{helplem1} and \ref{helplem2}.
\begin{example}\label{basis2_ex1} Consider the root space $\lie g_{\eta(\bold k)}$ where $\eta(\bold k)= 3\alpha_3+3\alpha_6 $ from Example \ref{basis1_ex1}. 
	%Let $I=\{ 1,2,3,4,5,6\},  \,{\Psi=\{3,5\}, I_1=\{3,5\}}, I_0=\{1,2,4,6\}, I^{re} =\{1,4\}, \eta(\bold k)= 3\alpha_3+3\alpha_6 $. 
	Fix $i=3$. Super Lyndon words of weight $\eta(\bold k)$ in $C^3(\bold k, G)=\{\bold w \in \chi^\ast_3: \wt(\bold w)=\eta(\bold k), \bold w \text{ is super Lyndon}  \}$ are $\{336636,333666,336366\}$ and 
	$\{[[3,[[3,6],6]],[3,6]], [3,[3,[[[3,6],6],6]]],   [3,[[3,6],[[3,6],6]]]\}$ spans $\mathfrak{g}_{\eta(\bold k)}$. We have $\mult(\eta(\bold k))=3$ [c.f. Example \ref{ch_ex1}]. So these Lie monomials form a basis for $\mathfrak{g}_{\eta(\bold k)}.$ 
\end{example}

\begin{example}\label{basis2_ex2} Consider the root space $\lie g_{\eta(\bold k)}$ where $\eta(\bold k)= 2\alpha_3+\alpha_4+2 \alpha_5+\alpha_6 $ from Example \ref{basis1_ex2}. 
	%Continuing with assumptions of example \ref{basis1_ex2} Let $I=\{ 1,2,3,4,5,6\},  \, \Psi=\{3,5\}, , I_0=\{1,2,4,6\}, I^{re} =\{1,4\}, \eta(\bold k)= 2\alpha_3+\alpha_4+2 \alpha_5+\alpha_6 $. 
	Fix $i=3$, observe that $\chi_3=\{3,34,345,3456,344,34545,\cdots\}$. Only Lyndon words on $\chi_3^\ast$ of weight $\eta(\bold k)$ are $\{334556,334565\}$ with standard factorization $3|34556,3|34565$ respectively. So the  corresponding Lie monomials 
		$\{[3,[[[[3,4],5],6],5]], [3,[[[[3,4],5],5],6]]\}$ spans $\mathfrak{g}_{\eta(\bold k)}$. We have $\mult(\eta(\bold k))=2$ [c.f. Example \ref{multf0}]. So these Lie monomials form a basis for  $\mathfrak{g}_{\eta(\bold k)}$. \end{example}

%%%%%%%%%%%%%%%%%%%%%%%%%%%%%%%%%%%%%%%%%%

\subsection{Proof of Lemma \ref{helplem1}}
We claim that the root space $\lie g_{\eta(\bold k)} = \lie g_{\eta(\bold k)}^i$. This is done in multiple steps. First, we claim that the left normed Lie words of weight $\bold k$ starting with a fixed $i \in I$ spans $\lie g_{\eta(\bold k)}$. More precisely,
\begin{lem}\label{s3}
	Fix an index $i\in I$. Then the root space $\lie g_{\eta(\bold k)}$ is spanned by the set of left normed lie words $\{e(\bold w): \bold w \in \chi_i^{\ast}, \wt(\bold w) = \eta(\bold k)\}$.
\end{lem}
\begin{pf}
	We observe that $\bold w \in \chi_i^{\ast}$ if, and only if, $\rm{IA}(\bold w)= \{i\}$. It is well-known that, the set $\mathcal B = \{e(\bold w): \bold w \in M_{\bold k}(I,G)\}$ forms a spanning set for $\lie g_{\eta(\bold k)}$. We will prove that each element of $\mathcal B$ can be written as linear combination of left normed Lie words $e(\bold{w})$ satisfying $\rm{wt}(\bold w)=\eta(\bold k)$ and $\rm{IA}(\bold w)=\{i\}$.  Let $\bold w = a_1a_2\cdots a_r \in M_{\bold k}(I,G)$. 
	Assume that $a_1 =i$. %and $\rm{IA}(\bold w)=\{i\}$. %$\Longrightarrow e(\bold{w})=[[[a_1,a_2],a_3]\cdots,a_r]$\\
	%case(i) If $ a_1=i $ then there is nothing to proof.\\
	If $|\rm{IA}(\bold w)|>1$ then $e(\bold w) = 0$ and nothing to prove. If $|\rm{IA}(w)| = 1$ then we have $\rm{IA}(\bold w) = \{i\}$ and the proof follows in this case. Assume $ a_1 \neq i $ and consider the set $i(\bold w)= \{j: a_j=i\}$. Assume $\min\{i(\bold w)\}= p+1$ and set $\bold w' =a_1a_2\cdots a_p i $.  %Then $ a_j \neq i $ for $1 \leq j \leq p$.
	
	First, we claim that
	\begin{align*}
	e(\bold w')= &e(ia_1a_2\cdots a_p)+ \sum\limits_{j_1=2}^{p} e(ia_{j_1}a_1a_2\cdots \hat{a_{j_1}} \cdots a_p)+ \sum\limits_{1<j_2<j_1\leq p} e(ia_{j_1}a_{j_2}a_1a_2\cdots \hat{a_{j_2}}\cdots \hat{a_{j_1}} \cdots a_p)\\
	& +\sum\limits_{1<j_3<j_2<j_1\leq p} e(ia_{j_1}a_{j_2}a_{j_3}a_1a_2\cdots\hat{a_{j_3}}\cdots \hat{a_{j_2}}\cdots \hat{a_{j_1}}\cdots a_p)\\
	&+ \sum\limits_{1<j_4<j_3<j_2<j_1\leq p} e(ia_{j_1}a_{j_2}a_{j_3}a_{j_4}a_1a_2 \cdots\hat{a_{j_4}}\cdots\hat{a_{j_3}}\cdots\hat{a_{j_2}}\cdots \hat{a_{j_1}} \cdots a_p)+\cdots  \\
	&+ \sum\limits_{1<j_{p-2}<j_{p-3}<\cdots<j_2<j_1\leq p} e(ia_{j_1}a_{j_2}a_{j_3} \cdots a_{j_{p-3}}a_{j_{p-2}} a_1a_2 \cdots \hat{a_{j_{p-2}}} \cdots\hat{a_{j_{p-3}}}\cdots\hat{a_{j_2}}\cdots \hat{a_{j_1}}\cdots a_p) \\ 
	&+e(ia_pa_{p-1}\cdots a_2a_1)
	\end{align*}
	where $\hat{a}$ means the omission of the alphabet $a$ in the expression. 
	
	We explain the above equation with an example for better understanding. Consider the root space $\lie g_{\eta(\bold k)}$ where $\eta(\bold k)= 2\alpha_3+\alpha_4+2 \alpha_5+\alpha_6 $ from Example \ref{basis1_ex2}. Fix $i=3$ and consider the word $\bold w =456353$. Then $p=3$ and $\bold {w'} = 4563$.
	Now, \begin{align*}
	e(\bold w') &=[[[4,5],6],3]= [3,[[4,5],6]] = [[3,[4,5]],6]-[[4,5],[3,6]]\\
	&=[[[3,4],5],6]+[[4,[3,5]],6]-[[4,5],[3,6]]\\
	&=[[[3,4],5],6]-[[[3,5],4],6]+[[3,6],[4,5]]\\
	&= [[[3,4],5],6]-[[[3,5],4],6]+[[[3,6],4],5]+[4,[[3,6],5]]\\
	&=[[[3,4],5],6]-[[[3,5],4],6]+[[[3,6],4],5]-[[[3,6],5],4]\\
	&= e(3456)- e(3546)+ e(3645)-e(3654)
	\end{align*}

	\iffalse	
	\begin{align*}
	e(\bold w')= & \sum\limits_{k=1}^{p} e(ia_ka_1a_2\cdots \hat{a_k} \cdots a_p)+ \sum\limits_{k=1}^{p} \sum\limits_{1<j_1<k} e(ia_ka_{j_1}a_1a_2\cdots \hat{a_{j_1}}\cdots \hat{a_k} \cdots a_p)\\
	& + \sum\limits_{k=1}^{p} \sum\limits_{1<j_1<k}\sum\limits_{1<j_2<j_1} e(ia_ka_{j_1}a_{j_2}a_1a_2\cdots\hat{a_{j_2}}\cdots \hat{a_{j_1}}\cdots \hat{a_k} \cdots a_p)\\
	&+ \sum\limits_{k=1}^{p} \sum\limits_{1<j_1<k}\sum\limits_{1<j_2<j_1}\sum\limits_{1<j_3<j_2} e(ia_ka_{j_1}a_{j_2}a_{j_3}a_1a_2 \cdots\hat{a_{j_3}}\cdots\hat{a_{j_2}}\cdots \hat{a_{j_1}}\cdots \hat{a_k} \cdots a_p) \\
	& + \cdots + \sum\limits_{k=1}^{p} \sum\limits_{1<j_1<k}\sum\limits_{1<j_2<j_1}\sum\limits_{1<j_3<j_2} \cdots \sum\limits_{1<j_{p-1}<j_{p-2}} e(ia_ka_{j_1}a_{j_2}a_{j_3} \cdots {a_{j_{p-1}}})  
	\end{align*}
	\fi
	
	We do induction on $p$. For $ p=1$, $\bold{w'} = a_1i \Rightarrow e(\bold{w'})= [i,a_1] = e(ia_1)$. Assume that the result is true for $ p=k.$ Now consider $p=k+1$
	\begin{equation}
	\label{e1} [[[[[[a_1,a_2],a_3],a_4] \cdots ,a_k],a_{k+1}],i]=[[[[[a_1^{'},a_3]a_4], \cdots,a_k],a_{k+1}],i]
	\end{equation}
	by taking  $[a_1,a_2]=a_1^{'}$. Using the induction hypothesis on right hand side of the above equation we get
	\begin{equation}
	\label{eq2}
	\begin{aligned}
	&[[[[[[[a_1^{'},a_{3}],a_4], \cdots,a_k],a_{k+1}],i] = \\	& e(ia_1^{'}a_3\cdots a_{k+1})+ \sum\limits_{3 \leq j \leq k+1} e(ia_ja_1^{'}a_3 \cdots \hat{a_j}\cdots a_{k+1}) \\
	&+\sum\limits_{3 \leq j_2 <j_1 \leq k+1} e(ia_{j_1}a_{j_2}a_1^{'}a_3 \cdots \hat{a_{j_2}} \cdots \hat{a_{j_1}}\cdots a_{k+1})
	+\cdots\\& +\sum\limits_{3\leq j_{k-1}<\cdots<j_2<j_1\leq {k+1}} e(ia_{j_1}a_{j_2}a_{j_3} \cdots a_{j_{k-1}} a_1^{'}a_3 \cdots\hat{a_{j_{k-1}}}\cdots\hat{a_{j_2}}\cdots \hat{a_{j_1}}\cdots a_{k+1})\\
	&+ e(ia_{k+1}a_k \cdots a_1^{'})
	\end{aligned}
	\end{equation}
	
	\iffalse
	Now  \begin{align*}
	e(ia_1^{'})&=[i,[a_1,a_2]]=[[i,a_1],a_2]+[[i,a_2],a_1]\\
	&= e(ia_1a_2)+e(ia_2a_1)
	\end{align*}
	
	and  \begin{align*}
	e(ia_{j_1}a_1^{'})&=[[i,a_{j_1}],a_1^{'}]
	=[\underbrace{[i, a_{j_1}]}_{x},[\underbrace{a_1}_{y},\underbrace{a_2}_{z}]]\\
	&=\underbrace{[[[i,a_{j_1}],a_1],a_2]}_{[[x,y],z]}+ \underbrace{[a_1,[[i,a_{j_1}],a_2]]}_{[y,[x,z]]}\\
	&=\underbrace{[[[i,a_{j_1}],a_1],a_2]}_{[[x,y],z]}+ \underbrace{[[[i,a_{j_1}],a_2],a_1]}_{[[x,z],y]}\\
	&= e(ia_{j_1}a_1a_2)+e(ia_{j_1}a_2a_1)
	\end{align*}
	
	Similarly
	\fi
	
	Now,  \begin{align*}
	e(ia_{j_1}a_{j_2} \cdots a_{j_t}a_1^{'})&= [\underbrace{[[[i,a_{j_1}],a_{j_2}], \cdots ,a_{j_t}]}_{x},[\underbrace{a_1}_{y}, \underbrace{a_2}_{z}]]\\
	&=\underbrace{[[[[[i,a_{j_1}],a_{j_2}], \cdots ,a_{j_t}],a_1],a_2]}_{[[x,y],z]}+ \underbrace{[[[[[i,a_{j_1}],a_{j_2}], \cdots ,a_{j_t}],a_2],a_1]}_{[[x,z],y]}\\
	&= e(ia_{j_1}a_{j_2}\cdots a_{j_t}a_1a_2)+ e(ia_{j_1}a_{j_2}\cdots a_{j_t}a_2a_1)
	\end{align*}
	\begin{align*}
	\Rightarrow e(ia_{j_1}a_{j_2} \cdots a_{j_t}a_1^{'}a_3\cdots a_{k+1})&= [[\underbrace{[[[[i,a_{j_1}],a_{j_2}], \cdots a_{j_t}]a_1^{'}]}_{e(ia_{j_1}a_{j_2} \cdots a_{j_t}a_1^{'})},a_3],\cdots a_{k+1}]\\
	&=[[[e(ia_{j_1}a_{j_2} \cdots a_{j_t}a_1^{'}),a_3], \cdots,a_{k+1}]\\
	&=[[[e(ia_{j_1}a_{j_2}\cdots a_{j_t}a_1a_2)+ e(ia_{j_1}a_{j_2}\cdots a_{j_t}a_2a_1),a_3], \cdots,a_{k+1}]\\
	&=[[e(ia_{j_1}a_{j_2}.. a_{j_t}a_1a_2),a_3],..,a_{k+1}]+[[ e(ia_{j_1}a_{j_2}.. a_{j_t}a_2a_1) ,a_3]..,a_{k+1}]\\
	&= e(ia_{j_1}a_{j_2}\cdots a_{j_t}a_1a_2a_3\cdots a_{k+1})+e(ia_{j_1}a_{j_2}\cdots a_{j_t}a_2a_1a_3\cdots a_{k+1})
	\end{align*} Using this in Equation \eqref{eq2} we get 
	\begin{align*}
	&[[[[[a_1^{'},a_3]a_4], \cdots,a_k],a_{k+1}],i] = \\
	& e(ia_1a_2a_3\cdots a_{k+1})+e(ia_2a_1a_3\cdots a_{k+1})\\
	&+ \sum\limits_{3 \leq j \leq k+1} \left( e(ia_ja_1a_2a_3 \cdots \hat{a_j}\cdots a_{k+1})+e(ia_ja_2a_1a_3 \cdots \hat{a_j}\cdots a_{k+1})\right)  \\
	&+\sum\limits_{3 \leq j_2 <j_1 \leq k+1} \left( e(ia_{j_1}a_{j_2}a_1a_2a_3 \cdots \hat{a_{j_2}} \cdots \hat{a_{j_1}}\cdots a_{k+1})+e(ia_{j_1}a_{j_2}a_2a_1a_3 \cdots \hat{a_{j_2}} \cdots \hat{a_{j_1}}\cdots a_{k+1})\right) +\cdots\\
	&+\sum\limits_{3\leq j_{k-1}<\cdots<j_2<j_1\leq {k+1}}  e(ia_{j_1}a_{j_2}a_{j_3} \cdots a_{j_{k-1}} a_1a_2 \cdots\hat{a_{j_{k-1}}}\cdots\hat{a_{j_2}}\cdots \hat{a_{j_1}}\cdots a_{k+1})\\
	&+e(ia_{j_1}a_{j_2}a_{j_3} \cdots a_{j_{k-1}} a_2a_1 \cdots\hat{a_{j_{k-1}}}\cdots\hat{a_{j_2}}\cdots \hat{a_{j_1}}\cdots a_{k+1})\\ 
	&+\left( e(ia_{k+1}a_k \cdots a_1a_2)+e(ia_{k+1}a_k \cdots a_2a_1)\right)  \\ 
	& = \left(  e(ia_1a_2a_3\cdots a_{k+1})+e(ia_2a_1a_3\cdots a_{k+1})
	+ \sum\limits_{3 \leq j \leq k+1}  e(ia_ja_1a_2a_3 \cdots \hat{a_j}\cdots a_{k+1}) \right) \\
	&+ \left(\sum\limits_{3 \leq j \leq k+1}  e(ia_ja_2a_1a_3 \cdots \hat{a_j}\cdots a_{k+1})+\sum\limits_{3 \leq j_2 <j_1 \leq k+1}  e(ia_{j_1}a_{j_2}a_1a_2a_3 \cdots \hat{a_{j_2}} \cdots \hat{a_{j_1}}\cdots a_{k+1})\right) 
		\end{align*} 
	\begin{align*}
	&+( \sum\limits_{3 \leq j_2 <j_1 \leq k+1} e(ia_{j_1}a_{j_2}a_2a_1a_3 \cdots \hat{a_{j_2}} \cdots \hat{a_{j_1}}\cdots a_{k+1})+ \\
	&\sum\limits_{3 \leq j_3< j_2 <j_1 \leq k+1}e(ia_{j_1}a_{j_2}a_{j_3}a_2a_1a_3 \cdots \hat{a_{j_3}} \cdots \hat{a_{j_2}} \cdots \hat{a_{j_1}}\cdots a_{k+1}) ) 
	+ \cdots+ \\
	&+\left(\sum\limits_{3\leq j_{k-1}<\cdots<j_2<j_1\leq {k+1}}e(ia_{j_1}a_{j_2}a_{j_3} \cdots a_{j_{k-1}} a_2a_1 \cdots\hat{a_{j_{k-1}}}\cdots\hat{a_{j_2}}\cdots \hat{a_{j_1}}\cdots a_{k+1})+e(ia_{k+1}a_k \cdots a_1a_2)\right)\\
	&+e(ia_{k+1}a_k \cdots a_2a_1)	\end{align*}
	\begin{align*}
	= &\sum\limits_{j=1}^{k+1}e(ia_ja_1a_2\cdots \hat{a_j}\cdots a_{k+1})+ \sum\limits_{1<j_2<j_1\leq k+1} e(ia_{j_1}a_{j_2}a_1a_2 \cdots \hat{a_{j_2}}\cdots \hat{a_{j_1}}\cdots a_{k+1})\\
	&+\sum\limits_{1<j_3<j_2<j_1\leq k+1} e(ia_{j_1}a_{j_2}a_{j_3}a_1a_2 \cdots \hat{a_{j_3}}\cdots \hat{a_{j_2}}\cdots \hat{a_{j_1}}\cdots a_{k+1})
	+ \cdots +\\
	&+\sum\limits_{1< j_{k-1}<\cdots<j_2<j_1\leq {k+1}} e(ia_{j_1}a_{j_2}a_{j_3} \cdots a_{j_{k-1}} a_1a_2 \cdots\hat{a_{j_{k-1}}}\cdots\hat{a_{j_2}}\cdots \hat{a_{j_1}}\cdots a_{k+1})\\
	&+e(ia_{k+1}a_k \cdots a_2a_1)
	\end{align*} Thus the result is true for p=k+1. This proves our claim. From this claim proof of the lemma follows from the following steps.
	\begin{align*}
	e(\bold w')= &e(ia_1a_2\cdots a_p)+\sum\limits_{j_1=2}^{p} e(\bold w_{j_1})+  \sum\limits_{1<j_2<j_1\leq p} e(\bold w_{j_1j_2})+ \sum\limits_{1<j_3<j_2<j_1\leq p} e(\bold w_{j_1j_2j_3})+ \cdots +\\ 
	&+e(\bold w_{p(p-1)\cdots 2 \cdot 1})
	\end{align*}
	
	where $\bold w_{j_1j_2\cdots j_l}= (ia_{j_1}a_{j_2}\cdots a_{j_l}a_1a_2 \cdots \hat{a_{j_l}}\cdots\hat{a_{j_{l-1}}}\cdots \hat{a_{j_1}}\cdots a_p).$ We observe that all the words $\bold w_{j_1,\dots j_m}$ have the same weight and belongs to $\chi_i$. This is because if some $a_{j_p} $ commutes with $i,a_{j_1},a_{j_2},\cdots,a_{j_{p-1}}$ then $e(\bold w_{j_1j_2\cdots j_l})=0$ . In our example, $e(\bold w')= e(\bold{w'_4})+e(\bold{w'_{42}})+e(\bold{w'_{43}})+e(\bold{w'_{432}}) \,\, \text{ where } \bold w'=a_1a_2a_3a_4=4563$. Now, By the linearity property of the brackets we have 
	
	$e(\bold w' \cdot a_{p+2})= \sum\limits_{j_1=1}^{p} e(\bold w_{j_1}\cdot a_{p+2})+ \sum\limits_{1<j_2<j_1\leq p} e(\bold w_{j_1j_2} \cdot a_{p+2})+ \sum\limits_{1<j_3<j_2<j_1\leq p} e(\bold w_{j_1j_2j_3} \cdot a_{p+2})+
	+ \cdots  +e(\bold w_{p(p-1)\cdots 2.1} \cdot a_{p+2})$.

	Similarly, we can add all the remaining alphabets $a_{p+3},\dots,a_r$ to the above expression. This will give us 
	
	$$e(\bold w)= \sum\limits_{\substack{\bold u \in \mathcal X_i^{*} \\ \rm{wt}(\bold u)=\eta(\bold k)}}\alpha(\bold u) e(\bold u) \text{ for some scalars }\alpha(\bold u).$$

	In our example,\begin{align*}
	e(\bold {w'}.5) &= [[[[4,5],6],3],5] = [e(\bold {w'}), 5]\\
	&=[e(3456)-e(3546)+e(3645)-e(3654), 5]\\
	&=[e(3456),5]-[e(3546),5]+[e(3645),5]-[e(3654),5]\\
	&=e(34565)-e(35465)+e(36455)-e(36545)
	\end{align*}
	Thus \begin{align*}
	e(\bold w) &= [e(45635),3]=[e(\bold {w'}.5),3]\\
	&=[e(34565)-e(35465)+e(36455)-e(36545), 3]\\
	&=[e(34565),3]-[e(35465),3]+[e(36455),3]- [e(36545),3]\\
	&=e(345653)- e(354653)+ e(364553)-e(365453)\\	
	\end{align*} This completes the proof.
\end{pf}
\begin{lem}\label{s1}
	If $ \bold u \ne \bold{v}  \in \chi_i^*$ are Lyndon words then exactly one element of the set $ \{\bold u \bold{v}, \bold{v} \bold u \} $ is Lyndon.
\end{lem}
\begin{pf}
	%A word $\bold w \in \mathcal X_i^*$ is Lyndon if $\bold w \in \mathcal{X}_i$ or $\bold w=\bold u \bold v$ for Lyndon words $\bold u$ and $\bold v$ with $\bold u<\bold v$. Now, 
	We observe that if $\bold u< \bold v$ then $\bold u \bold v$ is Lyndon otherwise $\bold v \bold u$ is Lyndon. 
\end{pf}
\begin{lem}\label{s2}
	If $ \bold w,\bold{\tilde{w}}  \in \mathcal X_i^*$ are Lyndon words with standard factorization $\bold w= \bold{u_1u_2}, \, \bold{\tilde{w}}= \bold{v_1v_2} $. Assume that $\bold w \bold{\tilde{w}}$ is a Lyndon word. Then
	$$ [L(\bold w), L(\bold{\tilde{w}})]\in \rm{span}\{ L(C^i(\rm{wt}(\bold w \tilde{\bold w})),G)\}$$
	% $[L(\bold w), L(\bold{\tilde{w}})]=\begin{cases}
	%L(\bold w \bold{\tilde{w}}) &  \text{if } \bold w \bold{\tilde{w}}=u_1u_2|v_1v_2 \text{ is the standard factorization}.\\
	%[L(\bold{u_1}),L(\bold{u_2}\bold{\tilde{w}})]+[L(\bold{u_1}\bold{\tilde{w}}),L(\bold{u_2})] & \text{otherwise}
	%\end{cases}.$
	%Moreover $[L(\bold w), L(\bold{\tilde{w}})] \in \rm{span}\{ L(C^i(\rm{wt}(w)+\rm{wt}(\tilde w)),G)\}$
	%Note:- If $\bold w \bold{\tilde{w}}$ is not Super Lyndon in \ref{s2}, then interchange $ \bold w$ with $\bold{\tilde{w}}$ as $\bold{\tilde{w}}\bold w$ will definitely be Super-Lyndon by \ref{s1}.
	
\end{lem}
\begin{pf}%Since $\bold w= u_1u_2, \, \bold{\tilde{w}}= v_1v_2 $ are Standard factorization of  Super-Lyndon words.\\without loss of generality assume  
	%We have $\bold w \bold{\tilde{w}}=u_1u_2v_1v_2 $ is super-Lyndon. 
	We have two possibles situations: either $\bold{u_2} \geq \bold{\tilde{w}}$ or $\bold{u_2} < \bold{\tilde{w}}.$ 
	
	If $\bold{u_2} \geq \bold{\tilde{w}}$ then $\bold{u_1u_2}|\bold{v_1v_2}$ is the standard factorization of $\bold w \bold{\tilde{w}}$.  Observe that $\bold{u_1u_{21}}|\bold{u_{22}v_1v_2}$ can't be standard  factorization for some standard factorization $\bold{u_2}=\bold{u_{21}}\bold{u_{22}}$ as $\bold{u_{22}}<\bold{\tilde{w}}\leq \bold{u_2}=\bold{u_{21}}\bold{u_{22}}$ means $\bold{u_{22}}< \bold{u_{21}}$ i.e. $\bold{u_2}=\bold{u_{21}}\bold{u_{22}}$ can't be a Lyndon word. From this observation, the proof is immediate in this case: We have $\bold w \bold{\tilde{w}} =\bold{u_1u_2}|\bold{v_1v_2}$ is the standard factorization
	and $[L(\bold w), L(\bold{\tilde{w}})]=L(\bold w \bold{\tilde{w}}) \in \rm{span}\{ L(C^i(\rm{wt}(\bold{w \tilde w})),G)\}.$
	
	If $\bold{u_2}<\bold{\tilde{w}}$ then $\bold{u_1}|\bold{u_2v_1v_2}$  is the standard factorization of $\bold{w\tilde{w}}$. Observe that $\bold{u_{11}}|\bold{u_{12}u_2v_1v_2}$ can't be the standard factorization for some standard factorization of $\bold{u_1}=\bold{u_{11}}\}bold{u_{12}}$ as $\bold{u_{12}}<\bold{u_2}$ means $\bold{u_{12}}\bold{u_2} $ is the longest right factor of $\bold w=\bold{u_{11}}\bold{u_{12}}\bold{u_2}$ which contradicts the statement:  $\bold w= \bold{u_1u_2}$ is the standard factorization.

	If $\bold w \bold{\tilde{w}} =\bold{u_1}|\bold{u_2v_1v_2}$ is standard factorization
	then \begin{align*}
	[L(\bold w), L(\bold{\tilde{w}})]&=
	[L(\bold{u_1u_2}),L(\bold{\tilde{w}})]\\
	&=[[L(\bold{u_1}),L(\bold{u_2})],L(\bold{\tilde{w}})]\\
	&=[L(\bold{u_1}),[L(\bold{u_2}),L(\bold{\tilde{w}})]+[[L(\bold{u_1}),L(\bold{\tilde{w}})],L(\bold{u_2})]
	\end{align*} 
	
	\textbf{subcase(i)}:- If $ \bold{u_2} \bold{\tilde{w}} $ is Lyndon word with standard factorization $\bold{u_2}| \bold{\tilde{w}} $ and $\bold{u_1} \bold{\tilde{w}} $ is a Lyndon word with standard factorization $\bold{u_1}| \bold{\tilde{w}} $ then
	\begin{align*}
	[L(\bold w), L(\bold{\tilde{w}})]&=[L(\bold{u_1}),L(\bold{u_2}\bold{\tilde{w}})]+[L(\bold{u_1}\bold{\tilde{w}}),L(\bold{u_2})]\\
	&=[L(\bold{u_1}),L(\bold{u_2}\bold{\tilde{w}})]+ L(\bold{u_1}\bold{\tilde{w}}\bold{u_2})\\
	&\text{as } \bold{u_2} <\bold{\tilde{w}} \text{ so } \bold{u_1}\bold{\tilde{w}}|\bold{u_2} \text{ is standard factorization}.
	\end{align*}
	Now repeat the above procedure again on $[L(\bold{u_1}),L(\bold{u_2}\bold{\tilde{w}})]$.
	% term of right hand side of the above equation, i.e., take $\bold{u_1}, \bold{u_2}\bold{\tilde{w}}$ in place of $\bold w, \bold{\tilde{w}}$ respectively. 
	Continue this procedure on the subsequent terms till we get the term like $[L(\bold{v_1}),L(\bold{v_2})]$ where $\bold{v_1}, \bold{v_2}$ are Lyndon words with $\bold{v_1} \in \chi_i.$ This is possible since $wt(\bold{u_1}) < wt(\bold w)$.	
	
	\textbf{subcase(ii)}:- If $ \bold{u_2} \bold{\tilde{w}} $ is Lyndon word with standard factorization $ \bold{u_{21}}| \bold{u_{22}} \bold{\tilde{w}} $  for standard factorization $ \bold{u_2}= \bold{u_{21}} \bold{u_{22}} $ and $\bold{u_1} \bold{\tilde{w}} $ is Lyndon word with standard factorization $ \bold{u_1}| \bold{\tilde{w}} $ then
	\begin{align*}
	&=[L(\bold{u_1}),[[L(\bold{u_{21}}), L(\bold{u_{22}})],L(\bold{\tilde{w}})]]+[L(\bold{u_1} \bold{\tilde{w}}),L(\bold{u_2})] \\
	&=[L(\bold{u_1}),[[L(\bold{u_{21}}),L(\bold{\tilde{w}})], L(\bold{u_{22}})]]+[L(\bold{u_1}),[L(\bold{u_{21}}),[L(\bold{u_{22}}), L(\bold{\tilde{w}})]]]+L(\bold{u_1} \bold{\tilde{w}}\bold{u_2})
	\end{align*}
	Repeat the above procedure firstly for $[L(\bold{u_{21}}),L(\bold{\tilde{w}})], [L(\bold{u_{22}}), L(\bold{\tilde{w}})] $, then using this in the above equation and repeat the procedure for subsequent terms and so on. This process will end when we got terms like $[L(\bold{v_1}),L(\bold{v_2})]$ where $\bold{v_1}, \bold{v_2}$ are Lyndon words with $\bold{v_1} \in \chi_i.$ 
	
	\textbf{subcase(iii)}:- If $ \bold{u_2} \bold{\tilde{w}} $ is Lyndon word with standard factorization $ \bold{u_{21}}|\bold{u_{22}} \bold{\tilde{w}} $  for the standard factorization $\bold{u_2}= \bold{u_{21}} \bold{u_{22}} $ and $\bold{u_1} \bold{\tilde{w}} $ is Lyndon word with standard factorization $\bold{u_{11}}| \bold{u_{12}} \bold{\tilde{w}} $  for standard factorization $ \bold{u_1}= \bold{u_{11}} \bold{u_{12}} $ then
	\begin{align*}
	& =[L(\bold{u_1}),[[L(\bold{u_{21}}), L(\bold{u_{22}})],L(\bold{\tilde{w}})]]+[[[L(\bold{u_{11}}),L(\bold{u_{12}})],L(\bold{\tilde{w}})],L(\bold{u_2})] \\
	&=[L(\bold{u_1}),[[L(\bold{u_{21}}),L(\bold{\tilde{w}})], L(\bold{u_{22}})]]+[L(\bold{u_1}),[L(\bold{u_{21}}),[L(\bold{u_{22}}), L(\bold{\tilde{w}})]]]\\
	& +[[L(\bold{u_{11}}),[L(\bold{u_{12}}),L(\bold{\tilde{w}})]], L(\bold{u_{2}})]+[[[L(\bold{u_{11}}), L(\bold{\tilde{w}})],L(\bold{u_{12}})]
	,L(\bold{u_1})] 
	\end{align*} Repeat the above procedure firstly for $[L(\bold{u_{12}}),L(\bold{\tilde{w}})],[L(\bold{u_{11}}), L(\bold{\tilde{w}})] $, then using this in the above equation and repeat the procedure for subsequent terms and so on. This process will end when we got terms like $[L(\bold{v_1}),L(\bold{v_2})]$ where $\bold{v_1}, \bold{v_2}$ are super Lyndon words with $\bold{v_1} \in \chi_i.$ \end{pf}

The following example explains the above lemma.
\begin{example} Consider the root space $\lie g_{\eta(\bold k)}$ where $\eta(\bold k)= 2\alpha_3+\alpha_4+2 \alpha_5+\alpha_6 $ from Example \ref{basis1_ex2}. Fix $i=3$ .
	Let $\bold w=334345, \bold w'= 34635364 \in \chi_3^\ast $ then $\bold w=3|34345=\bold{u_1}\bold{u_2}, \bold w'= 346|35364=v_1v_2$ are the standard factorizations . Thus $L(\bold w)=[L(\bold{u_1}),L(\bold{u_2})]=[3,[34,345]]$ and $L(\bold w')=[L(\bold{v_1}),L(\bold{v_2})]=[346,[35,364]]$.
	Since $\bold w \bold w'
	=\underbrace{334345}_{u
		_1u_2}\underbrace{34635364}_{\bold w'}
	=\underbrace{3}_{\bold{u_1}}|\underbrace{3434534635364}_{\bold{u_2}\bold w'}$ is the standard factorization, we have
	
$[L(\bold w),L(\bold w')]$
	\begin{align*}
 &=[\underbrace{[L(3),L(34345)]}_{[L(\bold{u_1}),L(\bold{u_2})]},\underbrace{L(34635364)}_{L(\bold w')}]\\
	&=\underbrace{[L(3),[L(34345),L(34635364)]]}_{[L(\bold{u_1}),[L(\bold{u_2}),L(\bold w')]]}+\underbrace{[[L(3),L(34635364)],L(34345)]}_ {[[L(\bold{u_1}),L(\bold w')],L(\bold{u_2})]}\\
	&=\underbrace{[L(3),[[L(34),L(345)],L(34635364)]]}_{[L(\bold{u_1}),[[L(\bold{u_{21}}),L(\bold{u_{22}})],L(\bold w')]]}+\underbrace{[L(334635364),L(34345)]}_ {[L(\bold{u_1}\bold w'),L(\bold{u_2})]}\\
	&=[L(3),\underbrace{[[L(34),L(34635364)],L(345)]+[L(34),[L(345),L(34635364)]]}_{[[L(\bold{u_{21}}),L(\bold w')],L(\bold{u_{22}})]+[L(\bold{u_{21}}),[L(\bold{u_{22}}),L(\bold w')]]}]+\underbrace{L(33463536434345)}_ {L(\bold{u_1}\bold w'\bold{u_2})}
	\end{align*}	
		\begin{align*}
	&=[L(3),\underbrace{[L(3434635364),L(345)]+[L(34),L(34534635364)]}_{[L(\bold{u_{21}}\bold w'),L(\bold{u_{22}})]+[L(\bold{u_{21}}\bold{u_{22}}),L(\bold w')]}]+\underbrace{L(33463536434345)}_ {L(\bold{u_1}\bold w'\bold{u_2})}\\
	&=[L(3),\underbrace{L(3434635364345)+L(3434534635364)}_{L(\bold{u_{21}}\bold w'\bold{u_{22}})+L(\bold{u_{21}}\bold{u_{22}})\bold w')}]+\underbrace{L(33463536434345)}_ {L(\bold{u_1}\bold w'\bold{u_2})}\\
	&=\underbrace{L(33434635364345)}_{L(\bold{u_1}\bold{u_{21}}\bold w'\bold{u_{22}})}+\underbrace{L(33434534635364)}_{L(\bold{u_1}\bold{u_{21}}\bold{u_{22}}\bold w')}+\underbrace{L(33463536434345)}_ {L(\bold{u_1}\bold w'\bold{u_2})}\\
	\end{align*}
	\end{example}

\begin{example}Let $I=\{ 1,2,3,4,5,6\},  \, {\Psi=\{3,5\}, I_1=\{3,5\}}, I_0=\{1,2,4,6\}, I^{re} =\{1,4\}, \eta(\bold k)= 2\alpha_3+\alpha_4+2 \alpha_5+\alpha_6 $. Fix $i=3$ .
	Let $\bold w=334365, \bold w'= 34635364 \in \chi_3^\ast $ then $\bold w=3|34365=\bold{u_1}\bold{u_2}, \bold w'= 346|35364=\bold{v_1}\bold{v_2}$ be standard factorization of these words.
	Since $\bold w \bold w'
	=\underbrace{334365}_{\bold{u_1}\bold{u_2}}\underbrace{34635364}_{\bold w'}
	=\underbrace{3}_{\bold{u_1}}|\underbrace{3436534635364}_{\bold{u_2}\bold w'}$
	is the standard factorization,
	\begin{align*}
	[L(\bold w),L(\bold w')] &=[\underbrace{[L(3),L(34365)]}_{[L(\bold{u_1}),L(\bold{u_2})]},\underbrace{L(34635364)}_{L(\bold w')}]\\
	&=\underbrace{[L(3),[L(34365),L(34635364)]]}_{[L(\bold{u_1}),[L(\bold{u_2}),L(\bold w')]]}+\underbrace{[[L(3),L(34635364)],L(34345)]}_ {[[L(\bold{u_1}),L(\bold w')],L(\bold{u_2})]}\\
	&=\underbrace{[L(3),L(3436534635364)]}_{[L(\bold{u_1}),L(\bold{u_2}\bold w')]}+\underbrace{[L(334635364),L(34345)]}_ {[L(\bold{u_1}\bold w'),L(\bold{u_2})]}\\
	&=\underbrace{L(33436534635364)}_{L(\bold{u_1}\bold{u_2}\bold w')}+\underbrace{L(33463536434345)}_ {L(\bold{u_1}\bold w'\bold{u_2})}
	\end{align*}
\end{example}

\begin{lem}\label{s4} If $\bold{ w_a}$ and $\bold{w_b}$ are super Lyndon words
	then $[L(\bold {w_a}),L(\bold{w_b})] \in \rm{span}\{ L(C^i(\rm{wt}(\bold{w_a}\bold{w_b})),G)\}.$
\end{lem}
\begin{pf}
	Since $\bold{ w_a}$ and $\bold{w_b}$ are super Lyndon words, we have the following cases:-
	\begin{enumerate}
		\item[(i)] $\bold{w_a}=\bold u \bold u$, $\bold{w_b}=\bold {v_1} \bold {v_2}$
		\item[(ii)] $\bold{w_a}=\bold {u_1} \bold {u_2}$, $\bold{w_b}=\bold {v} \bold {v}$
		\item[(iii)] $\bold{w_a}=\bold u \bold u$, $\bold{w_b}=\bold {v} \bold {v}$
		\item[(iv)] $\bold{w_a}=\bold {u_1} \bold {u_2}$, $\bold{w_b}=\bold {v_1} \bold {v_2}$
		where $ \bold {u_1} \neq \bold{u_2}, \, \bold {v_1} \neq \bold{v_2}$
	\end{enumerate}
	Case(i):- Since $\bold{w_a}=\bold u \bold u$, $\bold{w_b}=\bold {v_1} \bold {v_2}
	\Rightarrow \bold{w_a}\bold{w_b}= \bold u|\bold u \bold{v_1}\bold{v_2}$
	\begin{align*}
	[L(\bold{w_a}), L(\bold{w_b})] &= [[L(\bold u),L(\bold u)], L(\bold{w_b})]\\
	&=2 [L(\bold u),[L(\bold u), L(\bold{w_b})]]\\
	&= 2[L(\bold u), L(\bold u \bold{w_b})]\\
	&= L(\bold u \bold u \bold{w_b})
	\end{align*}
	Case(ii):-$\bold{w_a}=\bold {u_1} \bold {u_2}$, $\bold{w_b}=\bold {v} \bold {v}$.
	If $\bold {u_2} < \bold {v}$ then $\bold{w_a}\bold{w_b}= \bold {u_1}|\bold {u_2} \bold{v}\bold{v}$ is the standard factorization.
	\begin{align*}
	[L(\bold{w_a}), L(\bold{w_b})] &= [[L(\bold {u_1}),L(\bold {u_2})], L(\bold{w_b})]\\
	&= [[L(\bold {u_1}), L(\bold{w_b})], L(\bold {u_2})]+ [L(\bold {u_1}),[L(\bold {u_2}), L(\bold{w_b})]]\\
	&=[L(\bold {u_1}\bold{w_b}),L(\bold {u_2})]+ [L(\bold {u_1}),L(\bold {u_2}\bold{w_b})])]
	\end{align*}
	Otherwise  $\bold{w_a}\bold{w_b}= \bold {u_1}\bold {u_2}| \bold{v}\bold{v}$ is the standard factorization.
	$$\Rightarrow [L(\bold{w_a}), L(\bold{w_b})]=L(\bold{w_a}\bold{w_b})$$
	Case(iii):-$\bold{w_a}=\bold u \bold u$, $\bold{w_b}=\bold {v} \bold {v}$. Since
	$\bold u <\bold v,  \bold{w_a}\bold{w_b}= \bold u|\bold u \bold v \bold v$ is the standard factorization.
	\begin{align*}
	[L(\bold{w_a}), L(\bold{w_b})] &= [[L(\bold {u}),L(\bold {u})], L(\bold{w_b})]\\
	&=[L(\bold {u}),L(\bold {u}\bold{w_b})]
	\end{align*}
	Case(iv):- $\bold{w_a}=\bold {u_1} \bold {u_2}$, $\bold{w_b}=\bold {v_1} \bold {v_2}$.
	This case follows from Lemma \ref{s2}.
%	Hence proved.
	\end{pf}

\begin{lem}\label{spanlem}
	The root space $\lie g_{\eta(\bold k)}$ is contained in the span $ \{e(L(C^i(\bold k, G))\}$.
	
\end{lem}
\begin{pf} Let $e(\bold w) \in \lie g_{\eta(\bold k)}$ for some $\bold w \in M_{\bold k}(I,\eta)$. By Lemma \ref{s3}, we can assume that $\rm{IA}(\bold w)=\{i\} $ . We will do the proof by induction on $\htt(\eta(\bold k))$. If $\htt(\eta(\bold k))=1 $ then $\bold w=i $ and nothing to prove. Assume that the result is true for any $\bold {\tilde{w}}$ such that $\htt(\rm{wt}(\bold {\tilde{w}})) < ht(\eta(k))$. Let $\bold w =ia_1a_2\cdots a_r = i \cdot \bold u$. 
	
	If $i(\bold u) = \phi $ then $ \bold w \in \mathcal X_i \Rightarrow L(\bold w)=\bold w \Rightarrow e(\bold w) = e(L(\bold w)) \in \rm{span}\{ e(L(C^i(\bold k, G))\}$.

	If $i(\bold u) \neq \phi $ then let $min\{i(\bold u)\}= p+1\, $ (say). Set $\bold w'= ia_1a_2\cdots a_pi$. Now,
	\begin{align*}
	e(\bold w') &= [[[[i,a_1],a_2]\cdots ,a_p],i]\\
	&=- [i,[[[i,a_1],a_2]\cdots ,a_p]]\\
	&= e(L(iia_1\cdots a_p)) 
	\end{align*}
	$$\Rightarrow e(\bold w') \in \rm{span}\{e(L(C^i(wt(\bold w'), G))\} \text{ as } (iia_1\cdots a_p) \text{ is a super Lyndon word}.$$
	Now, 
	\begin{align*}
	e(\bold w' \cdot a_{p+2})&= [[[[[i,a_1],a_2]\cdots ,a_p],i],a_{p+2}]\\
	&=[e(\bold w'),a_{p+2}]\\
	&=[[e(i),e(ia_1a_2\cdots a_p)],a_{p+2}]\\
	&=[e(i),[e(ia_1a_2\cdots a_p),a_{p+2}]]+ [[e(i),a_{p+2}],e(ia_1a_2\cdots a_p)]\\
	&= [e(i),e(ia_1a_2\cdots a_pa_{p+2})]+ [e(ia_{p+2}),e(ia_1a_2\cdots a_p)]
	\end{align*}
	%$\Rightarrow e(\bold w'.a_{p+2}) \in \rm{span}\{ e(L(C^i(\wt(\bold {\tilde{w}}), G))\}.$
	Similarly, again using the Jacobi identity, we can write \begin{align*}
	e(\bold w) = &[e(i),e(ia_1a_2\cdots \hat{a_{p+1}} \cdots a_r)]+ \sum\limits_{t= p+2}^{k} [e(ia_1a_2\cdots \hat{a_{p+1}}\cdots \hat{a_t} \cdots a_r), e(ia_t)]+\\
	& + \sum\limits_{ p+2 \leq t_1 < t_2 \leq r}^{k} [e(ia_1a_2\cdots \hat{a_{p+1}}\cdots \hat{a_{t_1}}\cdots \hat{a_{t_2}} \cdots a_r), e(ia_{t_1}a_{t_2})]\\
	& + \sum\limits_{ p+2 \leq t_1 < t_2 <t_3 \leq r}^{k} [e(ia_1a_2\cdots \hat{a_{p+1}}\cdots \hat{a_{t_1}}\cdots \hat{a_{t_2}}\cdots \hat{a_{t_3}} \cdots a_r), e(ia_{t_1}a_{t_2}a_{t_3})]+ \cdots \\
	& + [e(ia_1\cdots a_{p}),e(ia_{p+2}\cdots a_r)]
	\end{align*}
	Using the induction hypothesis, we can check that each term on the right-hand side is of the form 
	$$[e(ia_1a_2\cdots \hat{a_p}a_{p+1}\cdots \hat{a_{t_1}}\cdots \hat{a_{t_2}} \cdots \hat{a_{t_j}}\cdots a_r), e(ia_{t_1}a_{t_2}\cdots a_{t_j})] =\left[ \sum\limits_{a} e(L(\bold{w_a})), \sum\limits_{b} e(L(\bold{w_b}))\right] $$
	as $\left( \rm{wt}(ia_1a_2\cdots \hat{a_p}a_{p+1}\cdots \hat{a_{t_1}}\cdots \hat{a_{t_2}} \cdots \hat{a_{t_j}}\cdots a_r) < \rm{wt}(\bold w) \right)$. So  
	
	$e(ia_1a_2\cdots \hat{a_p}a_{p+1}\cdots \hat{a_{t_1}}\cdots \hat{a_{t_2}} \cdots \hat{a_{t_j}}\cdots a_r)= \sum\limits_{a} e(L(\bold{w_a}))$ where $\bold{w_a}$ is super Lyndon word, $\rm{wt}(\bold{w_a})= \rm{wt}(ia_1a_2\cdots \hat{a_p}a_{p+1}\cdots \hat{a_{t_1}}\cdots \hat{a_{t_2}} \cdots \hat{a_{t_j}}\cdots a_r)$ and $\rm{wt}(ia_{t_1}a_{t_2}\cdots a_{t_j}) < \rm{wt}(\bold{w_b}).$ So $e(ia_{t_1}a_{t_2}\cdots a_{t_j})=\sum\limits_{b} e(L(\bold{w_b}))$ where $\bold{w_b}$ is super Lyndon word and $\rm{wt}(\bold {w_b})= \rm{wt}(ia_{t_1}a_{t_2}\cdots a_{t_j})$. 
	\begin{align*}
	\Rightarrow  \left[ \sum\limits_{a} e(L(\bold{w_a})), \sum\limits_{b} e(L(\bold{w_b}))\right] &= \sum\limits_{a,b}[e(L(\bold{w_a})),e(L(\bold{w_b}))]\\
	&=\sum\limits_{a,b} e\left( [L(\bold{w_a}),L(\bold{w_b})]\right)
	\end{align*}
	By  Lemma \ref{s4}, we can write $[L(\bold{w_a}),L(\bold{w_b})] \in  \rm{span} \{L(C^i(\bold k, G))\}$.\\ Thus $$ \sum\limits_{a,b} e\left( [L(\bold{w_a}),L(\bold{w_b})]\right) \in \rm{span}\, e(L(C^i(\bold k, G)).$$
	$$\Rightarrow e(\bold w) \in \rm{span} \{e(L(C^i(\bold k, G))\}.$$
	Hence $\lie g_{\eta(\bold k)}$ is contained in the span of $\{e(L(C^i(\bold k, G))\} \subseteq \lie g^{i}$ .
	
\end{pf}

\subsection{Proof of Lemma \ref{helplem2} (Identification of $C^{i}(\bold k,G)$ and super Lyndon heaps)}

Let $\lie g$ be the BKM superalgebra associated with the Borcherds-Kac-Moody supermatrix $(A,\Psi)$. Let $G$ be the associated quasi-Dynkin diagram of $\lie g$ with the vertex set $I$. Fix $\bold k \in \mathbb Z_+[I]$ such that $k_i \le 1$ for $i \in I^{re}\sqcup \Psi_0$. 

%We have super Lyndon heaps indexes as a basis for free partially commutative Lie superalgebras. 

Fix $i \in I$ and assume that $i$ is the minimum element in the total order of $I$. Consider $$\mathcal X_i = \{\bold w \in M(I,G,\Psi) : \rm{IA}_m(\bold w) = \{i\} \text{ and $i$ occurs only once in $\bold w$} \}.$$
Now, Let $\bold w \in \mathcal X_i$ and $ E = \psi (\bold w)$ be the corresponding heap. Then
\begin{enumerate}
	\item $\rm{IA}_m(\bold w) = \{i\}$ implies that $E$ is a pyramid.
	\item $i$ occurs exactly once in $\bold w$ implies that $E$ is elementary
	\item $i$ is the minimum element in the total order of $I$ implies that $ E$ is an admissible pyramid. 
\end{enumerate}
\begin{equation}\label{realization}
\text{ Therefore } \bold w \in \mathcal X_i \text{ if and only if }  E = \psi(\bold w) \text{ is a super-letter.}
\end{equation}
Let $\mathcal A_i(I,\zeta)$ be the set of all super-letters with basis $i$ in $\mathcal H(I,\zeta)$. Let $\mathcal A_i^{*}(I,\zeta)$ be the monoid generated by $\mathcal A_i(I,\zeta)$ in $\mathcal H(I,\zeta)$. Then $\mathcal A_i^{*}(I,\zeta) = \mathcal A_{i,0}^{*}(I,\zeta) \oplus \mathcal A_{i,1}^{*}(I,\zeta)$ is also $\mathbb Z_2$-graded. This monoid is free by the discussion below \cite[Definition 2.1.4]{la95}, and by \cite[Proposition 1.3.5 and Proposition 2.1.5]{la95}.  We have $\mathcal H(I,\zeta)$ is totally ordered and hence $\mathcal A_i(I,\zeta)$ is totally ordered. This implies that $\mathcal A_i^*(I,\zeta)$ is totally ordered by the lexicographic order induced from the order in $\mathcal A_i(I,\zeta)$ (call it $\le^{\ast}$).
The following proposition from \cite[Proposition 2.1.6]{la95} illustrates the relation between the total order $\le$ on the heaps monoid $\mathcal H(I,\zeta)$ and the total order $\le^{\ast}$ on the monoid $\mathcal A_i^{*}(I,\zeta)$.
\begin{prop}\label{order}
	Let $ E, F \in \mathcal A_i^{\ast}(I,\zeta)$. Then $ E \le^{\ast} F$ if, and only if, $E \le F$.
\end{prop} Given this, we can talk about the Lyndon words over the alphabets $\mathcal A_i(I,\zeta)$. The following proposition from \cite[Proposition 2.1.7]{la95} illustrates the relationship between the Lyndon words in $\mathcal A_i^*(I,\zeta)$ and the Lyndon heaps in $\mathcal H(I,\zeta)$.
%\begin{equation}\label{realizations}
%\bold w \in \mathcal X_i \text{ if and only if } \bold w \text{ is a super-letter.}
%\end{equation}
%Let $\mathcal A(I,\zeta)$ be the set of all super-letters in $\mathcal H(I,\zeta)$. Let $\mathcal A_i^{*}(I,\zeta)$ be the monoid generated by $\mathcal A_i$ in $H(I,\zeta)$. This monoid is free by the discussion below Definition 2.1.4 and by Proposition 1.3.5 and Proposition 2.1.5. We have $H(I,\zeta)$ is totally ordered and hence $\mathcal A_i(I,\zeta)$ is totally ordered (Section 1.4). This implies that $\mathcal A_i^*(I,\zeta)$ is totally ordered by the lexicographic order induced from the order in $\mathcal A_i(I,\zeta)$. Therefore, we can talk about the super Lyndon words in $\mathcal A_i^*(I,\zeta)$. Therefore, we can talk about the Lyndon words in $\mathcal A_i^*(I,\zeta)$. The following importa%nt proposition illustrates the relation between the Lyndon words in $\mathcal A_i^*(I,\zeta)$ and the Lyndon heaps in $\mathcal H(I,\zeta)$.

%Proposition 2.1.7 :
%The following proposition is the generalization of Proposition \ref{liffl} for the case of super Lyndon words and super Lyndon heaps.
\begin{prop}\label{liffl}
	Let $E \in \mathcal A_i^*(I,\zeta)$ then $E$ is a Lyndon word in $\mathcal A_i^*(I,\zeta)$ if and only if $E$ is a Lyndon heap as an element of $\mathcal H(I,\zeta)$.
\end{prop}

	Next, we prove the following generalization of Proposition \ref{liffl} for the case of super Lyndon words and super Lyndon heaps.
	\begin{prop}\label{sliffsl}
		Let $ E \in \mathcal A_i^*(I,\zeta)$ then $ E$ is a super Lyndon word in $\mathcal A_i^*(I,\zeta)$ if and only if $E$ is a super Lyndon heap as an element of $\mathcal H(I,\zeta)$.
	\end{prop}
	\begin{pf}
		Let $ E \in \mathcal A_i^*(I,\zeta)$ be a super Lyndon word. Then either $E$ is a Lyndon word in $\mathcal A_i^*(I,\zeta)$ or $ E = F \circ F$ for some Lyndon word $F \in \mathcal A_{i,1}^*(I,\zeta)$. Suppose the former case holds, then by Proposition \ref{liffl}, $E$ is a Lyndon heap and hence is a super Lyndon heap in $\mathcal H(I,\zeta)$.
		Suppose the latter case holds, then again by Proposition \ref{liffl}, $F$ is a Lyndon heap in $\mathcal H(I,\zeta)$ and hence $E = F \circ F$ is a super Lyndon heap in $\mathcal H(I,\zeta)$. 
		
		Conversely, suppose $E$ is a super Lyndon heap in $\mathcal H(I,\zeta)$. Then $ E$ is a Lyndon heap in $\mathcal H(I,\zeta)$ or $E = F \circ F$ for some Lyndon heap $ F \in \mathcal H(I,\zeta)$. Suppose the former case holds, then by Proposition \ref{liffl}, $E$ is a Lyndon word and hence is a super Lyndon word in $\mathcal A_i^*(I,\zeta)$.
		Suppose the latter case holds, then again by Proposition \ref{liffl}, $F$ is a Lyndon word in $\mathcal A_{i,1}^*(I,\zeta)$ and hence $E = F \circ F$ is a super Lyndon word in $\mathcal A_i^*(I,\zeta)$. 
	\end{pf}
	By Equation \eqref{realization}, we can identify $\mathcal X_i^*$ with $\mathcal A_i^*(I,\zeta)$.  This implies that %$$C^i(\bold k,G) = \{\text{ Lyndon words in }\mathcal X_i^*\} = \{\text{ Lyndon words of weight $\bold k$ in }\mathcal A_i^*(I,\zeta)\}$$
	\begin{align*}
	|C^i(\bold k,G)| &= |\{\text{ super Lyndon words in }\mathcal X_i^* \text{ of weight } \bold k\}| \\ &= |\{\text{ super Lyndon words of weight $\bold k$ in }\mathcal A_i^*(I,\zeta)\}| \\ &= |\{\text{ super Lyndon heaps of weight $\bold k$ in }\mathcal H(I,\zeta)\}| \\ &= \dim \mathcal {LS}_{\bold k}(G) \\ & = \dim \lie g_{\eta (\bold k)}  \,\,\,(\text{By Theorem \ref{lafps}})
	.\end{align*}
	This shows that the elements of $C^{i}(\bold k,G)$ are precisely the Lyndon heaps of weight $\bold k$ and  completes the proof of Lemma \ref{helplem2}.

\begin{rem}
	The important step in the proof of Theorem \ref{mainthmbb} given in \cite{akv17}   is the proof of the equality $\mult \eta(\bold k) = |C^{i}(\bold k,G)|$ (\cite[Proposition 4.5 (iii)]{akv17}). The key idea is to reduce the proof to $\bold k = \bold 1 = (1,1,1,\dots)$-case and then to use a result of Greene and Zaslavsky \cite{GZ83} for acyclic orientations of $G$ to complete the proof. Also, Theorem \ref{mainthmchb} is used in the proof (in other words the denominator identity of $\lie g$ is used in the proof). Above, we have given a simpler and  direct proof of this equality by identifying the set $C^{i}(\bold k,G)$ with the set of Lyndon heaps over graph $G$. Our proof doesn't use Theorem \ref{mainthmchb} and hence independent of the denominator identity of $\lie g$. Similarly, we don't need the result of Greene and Zaslavsky. Also, we have extended this equality to the super case using super Lyndon heaps. 
\end{rem}

\begin{example}\label{LLLN}
The following diagram explains the connection between the Lyndon basis and the LLN basis of the free root spaces of a BKM superalgebra $\lie g$.

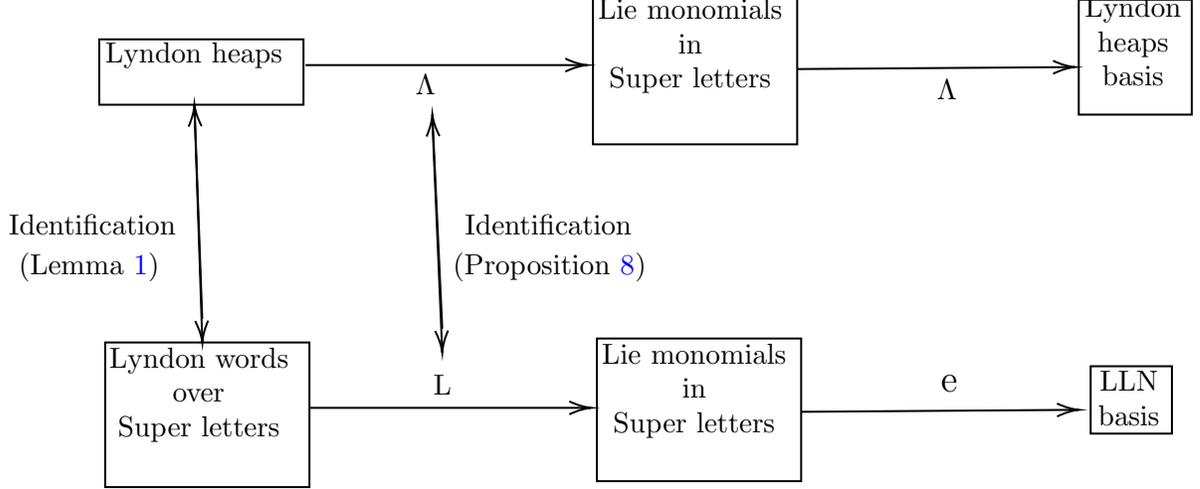
\tikzset{every picture/.style={line width=0.75pt}} %set default line width to 0.75pt        

\begin{tikzpicture}[x=0.75pt,y=0.75pt,yscale=-1,xscale=1]
%uncomment if require: \path (0,480); %set diagram left start at 0, and has height of 480

%Shape: Rectangle [id:dp5609262526893501] 
\draw   (42,42) -- (145,42) -- (145,75) -- (42,75) -- cycle ;
%Straight Lines [id:da18812051163056887] 
\draw    (146,55) -- (286,55) ;
\draw [shift={(288,55)}, rotate = 180] [color={rgb, 255:red, 0; green, 0; blue, 0 }  ][line width=0.75]    (10.93,-3.29) .. controls (6.95,-1.4) and (3.31,-0.3) .. (0,0) .. controls (3.31,0.3) and (6.95,1.4) .. (10.93,3.29)   ;
%Shape: Rectangle [id:dp47862887414343835] 
\draw   (291,20) -- (394,20) -- (394,95) -- (291,95) -- cycle ;
%Straight Lines [id:da375750567821727] 
\draw    (394,57) -- (532,56.01) ;
\draw [shift={(534,56)}, rotate = 539.5899999999999] [color={rgb, 255:red, 0; green, 0; blue, 0 }  ][line width=0.75]    (10.93,-3.29) .. controls (6.95,-1.4) and (3.31,-0.3) .. (0,0) .. controls (3.31,0.3) and (6.95,1.4) .. (10.93,3.29)   ;
%Shape: Rectangle [id:dp465785374795717] 
\draw   (536,22) -- (593,22) -- (593,80) -- (536,80) -- cycle ;
%Shape: Rectangle [id:dp11603658689461693] 
\draw   (45,195) -- (148,195) -- (148,268) -- (45,268) -- cycle ;
%Straight Lines [id:da6563443330780272] 
\draw    (148,228) -- (288,228) ;
\draw [shift={(290,228)}, rotate = 180] [color={rgb, 255:red, 0; green, 0; blue, 0 }  ][line width=0.75]    (10.93,-3.29) .. controls (6.95,-1.4) and (3.31,-0.3) .. (0,0) .. controls (3.31,0.3) and (6.95,1.4) .. (10.93,3.29)   ;
%Shape: Rectangle [id:dp35627090117207705] 
\draw   (293,193) -- (396,193) -- (396,265) -- (293,265) -- cycle ;
%Straight Lines [id:da9079861767878494] 
\draw    (396,230) -- (534,229.01) ;
\draw [shift={(536,229)}, rotate = 539.5899999999999] [color={rgb, 255:red, 0; green, 0; blue, 0 }  ][line width=0.75]    (10.93,-3.29) .. controls (6.95,-1.4) and (3.31,-0.3) .. (0,0) .. controls (3.31,0.3) and (6.95,1.4) .. (10.93,3.29)   ;
%Shape: Rectangle [id:dp7448937957894686] 
\draw   (542,207) -- (583,207) -- (583,241) -- (542,241) -- cycle ;
%Straight Lines [id:da2923828914863156] 
\draw    (90,75) -- (93.94,183) ;
\draw [shift={(94,195)}, rotate = 268.2] [color={rgb, 255:red, 0; green, 0; blue, 0 }  ][line width=0.75]    (10.93,-3.29) .. controls (6.95,-1.4) and (3.31,-0.3) .. (0,0) .. controls (3.31,0.3) and (6.95,1.4) .. (10.93,3.29)   ;
%Straight Lines [id:da12139495398318689] 
\draw    (94,195) -- (90.06,78) ;
\draw [shift={(90,75)}, rotate = 448.2] [color={rgb, 255:red, 0; green, 0; blue, 0 }  ][line width=0.75]    (10.93,-3.29) .. controls (6.95,-1.4) and (3.31,-0.3) .. (0,0) .. controls (3.31,0.3) and (6.95,1.4) .. (10.93,3.29)   ;

\draw    (210,80) -- (215,188) ;
\draw [shift={(215,200)}, rotate = 268.2] [color={rgb, 255:red, 0; green, 0; blue, 0 }  ][line width=0.75]    (10.93,-3.29) .. controls (6.95,-1.4) and (3.31,-0.3) .. (0,0) .. controls (3.31,0.3) and (6.95,1.4) .. (10.93,3.29)   ;
%Straight Lines [id:da12139495398318689] 
\draw    (215,200) -- (210,83) ;
\draw [shift={(210,80)}, rotate = 448.2] [color={rgb, 255:red, 0; green, 0; blue, 0 }  ][line width=0.75]    (10.93,-3.29) .. controls (6.95,-1.4) and (3.31,-0.3) .. (0,0) .. controls (3.31,0.3) and (6.95,1.4) .. (10.93,3.29)   ;

% Text Node
\draw (42,42) node [anchor=north west][inner sep=0.75pt]   [align=left] {\begin{minipage}[lt]{69.26412pt}\setlength\topsep{0pt}
	\begin{center}
	Lyndon heaps 
	\end{center}
	
	\end{minipage}};
% Text Node
\draw (291,20) node [anchor=north west][inner sep=0.75pt]   [align=left] {\begin{minipage}[lt]{71.32588000000001pt}\setlength\topsep{0pt}
	\begin{center}
	Lie monomials \\in \\Super letters
	\end{center}
	
	\end{minipage}};
% Text Node
\draw (536,19) node [anchor=north west][inner sep=0.75pt]   [align=left] {\begin{minipage}[lt]{38.643108000000005pt}\setlength\topsep{0pt}
	\begin{center}
	Lyndon \\heaps\\basis\\
	\end{center}
	
	\end{minipage}};
% Text Node
\draw (45,196) node [anchor=north west][inner sep=0.75pt]   [align=left] {\begin{minipage}[lt]{68.69088pt}\setlength\topsep{0pt}
	\begin{center}
	Lyndon words \\over \\Super letters
	\end{center}
	
	\end{minipage}};
% Text Node
\draw (293,194) node [anchor=north west][inner sep=0.75pt]   [align=left] {\begin{minipage}[lt]{71.32588000000001pt}\setlength\topsep{0pt}
	\begin{center}
	Lie monomials \\in \\Super letters
	\end{center}
	
	\end{minipage}};
% Text Node
\draw (542,208) node [anchor=north west][inner sep=0.75pt]   [align=left] {\begin{minipage}[lt]{26.541216000000002pt}\setlength\topsep{0pt}
	\begin{center}
	LLN \\basis\\
	\end{center}
	
	\end{minipage}};
% Text Node
\draw (-5,129) node [anchor=north west][inner sep=0.75pt]   [align=left] {Identification};
\draw (225,129) node [anchor=north west][inner sep=0.75pt]   [align=left] {Identification};
% Text Node
\draw (0,148) node [anchor=north west][inner sep=0.75pt]   [align=left] {(Lemma \ref{identification lem})};
\draw (219,148) node [anchor=north west][inner sep=0.75pt]   [align=left] {(Proposition \ref{b=b})};
% Text Node
\draw (209,210) node [anchor=north west][inner sep=0.75pt]   [align=left] {L};
% Text Node
\draw (463,61) node [anchor=north west][inner sep=0.75pt]   [align=left] {$\displaystyle \Lambda $};
% Text Node
\draw (200,58) node [anchor=north west][inner sep=0.75pt]   [align=left] {$\displaystyle \Lambda$};
% Text Node
\draw (465,210) node [anchor=north west][inner sep=0.75pt]   [align=left] {{\Large e}};

\end{tikzpicture}

To observe the above diagram through an example, consider the root space $\lie g_{\eta(\bold k)}$ where $\eta(\bold k)= 2\alpha_3+\alpha_4+2 \alpha_5+\alpha_6 $ from Example \ref{basis1_ex2}.   %we use Examples \ref{basis1_ex2} and \ref{basis2_ex2}.We have $I=\{ 1,2,3,4,5,6\},  \, \Psi=\{3,5\}, , I_0=\{1,2,4,6\}, I^{re} =\{1,4\}, \eta(\bold k)= 2\alpha_3+\alpha_4+2 \alpha_5+\alpha_6 $. 
Fix $i=3$,
set of super-letters with basis 3 is ${\mathcal{A}_3(I,\zeta)}=\{3,34,345,3456,344,34545,\cdots\}$.

Lyndon heaps on  ${\mathcal{A}_3^\ast(I,\zeta)}$ of weight $\eta(\bold k)  = \{334556,334565\} $
\begin{align*} \text{ Lie monomials in super-letters }
&=\{\Lambda(334556), \Lambda(334565)\}\\
&=\{ [3,34556], [3,34565]\}\\
\end{align*} 
\begin{align*} \text{ Lyndon heaps basis }
&=\{\Lambda([3,34565]), \Lambda([3,34556])\}\\
&=\{[\Lambda(3), \Lambda(34565)], [\Lambda(3),\Lambda(34556)]\}\\
&=\{[3,[3,[4,[[5,6],5]]]], [3,[3,[4,[5,[5,6]]]]]\}\\
\end{align*}

Lyndon words on super-letters of weight $\eta(\bold k) =  \{334556,334565\}$
\begin{align*} \text{ Lie monomials in super-letters }
&=\{L(334556),L(334565)\}\\
&=\{ [3,34556], [3,34565]\}\\
\end{align*}
\begin{align*} \text{ LLN basis }
&=\{e([3,34556]), e([3,34565])\}\\
&=\{[e(3), e(34556)], [e(3),e(34565)]\}\\
&=\{[3,[[[[3,4],5],5],6]], [3,[[[[3,4],5],6],5]]\}\\
\end{align*}

\end{example}

\subsection{Comparison of the Lyndon basis with the LLN basis}\label{same} In this section, we will give necessary and sufficient condition in which certain elements in both the bases are equal. %Since the basis elements consists of super-letters, so the following proposition gives the necessary and sufficient condition. 
We start with the following proposition whose proof is immediate from Propositions \ref{order} and \ref{liffl}. 

 \begin{prop}\label{b=b}
	Let $\bold w \in \mathcal X_i^*$ be a Lyndon word with standard factorization $\sigma(\bold w) = (\bold w_1,\bold  w_2)$. Then $\bold w,\bold w_1$ and $\bold w_2$ can be thought of as Lyndon heaps [c.f. Proposition \ref{sliffsl}] and the Standard factorization of the Lyndon heap $\bold w$ is $\Sigma(\bold w) = (\bold w_1, \bold w_2)$. Further $L(\bold w) = \Lambda(\bold w)$.
\end{prop}
The following lemma will be helpful.
\begin{lem}\label{eq_stfact}\cite[Lemma 2.3.5]{la95} Let $E \in \mathcal{H}(I,P)$ with $|E| \geq 2$. Then $\Sigma(E) = (F,N)$ iff $\sigma(St(E)) = (St(F),
St(N)).$
\end{lem}
%\begin{pf}
%	This proposition is a consequence of Proposition \ref{liffl}$: \sigma(w) = (w_1, w_2)$ is the standard factorization of the Lyndon word $w \in \mathcal X_i^{\ast}$ if, and only if, $w_2$ is the minimal Lyndon word right factor of $w$  if, and only if $w_2$ is the minimal Lyndon heap right factor of the Lyndon heap $w$ [c.f. Propositions \ref{liffl} and \ref{order}] if, and only if, $\Sigma(w) = (w_1,w_2)$ is the standard factorization of the Lyndon heap $w$.
%\end{pf}
From Example \ref{LLLN}, to compare elements in the Lyndon basis and LLN basis we have to compare the action of the maps $\Lambda$ and $e$ on the super-letters. The following proposition gives a necessary and sufficient condition for the Lyndon basis element and LLN basis element associated with a super-letter to be equal. 
\begin{prop}
	Let $E$ be a super-letter with the associated standard word $\bold w = a_1a_2\cdots a_r$. Then $\Lambda(E)=e(E)$ if, and only if, $a_1<a_r\leq a_{r-1}\leq \cdots \leq a_2$.
\end{prop}
\begin{proof}
	Let $E$ be a super-letter with the associated standard word $\bold w = a_1a_2\cdots a_r$. Assume that $a_1<a_r\leq a_{r-1}\leq \cdots \leq a_2$. Now, the standard factorization of $\bold w$ is given by $\sigma(\bold w)= (a_1a_2\cdots a_{r-1}, a_r)$. This implies that, by Lemma \ref{eq_stfact}, the standard factorization of $E$ is $\Sigma(E)=( F,G) $ where $ F$ and $G$ are Lyndon heaps satisfying $\st(F)=a_1a_2\cdots a_{r-1}$ and $\st(G)=a_r$. Therefore,
	\begin{align*}
	\Lambda(E) &=[\Lambda(F), \Lambda(G)]\\
	&=[\Lambda(F), a_r] \\
	&=[[\Lambda(F_1), \Lambda(a_{r-1})],a_r]  \text{ since }     \sigma(st(F))=(a_1a_2\cdots a_{r-1},a_{r-2})
		\\
	&= [[\Lambda(F_1),a_{r-1}], a_r] \\
	& \,\,\, \vdots \\
	&=[[[[[a_1,a_2],a_3],\cdots],a_{r-1}],a_r]\\
	&=e(E).
	\end{align*}
	Conversely, assume that $\Lambda(E)=e(E)$ for a super-letter $E$. Let $\bold w = a_1\cdots a_r$ be the standard word of $E$.  We will prove $a_1<a_r\leq a_{r-1}\leq \cdots \leq a_2$ by using induction on $r$.
	
	The base cases $r=1$ and $r=2$ are straight forward.
	
	%For $r=2$, $\bold w= a_1a_2 \in \chi_i$ .i.e. $a_1 < a_2$. 
	Assume that the result is true for $r-1$. Write $\bold w=a_1a_2\cdots a_r= \bold w' \cdot a_{r} $ where $\bold w'=a_1a_2\cdots a_{r-1}$. We have $E =F\circ G $ with $\st(F)=a_1a_2\cdots a_{r-1}$ and $\st(G)=a_r$.  We observe that  $\bold w^{'} \in \chi_i$ and $F$ is a super-letter [c.f Equation \eqref{realization}].
	We claim that $\Lambda(F)=e(F)$, i.e., $\Lambda(F)= [[[[a_1,a_2],a_3],\cdots],a_{r-1}]$. Suppose not, then $\Lambda(F)=[\Lambda(F_1),\Lambda(F_2)]$ such that $\st(F_2) \neq a_{r-1}$ and $$\Lambda(E)= \Lambda(F_1\circ F_2\circ a_r)= \begin{cases}
	[\Lambda(F_1),\Lambda(F_2\circ a_r)] & \text{ if } \rm{IA}(\st(F))<a_r\\
	[\Lambda(F_1\circ F_2),\Lambda(a_r)] & \text{otherwise.}\\
	\end{cases}$$
	There is no other standard factorization of $F_1\circ F_2\circ a_r$ is possible as if $\sum(F_1\circ F_2\circ a_r)=(F_1\circ F_{21},F_{22}\circ a_r) $ for some standard factorization of $F_2=F_{21}\circ F_{22}$ then $F_{22} <F_{21}$ which contradicts that $F_2=F_{21}\circ F_{22}$ is Lyndon heap.

	This shows that $ \Lambda(E) \neq [[[[[a_1,a_2],a_3],\cdots],a_{r-1}],a_r]$, a contradiction.
	Therefore  $\Lambda( F)=e( F)$. Thus by using the induction hypothesis, we have $a_1< a_{r-1}\leq \cdots \leq a_2$. It remains to prove that $a_r \leq a_{r-1}$. Suppose $a_r >a_{r-1}$. We have $a_{r-1}\leq a_{r-2}$ and $\sigma(\bold w)= (a_1a_2\cdots a_{r-2}, a_{r-1}a_r)$. Therefore $\Sigma(E)=(L,K)$   (by Lemma  \ref{eq_stfact}) where $st(L)=a_1a_2\cdots a_{r-2}$ and $st(K)=a_{r-1}a_r$. This implies that 
	
	$
	\Lambda(E) =[\Lambda(L),\Lambda(K)]
	=[\Lambda(L),[a_{r-1},a_r]]
	=[[[[[a_1,a_2],a_3],\cdots],a_{r-2}],[a_{r-1},a_r]]$ which is not equal to
	 $e(E)$, a contradiction to our hypothesis. Hence $a_r \leq a_{r-1}$ and the proof is complete.
\end{proof}
\begin{rem}Example \ref{basis1_ex1} and Example \ref{basis2_ex1} have same basis as super-letters $36$, $366$ and $3666$ satisfy condition given in the above proposition. We remark that all these super-letters satisfying the condition given in the above proposition because there are only two elements in the support of the root. If support has more than two elements then some super-letter among these might satisfy the condition and the others may not. In particular, the Lyndon basis and the LLN basis will share only some common elements.  Example \ref{basis1_ex2} and Example \ref{basis2_ex2} have different basis as super-letters $34565$ and $34556$ do not satisfy the condition given in the above proposition.
\end{rem}

\section{ Combinatorial properties of free roots of BKM superalgebras}\label{multis}
In this section, we explore the further combinatorial properties of free roots of BKM superalgebras. Also, we explain why the proof works in the papers \cite{VV15,akv17}: Chromatic polynomial cannot distinguish multi edges. So we lose multi edges in the Dynkin diagram when we consider the chromatic polynomial of the graph of a BKM superalgebra. In particular, we lose the Cartan integers and consequently Serre relations. The root spaces which are independent of the Serre relations are precisely the free root spaces. 
%dealt in \cite{VV15} and \cite{akv17} in the cases of Kac-Moody algebras and Borcherds algebras respectively.
%Also enough to prove chm poly section using fundamental lemmas.
 
\subsection{Free roots of BKM superalgebras} Let $(G,\Psi)$ be a finite simple supergraph with vertex set $I$, edge set $E$, and the set of odd vertices $\Psi \subseteq I$ [c.f. Definition \ref{sg}].  Let $(A=(a_{ij}),\Psi)$ be the adjacency matrix of $G$. We construct a class of BKM supermatrices from $(A,\Psi)$ as follows: Replace the diagonal zeros of $A$ by arbitrary real numbers. If one such number is positive then replace all the non-zero entries in the corresponding row of $A$ by arbitrary non-positive integers (resp. non-positive even integers) provided $i \notin \Psi$ (resp. $i \in \Psi$). Otherwise, replace the non-zero entries in the associated row of $A$ with arbitrary non-positive real numbers. Let $M_{\Psi}(G)$ be the set of BKM supermatrices associated with the supergraph $(G,\Psi)$ constructed in this way. Let $M(G) = \bigcup\limits_{\Psi \subseteq I}M_{\Psi}(G)$.  Let $\mathcal C(G)$ be the set of all BKM superalgebras whose quasi-Dynkin diagram is $(G,\Psi)$ for some $\Psi \subseteq I$. We observe that the set $C(G)$ consists of BKM superalgebras whose associated BKM supermatrices are in $M(G)$.  
%Note that for each choice of the subsets $I^{re}$ and $I_0$ of $I$ there is a BKM superalgebra $\lie g$ whose set of real simple roots (resp. even simple roots) is given by $I^{re}$ (resp. $I_0$). This is done by constructing a BKM supermatrix with appropriate diagonal entries and suitable degree assignments to the elements of $I$.
In the following proposition, we will prove that all the BKM superalgebras belong to $\mathcal C(G)$  share the same set of free roots and have equal respective multiplicities.

\begin{prop}\label{ml}
	Let $G$ be a graph. Let $\lie g$  be a BKM superalgebra which is an element of $\mathcal C(G)$. Then
	\begin{enumerate}
		\item A $\alpha \in Q_+$ is a free root in $\lie g$ if and only if $\supp \alpha$ is connected in $G$. This gives a one-one correspondence between the connected subgraphs of $G$ and the free roots of any $\lie g \in \mathcal C(G)$. In particular, $\Delta^m(\lie g_1) = \Delta^m(\lie g_2)$ for $\lie g_1, \lie g_2 \in \mathcal C(G)$. 
		\item For any $\lie g \in \mathcal C(G)$, the multiplicity of a free root $\alpha$ depends only on the graph $G$ and this multiplicity is equal to the number of super Lyndon heaps of weight $\bold k =(k_i:i \in I)$ where $\alpha = \sum\limits_{i \in I}k_i\alpha_i$. 
	\end{enumerate}
\end{prop}
\begin{pf}
	The necessity part of (1) is straight forward and we prove the sufficiency part. Assume $\alpha \in Q_+$ is free and $\supp \alpha$ is connected in $G$. We claim that $\alpha$ is a root of $\lie g$. We use induction on height of $\alpha$. The case in which $\htt(\alpha) = 1$ is clear. Suppose $\htt(\alpha) = 2$ then $\alpha = \alpha_i + \alpha_j$ and $a_{ij}<0$. Suppose one of $\alpha_i$ or $\alpha_j$ is real. Without loss of generality we assume $\alpha_i$ is real. Then $S_{\alpha_i}(\alpha_j) = \alpha_j - a_{ij}\alpha_i$ is a root of $\lie g$. 
	%Since $a_{ij} < 0$, $\alpha_j + k \alpha_i$ is a root of $\lie g$ for some $k \in \mathbb N$. 
	This implies that $\alpha_i + \alpha_j$ is a root as the root chain of $\alpha_j$ through $\alpha_i$ contains $\alpha_j + m \alpha_i$ for all $0 \le m \le k$ for some $k \in \mathbb N$. Suppose both $\alpha_i$ and $\alpha_j$ are imaginary then Lemma \ref{rootslem} completes the proof. Assume that the result is true for all connected free $\alpha \in Q_+$ of height  $r-1$. Let $\beta$ be a connected free element of height $r$ in $Q_+$.  Let $\alpha_i \in \supp \beta$ be such that $\supp \beta \backslash \{\alpha_i\}$ is connected in $G$. Since $\supp \beta$ is connected such a vertex exit. Now, $\alpha = \sum_{\substack{\alpha_j \in \supp \beta \\ j \ne i}} \alpha_j$ is connected, free, and has height $r-1$. By the induction hypothesis $\alpha$ is a root in $\lie g$. If $\alpha_i$ is real then $S_{\alpha_i}(\alpha)$ is a root in turn $\beta = \alpha + \alpha_i$ is also root. If $\alpha_i$ is imaginary then, again by Lemma \ref{rootslem}, $\beta$ is a root. This completes the proof of (1). Now, the proof of (2) follows from Lemma \ref{identification lem} and Theorem \ref{lafps}.
	%We observe that the above argument is independent of the absolute value of the non-zero entries occurring in the matrix $A$ of $\lie g$. Hence we conclude that $\Delta^m(\lie g_1) = \Delta^m(\lie g_2)$ for $\lie g_1, \lie g_2 \in \mathcal C(G)$.
		%To prove (2), consider a free roots $\alpha$ of an arbitrary $\lie g \in \mathcal C(G)$. Let $\lie g^{'}$ is the BKM superalgebra  corresponds to the BKM supermatrix $-A$ where $(A,\Psi)$ is the adjacency matrix of the graph $(G,\Psi)$. Then the positive part $\lie n_+$ of $\lie g^{'}$ is isomorphic to the free partially commutative Lie superalgebra $\mathcal {LS}(G,\Psi)$.  Now, Lemma \ref{identification lem} completes the proof.
		%	\textcolor{blue}{proof to be done for super multiliner roots}
\end{pf} 
\iffalse
\begin{cor}
	Let $\lie g$ be a BKM superalgebra with the associated graph $G$. Then any super free root space $\lie g_{\alpha}$ can be identified with the $\bold k$-grade space $\mathcal{LS}_{\bold k}(G)$ where $\bold k = (k_1,\dots,k_n)$ if $\alpha = \sum\limits_{i=1}^nk_i\alpha_i$.
\end{cor}
\begin{pf}
	Let $\lie g^{'}$ be the BKM superalgebra associated with the BKM supermatrix $-A$ where $A$ is the adjacency matrix of $G$. Then the positive part of $\lie g^{'}$ is isomorphic to the free partially commutative Lie superalgebra $\mathcal{LS}(G,\Psi)$. Now, both $\lie g$ and $\lie g^{'}$ have the same quasi-Dynkin diagram $\lie g$. Now, the proof follows from the above proposition.
\end{pf}
\fi
\begin{example}
	Let $l_1 \ge 1, l_2 \ge 2$ and $l_3 \ge 3$ be positive integers satisfying $l_1 = l_2 = l_3 $. Then the complex finite dimensional simple Lie algebras $A_{l_1}, B_{l_2}$ and $C_{l_3}$ have the same quasi-Dynkin diagram the path graph on $l_1$ vertices with $\Psi = \emptyset$. In particular, these algebras have the same set of free roots by the above proposition. In Table \ref{table}, using
	the following Proposition \ref{urmie}, we have listed the BKM superalgebras for which the path graph on $4$ vertices is the quasi-Dynkin diagram along with its free roots. 
	\begin{table}[] \label{table}
		\caption{BKM superalgebras  with equal set of free roots} 
		\begin{tabular}{|l|l|l|}
			\hline
			BKM superalgebras & Simple roots \cite[Section 2.5.4]{kac77}&Free roots   \\ \hline
			$A_4$&$\alpha_1=\varepsilon_1-\varepsilon_2, \alpha_2=\varepsilon_2-\varepsilon_3,$    & $\alpha_1,\alpha_2,\alpha_3,\alpha_4, \alpha_{12},\alpha_{13},$     \\ %%\cline{1-2}
			&$\alpha_3=\varepsilon_3-\varepsilon_4, \alpha_4=\varepsilon_4-\varepsilon_5.$&$\alpha_{14},\alpha_{23},\alpha_{24}, \alpha_{34}.$ \\ \hline
			$B_4$& $\alpha_1=\varepsilon_1-\varepsilon_2, \alpha_2=\varepsilon_2-\varepsilon_3,$   & $\alpha_1,\alpha_2,\alpha_3,\alpha_4, \alpha_{12},\alpha_{13},$  \\ %\cline{1-2}
			&$\alpha_3=\varepsilon_3-\varepsilon_4, \alpha_4=\varepsilon_4$. & $\alpha_{14},\alpha_{23},\alpha_{24}, \alpha_{34}.$ \\ \hline
			%$C_3$&  $\alpha_1=\varepsilon_1-\varepsilon_2, \alpha_2=\varepsilon_2-\varepsilon_3, \alpha_3=2\varepsilon_3$ & $\alpha_1,\alpha_2,\alpha_3,\alpha_{12},\alpha_{23},\alpha_{13}$    \\ \hline
			$C_4$& $\alpha_1=\varepsilon_1-\varepsilon_2, \alpha_2=\varepsilon_2-\varepsilon_3,$ & $\alpha_1,\alpha_2,\alpha_3,\alpha_4, \alpha_{12},\alpha_{13},$  \\ %\cline{1-2}
			&$\alpha_3=\varepsilon_3-\varepsilon_4, \alpha_4=2\varepsilon_4.$&$\alpha_{14},\alpha_{23},\alpha_{24}, \alpha_{34}.$ \\ \hline
			$A(3,0)$& $\alpha_1=\varepsilon_1-\varepsilon_2, \alpha_2=\varepsilon_2-\varepsilon_3,$ & $\alpha_1,\alpha_2,\alpha_3,\alpha_4, \alpha_{12},\alpha_{13},$  \\ %\cline{1-2}
			&$\alpha_3=\varepsilon_3-\varepsilon_4, \alpha_4=\varepsilon_4-\delta_1.$&$\alpha_{14},\alpha_{23},\alpha_{24}, \alpha_{34}.$ \\ \hline
			$A(2,1)$& $\alpha_1=\varepsilon_1-\varepsilon_2, \alpha_2=\varepsilon_2-\varepsilon_3,$ & $\alpha_1,\alpha_2,\alpha_3,\alpha_4, \alpha_{12},\alpha_{13},$  \\ %\cline{1-2}
			&$\alpha_3=\varepsilon_3-\delta_1, \alpha_4=\delta_1-\delta_2.$&$\alpha_{14},\alpha_{23},\alpha_{24}, \alpha_{34}.$ \\ \hline
			$B(0,4)$& $\alpha_1=\delta_1-\delta_2, \alpha_2=\delta_2-\delta_3,$ & $\alpha_1,\alpha_2,\alpha_3,\alpha_4, \alpha_{12},\alpha_{13},$  \\ %\cline{1-2}
			&$\alpha_3=\delta_3-\delta_4, \alpha_4=\delta_4.$&$\alpha_{14},\alpha_{23},\alpha_{24}, \alpha_{34}.$ \\ \hline
			$B(3,1)$& $\alpha_1=\delta_1-\varepsilon_1, \alpha_2=\varepsilon_1-\varepsilon_2,$ & $\alpha_1,\alpha_2,\alpha_3,\alpha_4, \alpha_{12},\alpha_{13},$  \\ %\cline{1-2}
			&$\alpha_3=\varepsilon_2-\varepsilon_3, \alpha_4=\varepsilon_3.$&$\alpha_{14},\alpha_{23},\alpha_{24}, \alpha_{34}.$ \\ \hline
			$C(4)$& $\alpha_1=\varepsilon_1-\delta_1, \alpha_2=\delta_1-\delta_2,$ & $\alpha_1,\alpha_2,\alpha_3,\alpha_4, \alpha_{12},\alpha_{13},$  \\ %\cline{1-2}
			&$\alpha_3=\delta_2-\delta_3, \alpha_4=2\delta_3.$&$\alpha_{14},\alpha_{23},\alpha_{24}, \alpha_{34}.$ \\ \hline
			$F(4)$& $\alpha_1=\frac{1}{2}(\varepsilon_1+\varepsilon_2+\varepsilon_3+\delta), \alpha_2=-\varepsilon_1,$ & $\alpha_1,\alpha_2,\alpha_3,\alpha_4, \alpha_{12},\alpha_{13},$  \\ %\cline{1-2}
			&$\alpha_3=\varepsilon_1-\varepsilon_2, \alpha_4=\varepsilon_2-\varepsilon_3.$&$\alpha_{14},\alpha_{23},\alpha_{24}, \alpha_{34}.$ \\ \hline
			  %& & $\alpha_{14},\alpha_{24},\alpha_{34}$    \\ \hline
		\end{tabular}
	\end{table}

\end{example}

\begin{prop}\cite[Corollary 2.1.23]{UR06}\label{urmie}
A simple finite dimensional Lie superalgebra $\mathfrak{g}$ is a BKM superalgebra if and only if $\mathfrak{g}$ is contragredient of type $A(m,0) = \mathfrak{ sl}(m + 1,1), A(m,1) = \mathfrak{ sl}(m + 1,2), B(0,n) = \mathfrak{osp}(1,2n), B(m, 1) =\mathfrak{osp}(2m + 1,2), C(n)=\mathfrak{osp}(2,2n-2), D(m,1)=\mathfrak{osp}(2m,2), D(2,1,\alpha)$ for $\alpha=0,-1,F(4),$ and $G(3)$.
\end{prop}

\begin{rem}
		We observe that the number of free roots of a BKM superalgebra is equal to the number of connected subgraphs $C(G)$ of $G$. In particular, when $G$ is a tree, this number is equal to the number of subtrees of $G$.  This number is well-studied in the literature. For example, see  \cite{dp19} and the references therein.
\end{rem}

	The rest of this section is dedicated to the proof of Theorem \ref{mainthmch} and Corollary \ref{recursionmult} which relates the $\bold k$-chromatic polynomial with root multiplicities of BKM superalgebras. 
Hence we obtain a Lie superalgebras theoretic interpretation of $\bold k$-chromatic polynomials.
%Theorem \ref{1mainthmch} is our first main theorem.

%\begin{thm}\label{mainthmch}
%	Let $G$ be the quasi Dynkin diagram of a BKM superalgebra $\lie g$. Then
%	$$
%\pi^{G}_{\mathbf{k}}(q)=\sum_{\mathbf{J}\in L_{G}(\bold k)} (-1)^{\mathrm{ht}(\eta(\bold k))+|\mathbf{J}|}\prod_{J\in\bold J_0}\binom{q\text{ mult}(\beta(J))}{D(J,\mathbf{J})}\prod_{J\in\bold J_1}\binom{-q\text{ mult}(\beta(J))}{D(J,\mathbf{J})}.
%	$$
%\end{thm}
%\begin{rem}
%The above theorem is a generalization of \cite[Theorem 1.1]{akv17} where the authors considered the case when $\lie g$ is a Borcherds algebra. 
%\end{rem}
%\section{Proof of Theorem \ref{mainthmch}}\label{pfmainthm}

%\subsection{Proof  of Theorem \ref{mainthmch}}\label{mainsch} Let $\lie g$ be a BKM superalgebra with the associated supergraph $(G,\Psi)$. In this section, we  establish a connection between the multiplicities free roots of $\lie g$ and the $\bold k$-chromatic polynomial of $G$.   
%The main results of this section are Theorem~\ref{mainthmch} and the closed formula stated in Corollary~\ref{recursionmult}.

\subsection{Multicoloring and the $\bold k$-chromatic polynomial of $G$} For any finite set $S$, let $\mathcal{P}(S)$ be the power set of $S$. For a tuple of non--negative integers
$\mathbf{k}=(k_i: i\in I)$, we have $\mathrm{supp}(\bold k)=\{i\in I: k_i\neq 0\}$.

\begin{defn}\label{chmpoly}Let $G$ be a graph with vertex set $I$ and the edge set $E(G)$. Let $\mathbf{k} \in \mathbb Z_+[I]$. 
	We call a map $\tau: I\rightarrow \mathcal{P}\big(\{1,\dots,q\}\big)$ a proper vertex $\bold k$-multicoloring of $G$ if the following conditions are satisfied:
	\begin{enumerate}
		\item[(i)]  For all $i\in I$ we have $|\tau(i)|=k_i$, 
		\item[(ii)] For all $i,j\in I$ such that $(i,j)\in E(G)$ we have $\tau(i)\cap \tau(j)=\emptyset$.
	\end{enumerate}
	%where we understand $|\emptyset|=0$.
\end{defn}

%It means that no two adjacent vertices share the same color. 
The case $k_i=1$ for $i\in I$ corresponds to the classical graph coloring of graph $G$. 
For more details and examples we refer to \cite{HMK04}. 
%\subsection{}\label{chromatic}
The number of ways a graph $G$ can be $\bold k$--multicolored using $q$ colors is a polynomial in $q$, called the generalized $\bold k$-chromatic polynomial ($\bold k$-chromatic polynomial in short) and denoted by $\pi_\mathbf{k}^G(q)$. 
The $\bold k$-chromatic polynomial has the following well--known description. 
We denote by $P_k(\bold k,G)$ the set of all ordered $k$--tuples $(P_1,\dots,P_k)$ such that:
\begin{enumerate}
	\item each $P_i$ is a non--empty independent subset of $I$, i.e. no two vertices have an edge between them; and
	
	%\item[(ii)] the disjoint union of  $P_1,\cdots, P_k$ is equal to the multiset $\{\underbrace{\alpha_i,\dots, \alpha_i}_{k_i\, \mathrm{times}}: i\in I \}$.
	
	\item For all $i \in I$, $\alpha_i$ occurs exactly $k_i$ times in total in the disjoint union $P_1\dot{\cup} \cdots \dot{\cup} P_k$.
\end{enumerate} 
Then we have
\begin{equation}\label{defgenchr}\pi^G_\mathbf{k}(q)= \sum\limits_{k\ge0}|P_k(\bold k, G)| \, {q \choose k}.\end{equation}
\iffalse
The definition of the $\bold k$-chromatic polynomial depends on $q\in \mathbb N$; however it makes sense to consider the evaluation at any integer. 
A famous result of Stanley \cite{S73} says that the evaluation at $q=-1$ of the chromatic polynomial counts (up to a sign) the number of acyclic orientations. 
\fi

We have the following relation between the ordinary chromatic polynomials and the $\bold k$-chromatic polynomials.  We have
\begin{equation} \label{connection}
\pi_\bold k^G(q)=\frac{1}{\bold k!}\pi_{\bold 1}^{G(\bold k)}(q)
\end{equation}
where $\pi_{\bold 1}^{G(\bold k)}(q) \ \text{is the chromatic polynomial of the graph
	$G(\bold k)$}$ and $\bold k!=\prod_{i\in I}k_i!$. 
The graph $G(\bold k)$ (the join of $G$ with respect to $\bold k$) is constructed as follows: For each $j\in \mathrm{supp}(\bold k)$, take a clique (complete graph) of size $k_j$ with vertex set $\{j^1,\dots, j^{k_j}\}$ and join all vertices of the $r$--th and $s$--th cliques if 
$(r,s)\in E(G).$ 

%\subsection{} Let $\bold k=(k_i)_{i \in I} \in \bz_+^{|I|}$. Define $\mathrm{supp}_0(\bold k) = \{i \in I_0 : k_i > 0\}$, $\mathrm{supp}_1^{'}(\bold k) = \{i \in I_1 : k_i > 0 \text{ and } a_{ii} \ne 0\}$ and $\mathrm{supp}_1^{''}(\bold k) = \{i \in I_1 : k_i > 0 \text{ and } a_{ii} = 0\}$. 
% In \cite[Section 3.4]{akv17}, a connection between the root multiplicities of a Borcherds algebra $\lie g$ and the generalized $k-$chromatic polynomials of its graph $G$ is proved under the following assumption on $\bold k \in \mathbb Z_+[I]$:  $k_i\leq 1$ for $i\in I^{\text{re}}$. We want to extend this result to the case of BKM Lie superalgebra. 
\textit{For the rest of this paper, we fix an element $\mathbf{k} \in \mathbb Z_+[I]$ satisfying $k_i\leq 1$ for $i\in I^{\text{re}}\sqcup \Psi_0$,  
	where $\Psi_0$ is the set of odd roots of zero norm.}   %We observe that, when $\lie g$ is a Borcherds algebra, this assumption is same as the above said assumption of  \cite[Section 3.4]{akv17}.

\subsection{Bond lattice and an isomorphism of lattices}
In this subsection, we prove a lemma which will be useful in the proof of Theorem \ref{mainthmch}.
\begin{defn}\label{bond}Let $L_{G}(\bold k)$ be the weighted bond lattice of $G$, which is the set of 
	$\mathbf{J}=\{J_1,\dots,J_k\}$  satisfying the following properties:
	\begin{enumerate}
		\item[(i)] $\bold J$ is a multiset, i.e. we allow $J_i=J_j$ for $i\neq j$
		
		\item[(ii)] each $J_i$ is a multiset and the subgraph spanned by the underlying set of $J_i$ is a connected subgraph of $G$ for each $1\le i\le k$ and
		
		\item[(iii)] For all $i \in I$, $\alpha_i$ occurs exactly $k_i$ times in total in the disjoint union $J_1\dot{\cup} \cdots \dot{\cup} J_k$.
		
	\end{enumerate}
	
	For $\mathbf{J}\in L_{G}(\bold k)$ we denote by $D(J_i,\mathbf{J})$ the multiplicity of $J_i$ in $\mathbf{J}$ and set $\text{mult}(\beta(J_i))=\text{dim }\mathfrak{g}_{\beta(J_i)}$, where $\beta(J_i)=\sum_{\alpha\in J_i} \alpha$. We define $\bold J_0 = \{J_i \in \bold J : \beta(J_i) \in \Delta_+^0\}$ and $\bold J_1 = \bold J \backslash \bold J_0$.
\end{defn}
%We record the following lemmas which will be needed later.
\begin{lem}\label{rootslem}\cite[Proposition 2.40]{WM01}
	Let $i \in I^{im}$ and $\alpha \in \Delta_+ \backslash \{\alpha_i\}$ such that $\alpha(h_i)<0$. Then $\alpha + j \alpha_i \in \Delta_+$ for all $j \in \mathbb Z_+$.
\end{lem}

\begin{lem}\label{bij}\cite[Lemma 3.4]{akv17} Let $\mathcal{P}$ be the collection of multisets $\gamma=\{\beta_1,\dots,\beta_r\}$ (we allow $\beta_i=\beta_j$ for $i\neq j$) such that each $\beta_i\in\Delta_+$ and $\beta_1+\dots+\beta_r=\eta(\bold k)$.
	The map $\psi: L_{G}(\bold k)\rightarrow \mathcal{P}$ defined by $\{J_1,\dots,J_k\}\mapsto \{\beta(J_1),\dots,\beta(J_k)\}$ is a bijection.
	%	\begin{pf}
	%		If $\alpha\in \big(Q_+\cap\sum_{j\in \Pi^{\mathrm{re}}} \mathbb{Z}_{\leq 1}\alpha_j\big)$ is non--zero and the support of $\alpha$ is connected, then $\alpha\in\Delta^{\mathrm{re}}_+$. Moreover, if $\alpha\in \Delta_+$ and $\alpha_i\in \Pi^{\mathrm{im}}$ is such that the support of $\alpha+\alpha_i$ is connected, then by Lemma~\ref{rootslem} we have that $\alpha+\alpha_i\in \Delta_+$. This shows that each $\beta(J_r)$ is a positive root and hence the map is well--defined. The map is obviously injective and since $\alpha\in\Delta_+$ implies that $\alpha$ has connected support, we also obtain that $\Psi$ is surjective.
	%	\end{pf}
\end{lem} 

\subsection{Proof of Theorem \ref{mainthmch}: (Chromatic polynomial and root multiplicities)}\label{pfmainthm}
For a Weyl group element $w\in W$,
we fix a reduced word $w=\bold {s}_{i_1}\cdots \bold{s}_{i_k}$ and let $I(w)=\{\alpha_{i_1},\dots,\alpha_{i_k}\}$. 
Note that $I(w)$ is independent of the choice of the reduced expression of $w$. For $\gamma = \sum\limits_{i \in I}m_i\alpha_i \in \Omega$, we set $I_m(\gamma)$ is the multiset $\{\underbrace{\alpha_i,\dots, \alpha_i}_{m_i\, \mathrm{times}}: i\in I\}$ and $I(\gamma)$ is the underlying set of $I_m(\gamma)$. We define $\Psi_0(\gamma) = I(\gamma) \cap \Psi_0$. Also, we define 
$\mathcal{J}(\gamma)=\{ w\in W\backslash \{e\}: I(w)\cup I(\gamma) \mbox{ is an independent set}\}.$
%Note that $\mathcal{J}(0)$ gives the set of independent subsets of $\Pi^\mathrm{re}$. 
The following lemma is a generalization of \cite[Lemma 2.3]{VV15} (for Kac-Moody Lie algebras) and \cite[Lemma 3.6]{akv17} (for Borcherds algebras) to the setting of BKM superalgebras. Since the proof of this lemma is similar to the proof of the Borcherds algebras case, we omit the proof here. Recall that $\mathbf{k}=(k_i: i\in I)$ satisfies $k_i\leq 1$ for $i\in I^{\text{re}}\sqcup \Psi_0$.

% $\bold k$ satisfies $k_i\leq 1$ for $i\in I^{\text{re}}$  
%and $\overline{\bold k} = (k_i^{'} : i \in I_{(\bold k)})$ satisfies $k_i^{'} = k_i$ for all $i \in \mathrm{supp}_0(\bold k) \cup \mathrm{supp}_1^{'}(\bold k)$, $k_{j^m}^{'} = 1$ for all $j \in \mathrm{supp}_1^{''}(\bold k)$ and $k_i^{'} = 0$ outside the support of $\bold k$. 
\begin{lem}\label{helplem}
	Let $w\in W$ and $\gamma= \sum\limits_{i \in I \backslash \Psi_0}\alpha_i + \sum\limits_{i \in \Psi_0}m_i\alpha_i \in \Omega$. We write $\rho-w(\rho)+w(\gamma)=\sum_{\alpha\in\Pi} b_{\alpha}(w,\gamma)\alpha$. Then we have
	
	\begin{enumerate}
		\item[(i)] $b_{\alpha}(w,\gamma)\in \mathbb{Z}_{+}$ for all $\alpha\in \Pi$ and $b_{\alpha}(w,\gamma)= 0$ if $\alpha\notin I(w)\cup I(\gamma)$.
		\item[(ii)] $b_{\alpha}(w,\gamma)\geq 1 \text{ for all } \alpha \in I(w)$.
		%\item $I(w)=\{\alpha\in \Pi^{\mathrm re} : b_{\alpha}(w,\gamma)\geq 1\}$, 
		\item [(iii)]$b_{\alpha}(w,\gamma)=1$ if $\alpha\in I(\gamma) \backslash \Psi_0(\gamma)$ and $b_{\alpha}(w,\gamma)=m_{\alpha}$ if $\alpha\in \Psi_0(\gamma)$. 
		\item[(iv)] If $w\in \mathcal{J}(\gamma)$, then $b_{\alpha}(w,\gamma)=1$ for all $\alpha\in I(w)\cup (I(\gamma) \backslash \Psi_0(\gamma))$, $b_{\alpha}(w,\gamma)=m_{\alpha}$ for all $\alpha \in \Psi_0(\gamma)$.
		\item[(v)] If $w\notin \mathcal{J}(\gamma)\cup \{e\}$, then there exists $\alpha\in I(w) \subseteq \Pi^{\mathrm re}$ such that $b_{\alpha}(w,\gamma)>1$.
		
	\end{enumerate}
\end{lem}
%\subsection{}
The following proposition is an easy consequence of the above lemma and essential to prove Theorem~\ref{mainthmch}.
Let $U$ be the sum-side	 of the denominator identity (Equation \eqref{denominator}).
\begin{prop}\label{helprop}
	Let $q\in \mathbb{Z}$. We have	
	$$U^q[e^{-\eta(\bold k)}]= (-1)^{\mathrm{ht}(\eta(\bold k))}\  \pi^{G}_{\mathbf{k}}(q),$$
	where $U^q[e^{-\eta(\bold k)}]$ denotes the coefficient of $e^{-\eta(\bold k)}$ in $U^q$.
\end{prop}
	\begin{pf}% If $q=0$, then there is nothing to prove. So assume that $0\neq q\in \mathbb{Z}.$
		Write $U^q=\sum\limits_{k\ge 0}{q \choose k}\, (U-1)^k$.
		%, \ \ \text{where ${q\choose k} = \frac{q(q-1)\cdots (q-(k-1))}{k!}$}.$$ 
		%%and $${-q\choose k} = 
		%(-1)^k{q+k-1\choose k},\ \ \text{for $q\in \mathbb{N}$}.$$
		From Lemma~\ref{helplem} we get 
		$$w(\rho)-\rho-w(\gamma)=-\gamma-\sum_{\alpha\in I(w)}\alpha,\ \ \text{for $w\in \mathcal{J}(\gamma)\cup\{e\}$}.$$
		
		Since $k_i \le 1$ for $i \in I^{re}\sqcup \Psi_0$, the coefficient of $e^{-\eta(\bold k)}$ in $(U-1)^k$ is equal to 
		\begin{equation}\label{coeff}
			\Bigg(\sum_{ \substack{\gamma \in
					\Omega \\ \gamma \ne 0}}\sum_{w \in \mathcal{J}(\gamma)} \epsilon(\gamma)\epsilon(w) e^{-\gamma-\sum_{\alpha\in I(w)}\alpha}\Bigg)^k[e^{-\eta(\bold k)}].
		\end{equation}
		%$$\Bigg(\sum_{w \in \mathcal{J}(0)} (-1)^{\ell(w)}e^{-\sum_{\alpha\in I(w)}\alpha} + \sum_{ \gamma \in
		%	\Omega \backslash \{0\} }(-1)^{\mathrm{ht}(\gamma)}\sum_{w \in \mathcal{J}(\gamma)\cup\{e\} } (-1)^{\ell(w)}  e^{-\gamma-\sum_{\alpha\in I(w)}\alpha}\Bigg)^k[e^{-\eta(\bold k)}].$$
		Hence the coefficient is given by
		\begin{equation}
			\sum_{\substack{(\gamma_1,\dots,\gamma_k)\\(w_1,\dots,w_k)}}\epsilon(\gamma_1)\cdots\epsilon(\gamma_k)\epsilon(w_1)\cdots\epsilon(w_k)
		\end{equation}
		where the sum ranges over all $k$--tuples $(\gamma_1,\dots,\gamma_k)\in \Omega^k$ and $(w_1,\dots,w_k) \in W^k$ such that
		\begin{align*}
		&\bullet \ w_i\in \mathcal{J}(\gamma_i)\cup \{e\},\text{  $1\le i\le k$},&\\&
		\bullet \ I(w_1)\ \dot{\cup} \cdots \dot{\cup}\ I(w_k)=\{\alpha_i: i\in I^{\mathrm{re}}, k_i=1\},&\\&
		\bullet \ I(w_i)\cup I(\gamma_i)\neq \emptyset \ \text{ for each $1\le i\le k$},&\\&
		%\bullet \ i(\gamma) \le 1 \text{ for each } i \in I_0(\gamma),&\\&
		\bullet \ \gamma_1+\cdots+\gamma_k=\sum_{i\in I^{\mathrm{im}}}k_i\alpha_i. 
		\end{align*}
		
		It follows that $\big(I(w_1)\cup I(\gamma_1),\dots,I(w_k)\cup I(\gamma_k)\big)\in P_k\big(\bold k,G\big)$ and each element is obtained in this way. 
		So the sum ranges over all elements in $P_k\big(\bold k,G\big)$.
		%Since $w_1\cdots w_k$ is a sub-word of a Coxeter element we get
		%$$(-1)^{\ell(w_1\cdots w_k)}=(-1)^{|\{i\in I^{\mathrm{re}} : k_i=1\}|},$$
		Hence $(U-1)^k[e^{-\eta(\bold k)}]$ is equal to $(-1)^{\mathrm{ht}(\eta(\bold k))} |P_k(\bold k, G)|$. Now, Equation \eqref{defgenchr} completes the proof.	\end{pf}

\begin{rem}
	We need the extra assumption $k_i \le 1$ for $i \in \Psi_0$ when we extend Theorem \ref{mainthmchb} to the case of BKM superalgebras. Suppose $k_i > 1$ for some $i \in \Psi_0$. We observe that each $\gamma_i$ contributes $I_m(\gamma_i)$ to the required coefficient in Equation \eqref{coeff} and $I_m(\gamma_i)$ can be a multiset. This implies that $I(w_i)\cup I_m(\gamma_i)$ can be a multiset. The independent set considered in Equation \eqref{defgenchr} are sets. By assuming $k_i \le 1$ for $i \in \Psi_0$, we have avoided the possibility of $I_m(\gamma_i)$ being a multiset. 
\end{rem}
%%%%%%%%%%%%%%%%%%%%%%%%%%%%%%%%%%%%%%%%%%%%%%%%%%%%%%%%%%%%%%%%%%%%%%%%%%%%%%%%%%%%%%%%%%%%%%%%%%%%%%%%%%%%%%%%%%%%%%%%%%%%%%%%%%%%%%%%%%%%%%%%%%%%%%%%%%%%%%%%%%%%%%%%
%\subsection{}
Now, we can prove Theorem~\ref{mainthmch} using the product side of the denominator identity \eqref{denominator}. Proposition~\ref{helprop} and Equation \eqref{denominator} together imply that the $\bold k$-chromatic polynomial $\pi^{G}_{\mathbf{k}}(q)$ is given by the coefficient of $e^{-\eta(\bold k)}$ in 
\begin{equation}\label{expa}(-1)^{\mathrm{ht}(\eta(\bold k))}\frac{\prod_{\alpha \in \Delta_+^0} (1 - e^{-\alpha})^{q\mult	(\alpha)}} {\prod_{\alpha \in \Delta_+^1} (1 + e^{-\alpha})^{q\mult(\alpha)}} = (-1)^{\mathrm{ht}(\eta(\bold k))} \prod_{\alpha \in \Delta_+} (1 - \epsilon(\alpha)e^{-\alpha})^{\epsilon(\alpha)q\mult(\alpha)} .\end{equation} where $\epsilon(\alpha) = 
1 \text{ if }\alpha \in \Delta_+^0 \text{ and } -1 \text{ if }\alpha \in \Delta_+^1$.
Now, $$\prod_{\alpha \in \Delta_+} (1 - \epsilon(\alpha)e^{-\alpha})^{\epsilon(\alpha)q\mult(\alpha)} = \prod_{\alpha \in \Delta_+}\Big(\sum\limits_{k \ge 0}(-\epsilon(\alpha))^k\binom{\epsilon(\alpha) q \mult \alpha}{k}e^{-k\alpha}\Big).$$ A direct calculation of the coefficient of $e^{-\eta(\bold k)}$ in the right-hand side of the above equation completes the proof of Theorem \ref{mainthmch}. %[c.f. Lemma~\ref{bij}].  %Expanding \eqref{expa} and using Lemma~\ref{bij} finishes the proof.

%%%%%%%%%%%%%%%%%%%%%%%%%%%%%%%%%%%%%%%%%%%%%%%%%%%%%%%%%%%%%%%%%%%%%%%%%%%%%%%%%%%%%%%%%%%%%%%%%%%%%%%%%%%%%%%%%%%%%%%%%%%%%%%%%%%%%%%%%%%%%%%%%%%%%%%%%%%%%%%%%%%%%%%%

%\subsection{}In this subsection, 

\subsection{Proof of Corollary \ref{recursionmult}: (Formula for multiplicities of free roots)}\label{pfmultf}
In this subsection, we prove Corollary \ref{recursionmult} which gives a combinatorial formula for the multiplicities of free roots. We consider the algebra of formal power series $\mathcal{A}:=\mathbb{C}[[X_i : i\in I]]$. For a formal power series $\zeta\in\mathcal{A}$ with constant term 1, its logarithm $\text{log}(\zeta)=-\sum_{k\geq 1}\frac{(1-\zeta)^k}{k}$ is well--defined.
%\begin{cor}\label{recursionmult}
%We have
%\begin{equation}\label{emult}
%\mult(\eta(\bold k)) = \sum\limits_{\ell | \bold k}\frac{\mu(\ell)}{\ell}\ |\pi^{G}_{\bold k/\ell}(q)[q]|\end{equation} 
%if $\beta(\bold k) \in \Delta_0^+$ and 
%\begin{equation}\label{omult}
%\mult(\eta(\bold k)) = \sum\limits_{\ell | \bold k}\frac{(-1)^l \mu(\ell)}{\ell}\ |\pi^{G}_{\bold k/\ell}(q)[q]|\end{equation}
%if $\beta(\bold k) \in \Delta_1^+$ where $|\pi^{G}_{\bold k}(q)[q]|$ denotes the absolute value of the coefficient of $q$ in $\pi^{G}_{\bold k}(q)$ and $\mu$ is the M\"{o}bius function.

\begin{pf}
	We consider $U$ as an element of $\mathbb{C}[[e^{-\alpha_i} : i\in I]]$ where $X_i = e^{-\alpha_i}$ [c.f. Lemma \ref{helplem}]. 
	From the proof of Proposition~\ref{helprop} we obtain that the coefficient of $e^{-\eta(\bold k)}$ in $-\text{log } U$ is equal to
	$$(-1)^{\mathrm{ht}(\eta(\bold k))}\sum\limits_{k\ge 1}\frac{(-1)^k}{k}|P_k(\bold k,G)|$$ which by Equation \eqref{defgenchr} is equal to $|\pi^{G}_{\bold k}(q)[q]|$.
	Now applying $-\text{log }$ to the right hand side of the denominator identity \eqref{denominator} gives 
	\begin{equation}\label{recursss}\sum_{\substack{\ell\in \mathbb{N}\\ \ell | \bold k}} \frac{1}{\ell}\mult \big(\eta\left(\bold k/\ell\right)\big)=
	|\pi^G_{\bold k}(q)[q]|\end{equation} if $\beta(\bold k) \in \Delta_0^+$ and 
	\begin{equation}\label{recursss}\sum_{\substack{\ell\in \mathbb{N}\\ \ell | \bold k}} \frac{(-1)^{l+1}}{\ell} \mult \big( \eta\left(\bold k/\ell\right)\big)=
	|\pi^G_{\bold k}(q)[q]|\end{equation}  if $\beta(\bold k) \in \Delta_1^+$. 
	
	The statement of the corollary is now an easy consequence of the following M\"{o}bius inversion formula: $g(d) = \sum_{d|n}f(d) \iff f(n) = \sum_{d|n}\mu(\frac{n}{d})g(d)$ where $\mu$ is the m\"{o}bius function.
\end{pf}
%\end{cor}
\begin{example}\label{ch_ex1} Consider the BKM superalgebra $\lie g$ and the root space $\eta(\bold k)= 3\alpha_3+3\alpha_6 \in \Delta_+^1$ from Example \ref{basis1_ex1}. 
\iffalse	
	%associated with the BKM supermatrix $$A=\begin{bmatrix}
	%2 & -1 & 0&0&0&0\\
%	-1 & -3 &-4&-1&0&0\\
%	0&-4&-4& 0&0&-1\\
%	0&-1&0&2&-1&0\\
%	0&0&0&-1&-2&0\\
%	0&0&-1&0&0&-3\\ 
%	\end{bmatrix}.$$ The quasi-Dynkin diagram $G$ of $\lie g$ is as follows: 	\begin{center}
		
	%	\begin{tikzpicture}
	%%	\node (a) at (0,0){$\alpha_1$};
	%	\node (b) at (2,0){$\alpha_2$};
	%	\node (c) at (4,1){$\alpha_3$};
	%	\node (d) at (4,-1){$\alpha_4$};
	%	\node (e) at (6,1){$\alpha_6$};
	%	\node (f) at (6,-1){$\alpha_5$};
	%	\draw (a)--(b)--(c)--(e);
	%	\draw (b)--(d)--(f);
	%	\end{tikzpicture}
	%\end{center}
	
	%We have $I=\{ 1,2,3,4,5,6\}, \Psi=\{3,5\}$ and $I^{re} =\{1,4\}$. Let $\bold k = (0,0,3,0,0,3) \in \mathbb Z_+[I]$. Then $\eta(\bold k)= 3\alpha_3+3\alpha_6 \in \Delta_+^1$. Let's calculate the multiplicity of the root $\eta(\bold k)$ using the multiplicity formula given in Corollary \ref{recursionmult}: 
\fi	
	The $\bold k$-chromatic polynomial of the quasi Dynkin diagram $G$ of $\lie g$ is equal to
	$$\pi^{G}_{\bold k}(q)= \binom{q}{3} \binom{q-3}{3}= \frac{1}{
		3! 3!}q(q-1)(q-2)(q-3)(q-4)(q-5).$$
	%	\begin{align*}
	%	&\pi^{\bold k}_G(q)= \binom{q}{3} \binom{q-3}{3}= \frac{1}{
	%		3! 3!}q(q-1)(q-2)(q-3)(q-4)(q-5)\\
	%	&\pi^{\bold k}_G(q)[q]=\frac{-10}{3}\\
	%	&\Rightarrow |\pi^{\bold k}_G(q)[q]|=\frac{10}{3}
	%	\end{align*}
	By Corollary \ref{recursionmult}, since $\eta(\bold k)$ is odd,
	\begin{align*}
	\mult(\eta(\bold k)) &= \sum\limits_{\ell | \bold k}\frac{(-1)^{l+1} \mu(\ell)}{\ell}\ |\pi^{G}_{\bold k/\ell}(q)[q]|\\
	&=|\pi_{\bold k}^G(q)[q]|+\frac{\mu(3)}{3}|\pi_{\bold k'}^G(q)[q]| \text{ where } \bold k'=(0,0,1,0,0,1)\\
	&= \frac{10}{3}-\frac{1}{3} \\
	&=3
	\end{align*}
\end{example}

\begin{example}\label{multf0} Consider the BKM superalgebra $\lie g$ from the previous example. Let $\bold k = (2,1,0,1,2,0) \in \mathbb Z_+[I]$. Then $\eta(\bold k)= 2\alpha_1+\alpha_2+\alpha_4 + 2 \alpha_5 \in \Delta_+^0$.  We have  \begin{equation*}
	\mult(\eta(\bold k)) = \sum\limits_{\ell | \bold k}\frac{\mu(\ell)}{\ell}\ |\pi^{G}_{\bold k/\ell}(q)[q]|\end{equation*} 
	This implies that
	$\mult(\eta(\bold k)) =  |\pi^{G}_{\bold k}(q)[q]|$. 
	We have $\pi^{G}_{\bold k}(q)= \frac{1}{4}q (q-1)^3(q-2)^2$. Therefore $ \mult(\eta(\bold k))=1$.

\end{example}

\bibliographystyle{plain}
%\bibliography{kv-bib}

\end{document}